\documentclass[a4paper,12pt]{amsart}

\usepackage[leqno]{amsmath}
\usepackage{amssymb,latexsym,amsthm}

\usepackage[small]{diagrams}
\diagramstyle[labelstyle=\scriptstyle,heads=littlevee]
\newarrow{Inline}{hookb}----

\numberwithin{equation}{section}
\theoremstyle{definition}
\newtheorem{defn}[equation]{Definition}
\newtheorem{notatn}[equation]{Notation}
\newtheorem{rmk}[equation]{Remark}
\theoremstyle{plain}
\newtheorem{lemma}[equation]{Lemma}
\newtheorem{prop}[equation]{Proposition}
\newtheorem{thm}[equation]{Theorem}
\newtheorem{cor}[equation]{Corollary}

\newcommand{\fbar}{{\overline{f}}}
\newcommand{\jbar}{{\overline{\jmath}}}
\newcommand{\ubar}{{\overline{u}}}
\newcommand{\etabar}{{\overline{\eta}}}
\newcommand{\phibar}{{\overline{\phi}}}

\newcommand{\Gammatilde}{{\widetilde{\Gamma}}}

\newcommand{\Aut}{{\operatorname{Aut}}}
\newcommand{\Inn}{{\operatorname{Inn}}}
\newcommand{\Out}{{\operatorname{Out}}}
\newcommand{\Ker}{{\operatorname{Ker}}}
\newcommand{\conj}{{\operatorname{\complement}}}
\newcommand{\id}{{\operatorname{id}}}
\newcommand{\res}{{\operatorname{res}}}
\newcommand{\infl}{{\operatorname{infl\,}}}
\newcommand{\tgr}{{\operatorname{tgr}}}
\newcommand{\rd}{{\operatorname{rd}}}

\newcommand{\upperleft}[2]{{\sideset{^{#1}}{}{\mathop{#2}}}}

\def\clap#1{\hbox to 0pt{\hss#1\hss}}
\def\mathllap{\mathpalette\mathllapinternal}
\def\mathrlap{\mathpalette\mathrlapinternal}
\def\mathclap{\mathpalette\mathclapinternal}
\def\mathllapinternal#1#2{%
  \llap{$\mathsurround=0pt#1{#2}$}}
\def\mathrlapinternal#1#2{%
  \rlap{$\mathsurround=0pt#1{#2}$}}
\def\mathclapinternal#1#2{%
  \clap{$\mathsurround=0pt#1{#2}$}}

\begin{document}

\title{Iterated group extensions}

\author{CheeWhye Chin}

\address{Department of Mathematics\\
	 National University of Singapore\\
         10 Lower Kent Ridge Road\\
         Singapore 119076\\
	 Singapore}


\email{cheewhye@math.nus.edu.sg}



\date{August 29, 2010}




%
%

\begin{abstract}
We introduce
the notion of
iterated group extensions,
which, roughly speaking,
is what one obtains
by forming a group extension
of a group extension.
We interpret
iterated extensions
in terms of
group cohomology,
in the same way as
Eilenberg-MacLane did
for usual group extensions.
From the $E_2$-spectral sequence
of a group extension,
there is a 6-term long exact sequence
in which
various cohomology groups
of degree~1 or~2
appear.
We give
an explicit identification of
each cohomology group
and
each morphism
appearing in
this long exact sequence
in terms of
iterated extensions
and associated notions.
These identifications
enable us
to uncover
natural relations
between
(iterated) extensions,
their automorphism groups,
and their outer actions. 
\end{abstract}

\maketitle


\section{Introduction}

A group extension
consists of
an exact sequence of
groups
\[
  \mathllap{
    (KGQ)
    \ :\ 
    \quad
  }
  K
  \rInto^{\qquad i \qquad}
  G
  \rOnto^{\qquad \pi \qquad}
  Q
\]
in which
$i$ is an isomorphism
of $K$ with
a normal subgroup of $G$,
and $\pi$ is a surjective homomorphism
from $G$ onto $R$
with $i(K)$ as kernel.
The conjugation action $\conj_K^G$
of $G$ on $K$
induces an outer action
$\theta : Q \rTo \Out(K)$
of $Q$ on $K$
making the following diagram
commute:
\[
  \begin{diagram}
  K
  &
  \rInto^{i}
  &
  G
  &
  \rOnto^{\pi}
  &
  Q
  \\
  \dOnto
  &
  &
  \dTo^{\conj_K^G}
  &
  &
  \dTo
  \
  \theta
  \\
  \mathllap{
    \conj_K^G(K)
    \ =\
  }
  \Inn(K)
  &
  \rInto^{\qquad}
  &
  \Aut(K)
  &
  \rOnto^{\qquad}
  &
  \Out(K)
  .
  \end{diagram}
\]
We say that
the triplet
$(\,
 G
 \,,\,
 i
 \,,\,
 \pi
 \,)
$
(or simply $G$ itself)
is an \emph{extension}
of $K$ by $Q$;
we refer to
$K$ as the \emph{kernel},
$Q$ as the \emph{quotient},
and
$\theta$ as the \emph{outer action}
of the extension.
Two extensions
$(\,
 G_\ell
 \,,\,
 i_\ell
 \,,\,
 \pi_\ell
 \,)
$
of $K$ by $Q$
(for $\ell=1,2$)
are \emph{isomorphic}
iff
there exists
an isomorphism
$\varphi : G_1 \rTo^{\simeq} G_2$
of groups
such that
$\varphi \circ i_1
 =
 i_2
$
and
$\pi_2 \circ \varphi
 =
 \pi_1
$.

When the groups
$K$ and $Q$
and the outer action
$\theta$
are given,
we may regard
the triplet
$(\,
 K
 \,,\,
 Q
 \,,\,
 \theta
 \,)
$
as constituting
an \emph{extension problem}.
The groups $G$
that can be obtained
as extensions
of $K$ by $Q$
with outer action $\theta$
can be classified:
according to
Eilenberg and MacLane
(cf.~\cite{EilenbergMacLane-GroupCohom-II}
theorem~11.1),
the set of
isomorphism classes of
all such extensions
forms a torsor
(possibly empty)
under the cohomology group
$H^2(Q,Z(K))$,
where the center $Z(K)$ of $K$
is regarded as a $Q$-module
via the action
induced by $\theta$.

\begin{center}
\begin{minipage}[t]{0.74\linewidth}
\setlength{\parindent}{20pt}
Let us now
iterate this process
of forming group extensions.
Thus,
we suppose
we have obtained
a group $N$
as an extension of
$K$ by $P$,
and we consider
an extension $G$
of $N$ by $R$
such that
$K$ is normal in $G$
(e.g.~when $K$ is
a characteristic subgroup
of $N$).
The group $G$ is then
an extension
whose kernel is $K$
and whose quotient $Q$
is itself
an extension of $P$ by $R$;
i.e.~we have
the lattice diagram
on the right.
What are the groups~$G$
which can be obtained
this way?
\end{minipage}
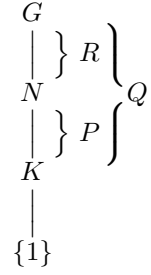
\hfill
\begin{minipage}[t]{0.19\linewidth}
\begin{footnotesize}
$ \begin{array}[t]{cll}
  G
  &
  &
  \\
  \Big{|}
  &
  \mathllap{
    \Big{\}}\
  }
  R
  &
  \mathllap{
    \smash{\bigg\rmoustache}
  }
  \\
  N
  &
  &
  Q
  \\
  \Big{|}
  &
  \mathllap{
    \Big{\}}\
  }
  P
  &
  \mathllap{
    \smash{\bigg\lmoustache}
  }
  \\
  K
  &
  &
  \\
  \Big{|}
  &
  &
  \\
  \{1\}
  &
  &
  \end{array}
$
\end{footnotesize}
\end{minipage}
\end{center}

To deal with this situation,
we fix the two
group extensions
\[
  (KNP)
  \ :\ 
  K
  \rInto^{\quad i_0 \quad}
  N
  \rOnto^{\quad \pi_0 \quad}
  P
  \qquad
  \text{and}
  \qquad
  (PQR)
  \ :\ 
  P
  \rInto^{\quad \jbar \quad}
  Q
  \rOnto^{\quad \phibar \quad}
  R
  ,
\]
and make
the following:

\begin{defn}
\label{defn:iterext}
An \emph{iterated extension}
of $(KNP)$ by $(PQR)$
is a triplet
$(\,
 G
 \,,\,
 j
 \,,\,
 \pi
 \,)
$
consisting of
a group $G$,
an injective homomorphism
$j : N \rInto G$
and a surjective homomorphism
$\pi : G \rOnto Q$,
such that
setting
$i := j \circ i_0$
and
$\phi := \phibar \circ \pi$,
one has
$\pi \circ j
 =
 \jbar \circ \pi_0$,
$j(N) = \Ker(\phi)$
and
$i(K) = \Ker(\pi)$;
in other words,
the following diagram
commutes
and has
exact rows and columns:
\begin{equation}
  \begin{diagram}
  K
  &
  \rInto^{\qquad i_0 \qquad}
  &
  N
  &
  \rOnto^{\qquad \pi_0 \qquad}
  &
  P
  \\
  \dEq
  &
  &
  \dInto_{j}
  &
  &
  \dInto_{\jbar}
  \\
  K
  &
  \rInto^{\qquad i \qquad}
  &
  G
  &
  \rOnto^{\qquad \pi \qquad}
  &
  Q
  \\
  &
  &
  \dOnto_{\phi}
  &
  &
  \dOnto_{\phibar}
  \\
  &
  &
  R
  &
  \rEq
  &
  R
  \end{diagram}
  \label{diag:KNGQR}
\end{equation}
Its \emph{$Q$-main extension}
is the extension
$(\,
 G
 \,,\,
 i
 \,,\,
 \pi
 \,)
$
of $K$ by $Q$
obtained by
setting
$i := j \circ i_0$.
Two iterated extensions
$(\,
 G_\ell
 \,,\,
 j_\ell
 \,,\,
 \pi_\ell
 \,)
$
of $(KNP)$ by $(PQR)$
(for $\ell=1,2$)
are \emph{isomorphic}
iff
there exists
an isomorphism
$\varphi : G_1 \rTo^{\simeq} G_2$
of groups
such that
$\varphi \circ j_1 = j_2$
and
$\pi_2 \circ \varphi = \pi_1$.
\end{defn}

When $P = \{1\}$
and hence
$R = Q$
and
$N = K$,
an iterated extension
of $(KNP)$ by $(PQR)$
reduces to
a usual group extension
of $K$ by $Q$
In general,
an iterated extension
of $(KNP)$ by $(PQR)$
always gives rise to
its $Q$-main extension
which is
a usual group extension
of $K$ by $Q$;
conversely:

\begin{defn}
\label{defn:P-subextension}
Let
$(\,
 G
 \,,\,
 i
 \,,\,
 \pi
 \,)
$
be an extension
of $K$ by $Q$.
Its \emph{$P$-subextension}
is the extension
$(\,
 N
 \,,\,
 i_0
 \,,\,
 \pi_0
 \,)
$
of $K$ by $P$
where
$N := \pi^{-1}(P)$,
and where
$i_0 : K \rInto N$
and
$\pi_0 : N \rOnto P$
are the homomorphisms
induced by
$i$ and $\pi$
respectively
(as in diagram~\eqref{diag:KNGQR}).
If we let
$j : N \rInto G$
denote
the canonical inclusion,
then
$(\,
 G
 \,,\,
 j
 \,,\,
 \pi
 \,)
$
is an iterated extension
of $(KNP)$ by $(PQR)$
as in definition~\ref{defn:iterext}.
\end{defn}

The purpose of
this work
is to illustrate
how iterated extensions
form
an integral part of
the theory of group extensions.
We will interpret
iterated extensions
in terms of
group cohomology,
extending the result of
Eilenberg-MacLane
mentioned earlier.
Specifically,
we fix
an outer action
$\theta : Q \rTo \Out(K)$
of $Q$ on $K$
and consider
the Lyndon-Hochschild-Serre spectral sequence
(cf.~\cite{Lyndon-CohomGroupExt},
\cite{HochschildSerre-CohomGroupExt})
for the extension $(PQR)$
with coefficients in $Z(K)$,
regarded as
an $Q$-module
via the action induced by $\theta$.
The $E_2$-terms
give the following
exact sequence:
\begin{equation}
  \begin{array}{lclclcl}
  0
  &
  \rTo
  &
  H^1(R,Z(K)^P)
  &
  \rTo^{\infl}
  &
  H^1(Q,Z(K))
  &
  \rTo^{\res}
  &
  H^1(P,Z(K))^{R}
  \\[2ex]
  &
  \rTo^{\tgr}
  &
  H^2(R,Z(K)^P)
  &
  \rTo^{\infl}
  &
  H^2_P(Q,Z(K))
  &
  \rTo^{\rd}
  &
  H^1(R,H^1(P,Z(K)))
  ,
  \end{array}
  \label{eqn:long exact seq}
\end{equation}
where
\[
  H^2_P(Q,Z(K))
  \ :=\
  \Ker
  \left(
    H^2(Q,Z(K))
    \rTo^{\res}
    H^2(P,Z(K))
  \right)
  .
\]
We will give
an explicit identification of
each cohomology group
and each morphism
appearing above
in terms of
iterated extensions
and associated notions.
These identifications
then enable us
to translate
the exactness of~\eqref{eqn:long exact seq}
into
natural relations between
(iterated) extensions,
their automorphism groups,
and their outer actions.
Our work extends
and (we hope) clarifies
that of Hochschild
in~\cite{Hochschild-BasicConstructions},
which gives
an interpretation
of the exactness of~\eqref{eqn:long exact seq}
entirely within
the context of
group extensions,
but under the assumption that
the group $K$
(in our notation here)
is abelian.
Much of
what is presented here
is probably known
in one form or another,
and we make
no claim of true originality
for any particular result
or construction.
Our aim coincides with
that of~\cite{Hochschild-BasicConstructions}:
to give
a self-contained
and systematic exposition
of the relevant ideas
---
one which (we hope)
will serve as
a basis for
future applications.

The main difficulty
in giving
a coherent discussion of
iterated group extensions
lies in
finding
the right generalization of
the notion of
outer action.
We believe
this is served by
the notion of
\emph{mod-$K$outer action},
which we introduce
and discuss
in section~\ref{sect:mod-K outer action}.
After that,
the paper is organized
``sequentially''
following
the long exact sequence~\eqref{eqn:long exact seq},
as a glance at the section headings
will reveal.
Throughout,
we adopt the convention that
the cohomology of various groups
are defined by normalized
cocycles and coboundaries
(cf.~\cite{EilenbergMacLane-GroupCohom-I}~\S6),
in the sense that
the relevant cochains
take on
the trivial value
whenever
any one of their arguments
is the identity element.
We also assume throughout that
any map between groups
sends the identity element
of the source
to the identity element
of the target;
this applies in particular
to sections 
and liftings
of homomorphisms.
These conventions
do not change
the substance of
our discussion,
but they do
tremendously simplify
the computations involved.
We use solid arrows
(such as $A \rTo B$)
to denote homomorphisms
and use dotted arrows
(such as $A \rDotsto B$)
to denote
maps between groups
which are not homomorphisms
---
for instance,
cochains, sections and liftings.

\section{Preliminaries}

\begin{notatn}
\label{notatn:fixed data}
For the rest of this paper,
we fix the triplet
$(\,
 K
 \,,\,
 PQR
 \,,\,
 \theta
 \,)
$
consisting of:
\[
  \begin{aligned}
    &
    \text{a group}
    &
    \quad
    &
    K
    ;
    \\
    &
    \text{an extension}
    &
    \quad
    &
    \smash{
    (PQR)
    \ :\ 
    P
    \rInto^{\quad \jbar \quad}
    Q
    \rOnto^{\quad \phibar \quad}
    R
    }
    ;
    \\
    \text{and}
    \quad
    &
    \text{an outer action}
    &
    \quad
    &
    \theta : Q \rTo \Out(K)
    \quad
    \text{of $Q$ on $K$}
    .
  \end{aligned}
  \qquad\qquad
\]
We let
\quad
$\theta|_P
 :
 P
 \rInto^{\jbar}
 Q
 \rTo^{\theta}
 \Out(K)
$
\quad
denote
the outer action
of $P$ on $K$
obtained
by restricting $\theta$
to $P$.
\end{notatn}

The outer action $\theta$
induces
an action of $Q$
on the center $Z(K)$ of $K$,
which we also denote by $\theta$.
Since
$P$ is a normal subgroup
of $Q$,
the subgroup $Z(K)^P$
of elements of $Z(K)$
fixed by $P$
is stable
under the action of $Q$.
Hence:

\begin{notatn}
\label{notatn:theta_0}
The outer action $\theta$
induces
an action of $R$
on $Z(K)^P$,
denoted by
\[
  \theta_0
  \ :\ 
  R
  \rTo
  \Aut(Z(K)^P)
  .
\]
An element $r \in R$
sends $z \in Z(K)^P$
to
$\upperleft
 {\theta_0(r)}
 {z}
 =
 \upperleft
 {\theta(q)}
 {z}
$,
where $q \in Q$
is any element
such that
$\phibar(q) = r$
in $R$.
If
$(\,
 G
 \,,\,
 i
 \,,\,
 \pi
 \,)
$
is an extension
of $K$ by $Q$
with outer action $\theta$,
and we set
$\phi := \phibar \circ \pi$,
then
$\upperleft
 {\theta_0(r)}
 {z}
 =
 g
 \cdot
 z
 \cdot
 g^{-1}
$
for any $g \in G$
such that
$\phi(g) = r$
in $R$.
\end{notatn}

The action $\theta_0$
of $R$ on $Z(K)^P$
is the one
which is used
for defining
$H^1(R,Z(K)^P)$
and
$H^2(R,Z(K)^P)$
in the exact sequence~\eqref{eqn:long exact seq};
these groups
will be discussed
in sections~\ref{sect:iterext aut}
and~\ref{sect:iterext classification}
respectively.

\begin{notatn}
\label{notatn:Q action on Z^1}
The outer action $\theta$
induces
an action of $Q$
on the abelian group of
all maps
from $P$ to $Z(K)$:
an element
$q \in Q$
sends such a map
$\lambda$
to $\upperleft{q}{\lambda}$
given by
\[
  \upperleft{q}{\lambda}(p)
  \ :=\
  \upperleft
  {\theta(q)}
  {\lambda(q^{-1} p q)}
  .
\]
It is clear that
this action
normalizes the subgroups
of 1-cocycles
and 1-coboundaries;
hence
$\theta$ induces
an action of $Q$
on $Z^1(P,Z(K))$
and on $B^1(P,Z(K))$.
\end{notatn}

Passing to the quotient,
one obtains
an action of $Q$
on $H^1(P,Z(K))$
induced by $\theta$,
which is also trivial
when restricted to
the subgroup $P$.
To verify the latter claim,
first note that
for any
$\lambda \in Z^1(P,Z(K)))$
and any
$p_0 \in P \subseteq Q$,
one has
$1
 =
 \lambda(p_0 p_0^{-1})
 =
 \lambda(p_0)
 \cdot
 \upperleft
 {\theta|_P(p_0)}
 {\lambda(p_0^{-1})}
$,
which implies that
$\upperleft
 {\theta|_P(p_0)}
 {\lambda(p_0^{-1})}
 =
 \lambda(p_0)^{-1}
$.
Thus
for any $p \in P$,
one has
\[
  \begin{aligned}
  \upperleft{p_0}{\lambda}
  (p)
  \ =\
  \upperleft
  {\theta|_P(p_0)}
  {\lambda(p_0^{-1}\,p\,p_0)}
  &
  \ =\
  \upperleft
  {\theta|_P(p_0)}
  {\left(\,
    \lambda(p_0^{-1})
    \cdot
    \upperleft
    {\theta|_P(p_0^{-1})}
    {\lambda(p\,p_0)}
   \,\right)
  }
  \\
  &
  \ =\
  \upperleft
  {\theta|_P(p_0)}
  {\lambda(p_0^{-1})}
  \cdot
  \lambda(p)
  \cdot
  \upperleft
  {\theta|_P(p)}
  {\lambda(p_0)}
  \ =\
  z_0^{-1}
  \cdot
  \upperleft
  {\theta|_P(p)}
  {z_0}
  \cdot
  \lambda(p)
  \end{aligned}
\]
where
$z_0 := \lambda(p_0)
 \in Z(K)
$;
in other words,
$\upperleft{p_0}{\lambda}
 =
 (\partial z_0)
 \cdot
 \lambda
$
in $Z^1(P,Z(K))$.
Hence:

\begin{notatn}
\label{notatn:R action on H^1}
The outer action $\theta$
induces
an action of $R$
on $H^1(P,Z(K))$:
an element $r \in R$
sends
$[\lambda]
 \in
 H^1(P,Z(K))
$
to the cohomology class
$[\upperleft{q}{\lambda}]
 \in
 H^1(P,Z(K))
$
of the 1-cocycle
$\upperleft{q}{\lambda}$,
where $q \in Q$
is any element
such that
$\phibar(q) = r$
in $R$.
\end{notatn}

The above action
of $R$ on $H^1(P,Z(K)^P)$
is that used for
defining
the cohomology groups
$H^1(R,H^1(P,Z(K)))$
in the exact sequence~\eqref{eqn:long exact seq};
this group
will be discussed
in section~\ref{sect:mod-K outer action classification}.

\section{Mod-$K$\,outer automorphism groups
and mod-$K$\,outer actions}
\label{sect:mod-K outer action}

\begin{defn}
\label{defn:Out(N;K)}
Let
\quad
$\smash{
 (KNP)
 \ :\ 
 K
 \rInto^{\quad i_0 \quad}
 N
 \rOnto^{\quad \pi_0 \quad}
 P
}$
\quad
be any extension
of $K$ by $P$.
We let
$\Aut_K(N)
 :=
 \{\ 
   \eta \in \Aut(N)
   \,:\,
   \eta(K) = K
 \ \}
$
denote
the group of
automorphisms of $N$
stabilizing $K$.
It contains
the normal subgroup
$\conj_N(K)$
consisting of
inner automorphisms
induced by
elements of $K$.
The \emph{mod-$K$outer automorphism group of $N$}
is the quotient group
\[
  \Out(N;K)
  \ :=\ 
  \dfrac{\Aut_K(N)}{\conj_N(K)}
  .
\]
A \emph{mod-$K$outer action on $N$}
is a homomorphism
(from the acting group)
to $\Out(N;K)$.
\end{defn}

The group
$\Out(N;K)$
serves as
an intermediary between
the outer automorphism groups
of $K$, $N$ and $P$:
one has the diagram
\[
  \begin{array}{rcrclcl}
  &
  \Aut(K)
  &
  \lTo^{\quad NK \quad}
  &
  \Aut_K(N)
  &
  \rTo^{\quad NP \quad}
  &
  \Aut(P)
  &
  \\
  \mathrlap{
    \smash{\bigg\lmoustache}
  }
  &
  \Big{|}
  &
  \mathrlap{
    \smash{\bigg\lmoustache}
  }
  &
  \Big{|}
  &
  \mathllap{
    \Big{\}}\
  }
  \Out_K(N)
  &
  \Big{|}
  &
  \mathllap{
    \Big{\}}\
  }
  \Out(P)
  \\
  \Out(K)
  &
  \Big{|}
  &
  \mathllap{
    \Out(N;K)
  }
  &
  \Inn(N)
  &
  \rOnto^{\quad\quad\quad}
  &
  \Inn(P)
  &
  \\
  \mathrlap{
    \smash{\bigg\rmoustache}
  }
  &
  \Big{|}
  &
  \mathrlap{
    \smash{\bigg\rmoustache}
  }
  &
  \Big{|}
  &
  &
  \Big{|}
  &
  \\
  &
  \Inn(K)
  &
  \lOnto^{\quad\quad\quad}
  &
  \conj_N(K)
  &
  \rOnto^{\quad\quad\quad}
  &
  \{1\}
  &
  \end{array}
\]
in which
$NK$ and $NP$
are the canonical homomorphisms
obtained by
considering the effects
induced on $K$ and on $P$
respectively
by an automorphism of $N$
which stabilizes $K$.
From this,
we see that
there are
canonical homomorphisms
making the following diagram
commute:
\[
  \begin{diagram}
  \Aut(K)
  &
  \lTo^{\quad NK \quad}
  &
  \Aut_K(N)
  &
  &
  \\
  \dOnto
  &
  &
  \dOnto
  &
  \rdTo^{\quad NP \quad}
  \qquad
  &
  \\
  \Out(K)
  &
  \lTo
  &
  \Out(N;K)
  &
  \rTo^{\qquad\quad}
  &
  \Aut(P)
  \\
  &
  &
  \dOnto
  &
  &
  \dOnto
  \\
  &
  &
  \Out_K(N)
  &
  \rTo^{\qquad\quad}
  &
  \Out(P)
  .
  \end{diagram}
\]
A mod-$K$outer action on $N$
thus induces
an outer action on $K$
and
a ``true'' action on $P$.
\begin{footnote}
{\quad
There is also
an induced
outer action on $N$,
but this will not be
important
for our discussion here.
}
\end{footnote}

\begin{defn}
\label{defn:mod-K outer action of ext}
Let
\quad
$\smash{
 (KNP)
 \ :\ 
 K
 \rInto^{\quad i_0 \quad}
 N
 \rOnto^{\quad \pi_0 \quad}
 P
}$
\quad
be an extension
of $K$ by $P$.
The conjugation action $\conj_N$
of $N$ on itself
induces a homomorphism
$\Theta_P : P \rTo \Out(N;K)$
making
the following diagram
commute:
\[
  \begin{diagram}
  K
  &
  \rInto^{i_0}
  &
  N
  &
  \rOnto^{\pi_0}
  &
  P
  \\
  \dOnto
  &
  &
  \dTo^{\conj_N}
  &
  &
  \dTo
  \
  \Theta_P
  \\
  \conj_N(K)
  &
  \rInto^{\qquad}
  &
  \Aut_K(N)
  &
  \rOnto^{\qquad}
  &
  \Out(N;K)
  .
  \end{diagram}
\]
The homomorphism
$\Theta_P : P \rTo \Out(N;K)$
is called
the \emph{mod-$K$outer action
of the extension $(KNP)$}.
\end{defn}

Suppose the extension~$(KNP)$
has outer action
given by $\theta|_P$.
The mod-$K$outer action $\Theta_P$
then induces both
the outer action $\theta|_P$
of $P$ on $K$
as well as
the conjugation action $\conj_P$
of $P$ on itself,
thus making
the following diagram
commute:
\[
  \begin{diagram}
  &
  &
  P
  &
  \rOnto
  &
  \Inn(P)
  \\
  &
  \ldTo^{\theta|_P}
  &
  \dTo
  \
  \Theta_P
  &
  \rdTo^{\conj_P}
  &
  \dInto
  \\
  \Out(K)
  &
  \lTo^{\qquad}
  &
  \Out(N;K)
  &
  \rTo^{\qquad}
  &
  \Aut(P)
  .
  \end{diagram}
\]
The given outer action
$\theta : Q \rTo \Out(K)$
of $Q$ on $K$
is a homomorphism
which prolongs
the outer action $\theta|_P$
of $P$ on $K$;
one has
$\theta \circ \jbar
 =
 \theta|_P
$.
On the other hand,
in the extension~$(PQR)$,
the conjugation action
of $Q$ on $P$
is a homomorphism
$\conj_P^Q : Q \rTo \Aut(P)$
which prolongs
the conjugation action $\conj_P$
of $P$ on itself;
one has
$\conj_P^Q \circ \jbar
 =
 \conj_P
$.
The homomorphisms
$\theta$ and $\conj_P^Q$
thus make
the following diagram
commute:
\begin{equation}
  \begin{diagram}
  &
  &
  P
  &
  &
  &
  &
  \\
  &
  \ldTo(2,5)^{\theta|_P}
  &
  &
  \rdTo(2,5)^{\Theta_P}
  \rdTo(4,3)^{\conj_P}
  &
  &
  &
  \\
  &
  &
  \dInto_{\jbar}
  &
  &
  &
  &
  \\
  &
  &
  Q
  &
  &
  \rTo^{\conj_P^Q}
  &
  &
  \Aut(P)
  \\
  &
  \ldTo^{\theta}
  \qquad
  \qquad
  &
  &
  &
  &
  \ruTo
  &
  \\
  \Out(K)
  &
  \lTo
  &
  &
  &
  \Out(N;K)
  &
  &
  \end{diagram}
  \label{diag:theta conj_P^Q}
\end{equation}

\begin{defn}
\label{defn:theta conj_P^Q prolongation}
Let $\Theta_P$ be
as in definition~\ref{defn:mod-K outer action of ext}.
A \emph{prolongation of $\Theta_P$}
is a mod-$K$outer action
$\Theta
 :
 Q
 \rTo
 \Out(N;K)
$
of $Q$ on $N$
such that
$\Theta \circ \jbar
 =
 \Theta_P
$
as homomorphisms
$P \rTo \Out(N;K)$.
When the extension $(KNP)$
has outer action
given by $\theta|_P$,
we may speak of
a \emph{$(\theta,\conj_P^Q)$-prolongation of $\Theta_P$},
which is
a prolongation $\Theta$ of $\Theta_P$
that can be inserted into
the diagram~\eqref{diag:theta conj_P^Q}
to make it commutative,
i.e.~such that
the following diagram
commutes:
\[
  \begin{diagram}
  &
  &
  Q
  &
  &
  \\
  &
  \ldTo^{\theta}
  &
  \dTo
  \
  \Theta
  &
  \rdTo^{\conj_P^Q}
  &
  \\
  \Out(K)
  &
  \lTo^{\qquad}
  &
  \Out(N;K)
  &
  \rTo^{\qquad}
  &
  \Aut(P)
  .
  \end{diagram}
\]
\end{defn}

\begin{defn}
\label{defn:mod-K outer action of iterext}
Let
$(\,
 G
 \,,\,
 j
 \,,\,
 \pi
 \,)
$
be an iterated extension
of $(KNP)$ by $(PQR)$,
as in definition~\ref{defn:iterext}.
The conjugation action
$\conj_N^G : G \rTo \Aut_K(N)$
of $G$ on $N$
stabilizes $K$
and induces
a homomorphism
$\Theta : Q \rTo \Out(N;K)$
making the following diagram
commute:
\begin{equation}
  \begin{diagram}
  \mathllap{
    (KGQ)
    \ :\ 
    \quad
  }
  K
  &
  \rInto^{i}
  &
  G
  &
  \rOnto^{\pi}
  &
  Q
  \\
  \dOnto
  &
  &
  \dTo^{\conj_N^G}
  &
  &
  \dTo
  \
  \Theta
  \\
  \conj_N(K)
  &
  \rInto^{\qquad}
  &
  \Aut_K(N)
  &
  \rOnto^{\qquad}
  &
  \Out(N;K)
  .
  \end{diagram}
  \label{diag:G Theta}
\end{equation}
The homomorphism
$\Theta : Q \rTo \Out(N;K)$
is called
the \emph{mod-$K$outer action}
of the iterated extension
$(\,
 G
 \,,\,
 j
 \,,\,
 \pi
 \,)
$.
If the $Q$-main extension
$(KGQ)$
of the iterated extension
has outer action
given by $\theta$,
the mod-$K$outer action $\Theta$
is a $(\theta,\conj_P^Q)$-prolongation
of $\Theta_P$.
\end{defn}

\begin{center}
\begin{minipage}[t]{0.60\linewidth}
\relax
\begin{defn}
\label{defn:iterext pbm}
Generalizing
the notion of
an extension problem,
we define
an \emph{iterated extension problem}
as a triplet
$(\,
 KNP
 \,,\,
 PQR
 \,,\,
 \Theta
 \,)
$
in which
$(KNP)$ and $(PQR)$
are group extensions
and
$\Theta$
is a mod-$K$outer action
of $Q$ on $N$,
satisfying
the following conditions:
the outer action of
the extension~$(KNP)$
is $\theta|_P$,
and
$\Theta$ is a
$(\theta,\conj_P^Q)$-prolongation of
the mod-$K$outer action
$\Theta_P$
of $P$ on $N$
induced by
the extension $(KNP)$.
These data
are conveniently organized
in the form of
the diagram
on the right.
\end{defn}
\end{minipage}
\hfill
\begin{minipage}[t]{0.35\linewidth}
\medskip
$\begin{diagram}[t]
  K
  &
  \rInto^{\qquad i_0 \qquad}
  &
  N
  &
  \rOnto^{\qquad \pi_0 \qquad}
  &
  P
  \\
  \dEq
  &
  &
  &
  \Theta
  &
  \dInto_{\jbar}
  \\
  K
  &
  &
  &
  &
  Q
  \\
  &
  &
  &
  &
  \dOnto_{\phibar}
  \\
  &
  &
  R
  &
  \rEq
  &
  R
  \end{diagram}
$
\end{minipage}
\end{center}

When the group $K$
is contained in
the center $Z(N)$ of $N$,
the mod-$K$outer action $\Theta$
becomes
a ``true'' action
$Q \rTo \Aut_K(N)$
of $Q$ on $N$
which induces
the conjugation action
of $Q$ on $P$;
the iterated extension problem
$(\,
 KNP
 \,,\,
 PQR
 \,,\,
 \Theta
 \,)
$
then amounts to
a \emph{crossed module}
in the sense of
Whitehead
(via the composite homomorphism
$N
 \rTo^{\quad \jbar \,\circ\, \pi_0 \quad}
 Q
$;
cf.~\cite{Whitehead-CombinatorialHomotopy}~\S2).
If $K$ is in fact
equal to
the center $Z(N)$ of $N$,
we obtain
the notion of
an \emph{$S$-exact sequence}
considered by MacLane
in~\cite{MacLane-GroupCohom-III}~\S2.

\section{$H^1(R,Z(K)^P)$
and the automorphisms of
iterated extensions}
\label{sect:iterext aut}

Let
$(\,
 G
 \,,\,
 j
 \,,\,
 \pi
 \,)
$
be an iterated extension
of $(KNP)$ by $(PQR)$,
whose $Q$-main extension $(KGQ)$
has outer action $\theta$.
In accordance
with~definition~\ref{defn:iterext},
the automorphism group
of the iterated extension
is
\[
  \Aut(KNGQR)
  \ :=\
  \{\
    \xi \in \Aut(G)
    \ :\
    \text{$\xi \circ j = j$
          and
          $\pi \circ \xi = \pi$}
  \ \}
  .
\]

\begin{thm}
\label{thm:iterext aut}
The map
\[
  \begin{array}{rcl}
  - \star
  \ :\
  Z^1(R,Z(K)^P)
  &
  \rTo^{\simeq}
  &
  \Aut(KNGQR)
  ,
  \\[1ex]
  \lambda
  &
  \rMapsto
  &
  \text{the map}
  \quad
  \lambda \star
  \ :=\ 
  \bigl(\
  g
  \mapsto
  \lambda(\phi(g))
  \cdot
  g
  \ \bigr)
  ,
  \end{array}
\]
is a well-defined
isomorphism of groups.
\end{thm}

Here,
$Z(K)^P$ is regarded as
an $R$-module
via the action $\theta_0$
as in notation~\ref{notatn:theta_0}.
Note that
$Z^1(R,Z(K)^P)$ depends only on
the given data
$(\,
 K
 \,,\,
 PQR
 \,,\,
 \theta
 \,)
$
as in notation~\ref{notatn:fixed data},
whereas
the automorphism group
$\Aut(KNGQR)$ is defined
only when
the iterated extension
$(\,
 G
 \,,\,
 j
 \,,\,
 \pi
 \,)
$
is given.

\begin{proof}
Let
$\lambda
 \in
 Z^1(R,Z(K)^P)
$
be any 1-cocycle,
and let
$\xi := \lambda \star$
be the map
from $G$ to itself
given by
$\upperleft{\xi}{g}
 :=
 \lambda(\,\phi(g)\,)
 \cdot
 g
$.
For any
$g_1,g_2 \in G$,
the cocycle relation
satisfied by $\lambda$
yields
\[
  \begin{aligned}
  \lambda(\,\phi(g_1)\,\phi(g_2)\,)
  \cdot
  g_1
  \cdot
  g_2
  &
  \ =\
  \lambda(\,\phi(g_1)\,)
  \cdot
  \underbrace
  { \upperleft
    {\theta_0(\phi(g_1))}
    {\lambda(\,\phi(g_2)\,)}
  }_{\ =\
     g_1
     \cdot
     \lambda(\,\phi(g_2)\,)
     \cdot
     g_1^{-1}}
  \cdot
  g_1
  \cdot
  g_2
  \\
  &
  \ =\
  \lambda(\,\phi(g_1)\,)
  \cdot
  g_1
  \cdot
  \lambda(\,\phi(g_2)\,)
  \cdot
  g_2
  \qquad
  \text{in $G$}
  ,
  \end{aligned}
\]
which shows that
$\upperleft{\xi}{(g_1g_2)}
 =
 \upperleft{\xi}{g_1}
 \cdot
 \upperleft{\xi}{g_2}
$;
thus $\xi$ is
an endomorphism of $G$.
As $\lambda(1_R) = 1_{Z(K)^P}$,
we have
$\xi \circ j = j$,
and since
$\lambda$ takes values
in $Z(K)^P \subseteq K$,
we have
$\pi \circ \xi = \pi$.
It follows that
$\xi \in \Aut(KNGQR)$
is an automorphism
of the iterated extension.
The map $- \star$
which sends $\lambda$ to $\xi$
is thus
a well-defined map
from $Z^1(R,Z(K)^P)$
to $\Aut(KNGQR)$.
For any
$\lambda_1,\lambda_2 \in Z^1(R,Z(K)^P)$,
applying the automorphism
$\lambda_2 \star$
followed by
$\lambda_1 \star$
to $g \in G$
gives
\[
  \lambda_1
  \bigl(\,
  \underbrace{
    \phi
    \bigl(
      \lambda_2(\phi(g))
      \cdotp
      g
    \bigr)
  }_{
    \ =\ 
    \phi(g)
  }
  \,\bigr)
  \cdot
  \lambda_2(\phi(g))
  \cdot
  g
  \ =\ 
  (\lambda_1\lambda_2)(\phi(g))
  \cdot
  g
  \qquad
  \text{in $G$}
  ,
\]
which is the same as
applying
$(\lambda_1\lambda_2) \star$
to $g$;
this shows that
$- \star$ is
a group homomorphism.
If $- \star$
sends
$\lambda \in Z^1(R,Z(K)^P)$
to
$\id_G \in \Aut(KNGQR)$,
then
$\lambda(\phi(g))
 =
 1_G
$
for every $g \in G$,
which implies that
$\lambda$ is
the trivial 1-cocycle;
hence $- \star$ is
injective.

We now show that
$- \star$ is surjective.
Given an automorphism
$\xi \in \Aut(KNGQR)$,
we choose
any section
$u : R \rDotsto G$
of $G \rOnto^{\phi} R$,
and define the map
$\lambda : R \rDotsto G$
by
$\lambda(r)
 :=
 \upperleft{\xi}{u(r)}
 \cdot
 u(r)^{-1}
$.
(It will be seen
eventually that
$\lambda$ is in fact
independent of the choice of
the section $u$.)
For any $r \in R$,
the fact that
$\pi \circ \xi = \pi$
implies that
$\pi
 (\,
   \upperleft{\xi}{u(r)}
 \,)
 \ =\
 \pi
 (\,
 u(r)
 \,)
$
in $Q$;
hence
$\lambda$ takes values
in $K$.
On the other hand,
the fact that
$\xi \circ j = j$
implies that
for any $n \in N$,
one has
$\upperleft{\xi}{j(n)}
 =
 j(n)
$,
and hence
\[
  \begin{aligned}
  j
  \Bigl(\
    \upperleft
    {\conj_N(\lambda(r))}
    {\bigl(\,
       \upperleft
       {\conj_N^G(u(r))}
       {n}
     \,\bigr)
    }
  \ \Bigr)
  &
  \ =\
  \upperleft{\xi}{u(r)}
  \cdot
  j(n)
  \cdot
  \upperleft{\xi}{u(r)}^{-1}
  \\
  &
  \ =\
  \upperleft
  {\xi}
  {\Bigl(\ 
    u(r)
    \cdot
    j(n)
    \cdot
    u(r)^{-1}
   \ \Bigr)
  }
  \\
  &
  \ =\
  \upperleft
  {\xi}
  {\Bigl(\ 
   j
   \bigl(\,
     \upperleft
     {\conj_N^G(u(r))}
     {n}
   \,\bigr)
   \ \Bigr)
  }
  \ =\
  j
  \left(\,
    \upperleft
    {\conj_N^G(u(r))}
    {n}
  \,\right)
  \qquad
  \text{in $G$}
  ,
  \end{aligned}
\]
which implies that
$\conj_N(\,\lambda(r)\,)
 =
 \id_N
$;
this shows that
$\lambda$ takes values in
$Z(K)^P = Z(N) \cap K$.
Now let
$f
 :
 R \times R
 \rDotsto
 N
$
be the (right) factor set
corresponding to
the section $u$,
characterized by
the property that
for any $r_1,r_2 \in R$,
one has
\[
  u(r_1)
  \cdot
  u(r_2)
  \ =\
  u(r_1r_2)
  \cdot
  j
  (\,
    f(r_1,r_2)
  \,)
  \qquad
  \text{in $G$}
  .
\]
Using the fact that
$\upperleft
 {\xi}
 {\bigl(\,
    j(\,f(r_1,r_2)\,)
  \,\bigr)
 }
 =
 j(\,f(r_1,r_2)\,)
$,
we have
\[
  \begin{aligned}
  \lambda(r_1r_2)
  &
  \ =\
  \upperleft{\xi}{u(r_1r_2)}
  \cdot
  u(r_1r_2)^{-1}
  \\
  &
  \ =\
  \upperleft
  {\xi}
  {\bigl(\,
    u(r_1)
    \cdot
    u(r_2)
    \cdot
    j(\,f(r_1,r_2)\,)^{-1}
   \,\bigr)
  }
  \cdot
  \bigl(\,
    u(r_1)
    \cdot
    u(r_2)
    \cdot
    j(\,f(r_1,r_2)\,)^{-1}
  \,\bigr)^{-1}
  \\
  &
  \ =\
  \upperleft{\xi}{u(r_1)}
  \cdot
  \upperleft{\xi}{u(r_2)}
  \cdot
  u(r_2)^{-1}
  \cdot
  u(r_1)^{-1}
  \\
  &
  \ =\
  \lambda(r_1)
  \cdot
  u(r_1)
  \cdot
  \lambda(r_2)
  \cdot
  u(r_1)^{-1}
  \ =\
  \lambda(r_1)
  \cdot
  \upperleft
  {\theta_0(r_1)}
  {\lambda(r_2)}
  ,
  \end{aligned}
\]
and so
$\lambda : R \rDotsto Z(K)^P$
is a 1-cocycle.
The homomorphism $- \star$
maps
$\lambda \in Z^1(R,Z(K)^P)$
to the automorphism
$\xi' \in \Aut(KNGQR)$
which sends
an arbitrary element
$g \in G$,
written in the form
$g = u(r) \cdot j(n)$
(with $n \in N$ and $r \in R$),
to the element
\[
  \upperleft
  {\xi'}
  {g}
  \ =\
  \lambda(r)
  \cdot
  u(r)
  \cdot
  j(n)
  \ =\
  \bigl(\,
  \upperleft
  {\xi}
  {u(r)}
  \cdot
  u(r)^{-1}
  \,\bigr)
  \cdot
  \bigl(\,
  u(r)
  \cdot
  \upperleft
  {\xi}
  {j(n)}
  \,\bigr)
  \ =\
  \upperleft
  {\xi}
  {\bigl(\,
    u(r)
    \cdot
    j(n)
   \,\bigr)
  }
  \ =\
  \upperleft
  {\xi}
  {g}
  \qquad
  \text{in $G$}
  .
\]
This shows that
$\xi' = \xi$,
and hence
$- \star$ is surjective.
\end{proof}

\begin{rmk}
The proof shows that
the inverse of
the isomorphism $- \star$
of theorem~\ref{thm:iterext aut}
is given by
\[
  \Aut(KNGQR)
  \rTo^{\simeq}
  Z^1(R,Z(K)^P)
  ,
  \qquad
  \xi
  \rMapsto
  \bigl(\,
    r
    \mapsto
    \upperleft{\xi}{u(r)}
    \cdot
    u(r)^{-1}
  \,\bigr)
  ,
\]
for any choice of
a section $u$ of $\phi$.
\end{rmk}

\begin{rmk}
It follows from
theorem~\ref{thm:iterext aut}
that
the group
$\Aut(KNGQR)$
is an abelian subgroup of
$\Aut(G)$,
which
(via the canonical isomorphism $- \star$
of the theorem)
depends only on
the given data
$(\,
 K
 \,,\,
 PQR
 \,,\,
 \theta
 \,)
$
(cf.~notation~\ref{notatn:fixed data})
and not on
the iterated extension~$(KNGQR)$.
\end{rmk}

The automorphism group
$\Aut(KNGQR)$
of the iterated extension
$(\,
 G
 \,,\,
 j
 \,,\,
 \pi
 \,)
$
contains
the normal subgroup
\[
  \conj_G(Z(K)^P)
  \ =\
  \{\
    \conj_G(z) \in \Inn(G)
    \ :\
    z \in Z(K)^P
  \ \}
\]
consisting of
inner automorphisms of $G$
induced by
elements of $Z(K)^P$.
The fact that
$Z(K)^P = Z(N) \cap K$
imples
\[
  \conj_G(Z(K)^P)
  \ =\ 
  \conj_G(K)
  \ \cap\ 
  \Aut(KNGQR)
  \qquad
  \text{in $\Aut(G)$}
  .
\]

\begin{cor}
The isomorphism $- \star$
of theorem~\ref{thm:iterext aut}
restricts to
an isomorphism
\[
  \begin{array}{rcl}
  - \star
  \ :\
  B^1(R,Z(K)^P)
  &
  \rTo^{\simeq}
  &
  \conj_G(Z(K)^P)
  ,
  \\[1ex]
  \partial z_0
  &
  \rMapsto
  &
  \conj_G(z_0^{-1})
  ,
  \end{array}
\]
where
$\partial z_0$ denotes
the 1-coboundary
$\partial z_0(r)
 :=
 z_0^{-1}
 \cdot
 \upperleft
 {\theta_0(r)}
 {z_0}
$
for any
$z_0 \in Z(K)^P$.
(Note the presence of
the inversion
in $\conj_G(z_0^{-1})$
corresponding to $\partial z_0$.)
\end{cor}

\begin{proof}
For any $z_0 \in Z(K)^P$,
the coboundary
$\partial z_0 \in B^1(R,Z(K)^P)$
is mapped by $- \star$
to the automorphism of $G$
which sends
an arbitrary element
$g \in G$
to
\[
  (\partial z_0)
  (\phibar(g))
  \cdot
  g
  \ =\ 
  z_0^{-1}
  \cdot
  \underbrace{
  \upperleft
  {\theta_0(\phi(g))}
  {z_0}
  }_{\mathclap{
    \ =\
    g
    \cdot
    z_0
    \cdot
    g^{-1}
  }}
  \cdot
  g
  \ =\
  z_0^{-1}
  \cdot
  g
  \cdot
  z_0
  \ =\ 
  \upperleft
  {\conj_G(z_0^{-1})}
  {g}
  \qquad
  \text{in $G$}
  .
\]
Thus
$- \star$
maps
$\partial z_0$
to
$\conj_G(z_0^{-1})$
in $\Aut(KNGQR)$,
and the corollary follows.
\end{proof}

Extending the notation
in definition~\ref{defn:Out(N;K)},
we let
$\Aut_{N,K}(G)$ denote
the subgroup of
$\Aut_K(G)$
consisting of
automorphisms of $G$
stabilizing
both $N$ and $K$;
it also contains
$\conj_G(K)$
as a normal subgroup,
and we let
$\Out_N(G;K)
 :=
 \Aut_{N,K}(G)
 \,/\,
 \conj_G(K)
$
denote the quotient group.
Let $\Out(KNGQR;K)$ denote
the image of $\Aut(KNGQR)$
in $\Out_N(G;K)$.
There are
canonical homomorphisms
from $\Aut_{N,K}(G)$
to $\Aut_K(N)$ and $\Aut(Q)$,
obtained by considering
the effects
induced on $N$ and on $Q$
respectively
by an automorphism of $G$
which stabilizes
both $N$ and $K$;
passing to the quotient
modulo $\conj_G(K)$,
these induce
corresponding homomorphisms
from $\Out_N(G;K)$
to $\Out(N;K)$ and $\Aut(Q)$,
and we have
\[
  \Out(KNGQR;K)
  \ =\ 
  \Ker
  \Bigl(\ 
    \Out_N(G;K)
    \rTo
    \Out(N;K) \times \Aut(Q)
  \ \Bigr)
  \qquad
  \text{in $\Out_N(G;K)$}
  .
\]
In virtue of
the identity
$\conj_G(Z(K)^P)
 =
 \conj_G(K)
 \cap
 \Aut(KNGQR)
$,
we also have
\[
  \Out(KNGQR;K)
  \ =\
  \dfrac{\Aut(KNGQR)}{\conj_G(Z(K)^P)}
  ,
\]
and hence:

\begin{cor}
\label{cor:iterext aut H^1}
The isomorphism $- \star$
of theorem~\ref{thm:iterext aut}
induces an isomorphism
\[
  - \star
  \ :\
  H^1(R,Z(K)^P)
  \rTo^{\simeq}
  \Out(KNGQR;K)
  \quad
  \ \subseteq\ 
  \Out_N(G;K)
  .
\]
\end{cor}

\section{$H^1(Q,Z(K))$
and the automorphisms of
extensions}

Let
$(\,
 G
 \,,\,
 i
 \,,\,
 \pi
 \,)
$
be an extension
of $K$ by $Q$
with outer action
$\theta$.
Its automorphism group
\[
  \Aut(KGQ)
  \ :=\
  \{\
    \xi \in \Aut(G)
    \ :\
    \text{$\xi \circ i = i$
          and
          $\pi \circ \xi = \pi$}
  \ \}
\]
contains
the normal subgroup
\[
  \conj_G(Z(K))
  \ =\
  \{\
    \conj_G(z) \in \Inn(G)
    \ :\
    z \in Z(K)
  \ \}
  \ =\ 
  \conj_G(K)
  \cap
  \Aut(KGQ)
\]
consisting of
inner automorphisms of $G$
induced by
elements of $Z(K)$.
Let $\Out(KGQ;K)$ denote
the image of $\Aut(KGQ)$
in $\Out(G;K)$;
one has
\[
  \Out(KGQ;K)
  \ =\ 
  \dfrac{\Aut(KGQ)}{\conj_G(Z(K))}
  \ =\ 
  \Ker
  \Bigl(\ 
    \Out(G;K)
    \rTo
    \Out(K) \times \Aut(Q)
  \ \Bigr)
  .
\]

The results of
section~\ref{sect:iterext aut}
specialize
to analogous results
for the extension $(KGQ)$
by putting $P = \{1\}$
and hence
$R = Q$,
$\phi = \pi$,
and
$N = K$,
$j = i$.
We state these results
in this section
for later references.

\begin{thm}
\label{thm:ext aut}
The map
\[
  \begin{array}{rcl}
  - \star
  \ :\ 
  Z^1(Q,Z(K))
  &
  \rTo^{\simeq}
  &
  \Aut(KGQ)
  ,
  \\[1ex]
  \lambda
  &
  \rMapsto
  &
  \text{the map}
  \quad
  \lambda \star
  \ :=\ 
  \bigl(\
  g
  \mapsto
  \lambda(\pi(g))
  \cdot
  g
  \ \bigr)
  ,
  \end{array}
\]
is a well-defined
isomorphism of groups,
whose inverse
is given by
\[
  \Aut(KGQ)
  \rTo^{\simeq}
  Z^1(Q,Z(K))
  ,
  \qquad
  \xi
  \rMapsto
  \bigl(\,
    q
    \mapsto
    \upperleft{\xi}{s(q)}
    \cdot
    s(q)^{-1}
  \,\bigr)
  ,
\]
for any choice of
a section $s$ of $\pi$.
Thus
the group
$\Aut(KGQ)$
is an abelian subgroup of
$\Aut(G)$
which depends only on
the extension problem
$(\,
 K
 \,,\,
 Q
 \,,\,
 \theta
 \,)
$
and not on
the extension~$(KGQ)$.
\end{thm}

\begin{cor}
\label{cor:ext aut B^1}
The isomorphism of theorem~\ref{thm:ext aut}
restricts to
an isomorphism
\[
  \begin{array}{rcl}
  - \star
  \ :\ 
  B^1(Q,Z(K))
  &
  \rTo^{\simeq}
  &
  \conj_G(Z(K))
  ,
  \\[1ex]
  \partial z_0
  &
  \rMapsto
  &
  \conj_G(z_0^{-1})
  ,
  \end{array}
\]
where
$\partial z_0$ denotes
the 1-coboundary
$\partial z_0(q)
 :=
 z_0^{-1}
 \cdot
 \upperleft
 {\theta(q)}
 {z_0}
$
for any
$z_0 \in Z(K)$.
(Note the presence of
the inversion
in $\conj_G(z_0^{-1})$
corresponding to $\partial z_0$.)
\end{cor}

\begin{cor}
\label{cor:ext aut H^1}
The isomorphism of theorem~\ref{thm:ext aut}
induces an isomorphism
\[
  - \star
  \ :\ 
  H^1(Q,Z(K))
  \rTo^{\simeq}
  \Out(KGQ;K)
  \quad
  \ \subseteq\ 
  \Out(G;K)
  .
\]
\end{cor}

\begin{rmk}
In the literature
(e.g.~in~\cite{EilenbergMacLane-GroupCohom-I}~\S3),
the groups
$Z^1$ and $B^1$
of 1-cocycles and 1-coboundaries
are often described
as the groups of
crossed homomorphisms
and
principal homomorphisms
respectively,
and $H^1$ is described as
the quotient of
these two groups.
Our results above
offer an interpretation of
$Z^1(Q,Z(K))$,
$B^1(Q,Z(K))$
and
$H^1(Q,Z(K))$
which is more relevant for
studying
group extensions;
this will be seen later
in sections~\ref{sect:ext aut Theta-compatible}
and~\ref{sect:mod-K outer action classification}.
It seems that
the description of $H^1$
as in corollary~\ref{cor:ext aut H^1}
is usually given
(e.g.~\cite{Eilenberg-TopologicalMethods} end of~\S4)
only for the case
when $K = Z(K)$ is abelian,
though the general case
is not more difficult.
\end{rmk}

\section{Inflation
from $H^1(R,Z(K)^P)$
to $H^1(Q,Z(K))$}

Let
$(KNGQR)
 =
 (\,
 G
 \,,\,
 j
 \,,\,
 \pi
 \,)
$
be an iterated extension
of $(KNP)$ by $(PQR)$,
whose $Q$-main extension $(KGQ)$
has outer action $\theta$.
By definition,
any automorphism of
the iterated extension
$(KNGQR)$
is also an automorphism of
its $Q$-main extension
$(KGQ)$,
so there is
a canonical inclusion homomorphism
\begin{equation}
  \Aut(KNGQR)
  \rInto
  \Aut(KGQ)
  .
  \label{eqn:aut infl}
\end{equation}
Via the canonical isomorphisms of
theorem~\ref{thm:iterext aut}
and
theorem~\ref{thm:ext aut},
one sees that
this inclusion homomorphism
corresponds to
the inflation map
of cocycles;
i.e.~the following diagram
commutes:
\[
  \begin{diagram}
  Z^1(R,Z(K)^P)
  &
  \rInto^{\quad \infl \quad}
  &
  Z^1(Q,Z(K))
  \\
  \dTo^{
    \text{thm.~\ref{thm:iterext aut}}
    \quad
    \wr|
  }_{
    \quad
    - \star
  }
  &
  &
  \dTo^{
    - \star
    \quad
  }_{
    |\wr
    \quad
    \text{thm.~\ref{thm:ext aut}}
  }
  \\
  \Aut(KNGQR)
  &
  \rInto_{\text{\eqref{eqn:aut infl}}}
  &
  \Aut(KGQ)
  .
  \end{diagram}
\]
Passing to
the quotient in cohomology,
we obtain
the corresponding
commutative diagram:
\begin{equation}
  \begin{diagram}
  H^1(R,Z(K)^P)
  &
  \rTo^{\quad \infl \quad}
  &
  H^1(Q,Z(K))
  \\
  \dTo^{
    \text{cor.~\ref{cor:iterext aut H^1}}
    \quad
    \wr|
  }_{
    \quad
    - \star
  }
  &
  {\Bigg.}
  &
  \dTo^{
    - \star
    \quad
  }_{
    |\wr
    \quad
    \text{cor.~\ref{cor:ext aut H^1}}
  }
  \\
  \Out(KNGQR;K)
  &
  \rTo
  &
  \Out(KGQ;K)
  .
  \end{diagram}
  \label{diag:ext aut infl}
\end{equation}
Consequently,
the injectivity of
the inflation homomorphism
\[
  0
  \rTo
  H^1(R,Z(K)^P)
  \rTo^{\quad \infl \quad}
  H^1(Q,Z(K))
\]
in the sequence~\eqref{eqn:long exact seq}
translates as:

\begin{prop}
The inclusion homomorphism~\eqref{eqn:aut infl}
induces
an injective homomorphism
\[
  \Out(KNGQR;K)
  \rInto
  \Out(KGQ;K)
  .
\]
In other words,
an automorphism
$\xi \in \Aut(KNGQR)$
of the iterated extension~$(KNGQR)$
lies in $\conj_G(Z(K))$
if and only if
it lies in $\conj_G(Z(K)^P)$.
\end{prop}

\begin{proof}
We can see this directly.
It is clear that
$\conj_G(Z(K)^P)$
is a subgroup of
$\conj_G(Z(K))$.
Conversely,
suppose
$\xi \in \Aut(KNGQR)$
is of the form
$\conj_G(z_0^{-1})$
for some
$z_0 \in Z(K)$.
For any $p \in P$
and any element
$n \in N$
such that
$\pi_0(n) = p$,
one has
\[
  z_0^{-1}
  \cdot
  \upperleft
  {\theta|_P(p)}
  {z_0}
  \ =\
  z_0^{-1}
  \cdot
  n
  \cdot
  z_0
  \cdot
  n^{-1}
  \ =\
  \upperleft
  {\conj_G(z_0^{-1})}
  {n}
  \cdot
  n^{-1}
  \ =\
  \upperleft{\xi}{n}
  \cdot
  n^{-1}
  \qquad
  \text{in $N$}
  .
\]
Since
$\xi$ acts trivially
on $n \in N$
by hypothesis,
this implies that
$\upperleft
 {\theta|_P(p)}
 {z_0}
 =
 z_0
$,
and the proposition
follows.
\end{proof}

\section{$H^1(P,Z(K))^R$
and the $\Theta$-compatible
automorphisms of
extensions}
\label{sect:ext aut Theta-compatible}

Throughout this section,
we fix
an extension
\quad
$\smash{
 (KNP)
 \ :\ 
 K
 \rInto^{\quad i_0 \quad}
 N
 \rOnto^{\quad \pi_0 \quad}
 P
}$
\quad
of $K$ by $P$
with mod-$K$outer action $\Theta_P$,
and we let
$\Theta : Q \rTo \Out(N;K)$
be a prolongation of $\Theta_P$.

For any
$q \in Q$,
let
$\Sigma(q) \in \Aut_K(N)$ be
a lift of $\Theta(q) \in \Out(N;K)$.
Then
for any automorphism
$\eta \in \Aut(KNP)$,
the automorphism
$\Sigma(q)
 \circ
 \eta
 \circ
 \Sigma(q)^{-1}
$
of $N$
also acts trivially
on $K$ and on $P$,
and so
it lies in $\Aut(KNP)$
as well.
Another lift of $\Theta(q)$
would be of the form
$\Sigma(q) \circ \conj_N(k)$
for some $k \in K$;
but the identity
$\eta^{-1} \circ \conj_N(k) \circ \eta
 =
 \conj_N(\,\upperleft{\eta^{-1}}{k}\,)
 =
 \conj_N(k)
$
shows that
$\conj_N(k)$
commutes with $\eta$,
and so
it follows that
the automorphism
$\Sigma(q)
 \circ
 \eta
 \circ
 \Sigma(q)^{-1}
 \in
 \Aut(KNP)
$
is independent of
the choice of $\Sigma(q)$
as a lift of $\Theta(q)$.
Hence:

\begin{notatn}
\label{notatn:Q action on Aut(KNP)}
The mod-$K$outer action $\Theta$
induces an action of $Q$
on the abelian group
$\Aut(KNP)$:
an element $q \in Q$
sends
$\eta \in \Aut(KNP)$
to the automorphism
$\Sigma(q)
 \circ
 \eta
 \circ
 \Sigma(q)^{-1}
 \in
 \Aut(KNP)
$
for any choice of
a lift
$\Sigma(q) \in \Aut_K(N)$
of $\Theta(q) \in \Out(N;K)$.
\end{notatn}

It is clear that
this action
normalizes the subgroup
$\conj_N(Z(K))$.
Passing to the quotient,
one recovers
the obvious action of $Q$
on the subgroup
$\Out(KNP;K)$ of $\Out(N;K)$
induced by $\Theta$
(given by conjugation
in $\Out(N;K)$).
Since $\Theta$ is
a prolongation of $\Theta_P$,
this action
becomes trivial
when it is
restricted to $P$,
as indicated by
the following:

\begin{lemma}
\label{lemma:Theta_P fixed under Aut(KNP)}
For any
$\eta \in \Aut(KNP)$
and any $p \in P$,
one has
\[
  \etabar^{-1}
  \circ
  \Theta_P(p)
  \circ
  \etabar
  \ =\
  \Theta_P(p)
  \qquad
  \text{in $\Out(N;K)$}
  ,
\]
where $\etabar \in \Out(KNP;K)$ denotes
the image of $\eta$
in $\Out(N;K)$.
\end{lemma}

\begin{proof}
Let $n \in N$ be
any element
such that
$\pi_0(n) = p$
in $P$;
then
$\conj_N(n) \in \Aut_K(N)$
is a lift of
$\Theta_P(p) \in \Out(N;K)$.
If $\lambda \in Z^1(P,Z(K))$
is the 1-cocycle
corresponding to
the automorphism
$\eta \in \Aut(KNP)$,
then
$\upperleft{\eta^{-1}}{n}
 =
 \lambda(p)^{-1}
 \cdot
 n
$,
whence
\[
  \eta^{-1}
  \circ
  \conj_N(n)
  \circ
  \eta
  \ =\
  \conj_N
  \left(
  \upperleft{\eta^{-1}}{n}
  \right)
  \ =\
  \underbrace{
    \conj_N(\lambda(p))^{-1}
  }_{
    \text{in $\conj_N(Z(K))$}
  }
  \ \circ\
  \conj_N(n)
  \qquad
  \text{in $\Aut_K(N)$}
  ,
\]
and the lemma follows.
\end{proof}

\begin{notatn}
\label{notatn:R action on Out(KNP;K)}
The mod-$K$outer action $\Theta$
induces
an action of $R$
on $\Out(KNP;K)$:
an element $r \in R$
sends
$\etabar
 \in
 \Out(KNP;K)
$
to
$\Theta(q)
 \circ
 \etabar
 \circ
 \Theta(q)^{-1}
 \in
 \Out(KNP;K)
$,
where $q \in Q$
is any element
such that
$\phibar(q) = r$
in $R$.
\end{notatn}

\begin{thm}
\label{thm:ext aut Q-equivariant}
Suppose
$(\,
 KNP
 \,,\,
 PQR
 \,,\,
 \Theta
 \,)
$
is an iterated extension problem.
Then
theorem~\ref{thm:ext aut}
applied to
the extension~$(KNP)$
yields
the canonical isomorphism
\[
  \begin{array}[t]{rcl}
  - \star
  \ :\ 
  Z^1(P,Z(K))
  &
  \rTo^{\simeq}
  &
  \Aut(KNP)
  \\
  \lambda
  &
  \rMapsto
  &
  \bigl(\
  n
  \mapsto
  \lambda(\pi_0(n))
  \cdot
  n
  \ \bigr)
  \end{array}
  \qquad
  \text{which is $Q$-equivariant}
  .
\]
Here,
the action of $Q$
on $Z^1(P,Z(K))$
is given by
notation~\ref{notatn:Q action on Z^1},
while
the action of $Q$
on $\Aut(KNP)$
is given by
notation~\ref{notatn:Q action on Aut(KNP)}.
In other words,
for any $q \in Q$
and any choice of
a lift
$\Sigma(q) \in \Aut_K(N)$
of $\Theta(q) \in \Out(N;K)$,
if
$\lambda \in Z^1(P,Z(K))$
and
$\eta \in \Aut(KNP)$
correspond to each other,
then
$\upperleft{q}{\lambda}
 \in
 Z^1(P,Z(K))
$
and
$\Sigma(q)
 \circ
 \eta
 \circ
 \Sigma(q)^{-1}
 \in
 \Aut(KNP)
$
correspond to each other.
\end{thm}

\begin{proof}
Recall that
by the definition~\ref{defn:iterext pbm}
of an iterated extension problem,
the extension $(KNP)$
has outer action
given by $\theta|_P$,
and $\Theta$ is
a $(\theta,\conj_P^Q)$-prolongation of $\Theta_P$.
Therefore,
the automorphism
$\Sigma(q) \in \Aut_K(N)$,
being a lift of
$\Theta(q) \in \Out(N;K)$,
must induce
both
the action of
$\theta(q) \in \Out(K)$
on $Z(K)$
as well as
the conjugation action
$\conj_P^Q(q) \in \Aut(P)$
on $P$.
Suppose
$\lambda \in Z^1(P,Z(K))$
and
$\eta \in \Aut(KNP)$
correspond to each other.
For any $n \in N$,
we then have
\[
  \begin{aligned}
  \upperleft
  {(\ 
     \Sigma(q)
     \ \circ\ 
     \eta
     \ \circ\ 
     \Sigma(q)^{-1}
   \ )
  }
  {n}
  &
  \ =\ 
  \upperleft
  {\Sigma(q)}
  {\Bigl(\
    \lambda
    \bigl(\,
      \pi_0
      (\upperleft
       {\Sigma(q)^{-1}}
       {n}
      )
    \,\bigr)
    \cdot
      (\upperleft
       {\Sigma(q)^{-1}}
       {n}
      )
   \ \Bigr)
  }
  \\
  &
  \ =\ 
  \upperleft
  {\theta(q)}
  {\lambda
   (\,
    q^{-1}\,\pi_0(n)\,q
   \,)
  }
  \cdot
  n
  \ =\ 
  \upperleft{q}{\lambda}
  (\pi_0(n))
  \cdot
  n
  \qquad
  \text{in $N$}
  .
  \end{aligned}
\]
This shows that
$\upperleft{q}{\lambda}
 \in
 Z^1(P,Z(K))
$
and
$\Sigma(q)
 \circ
 \eta
 \circ
 \Sigma(q)^{-1}
 \in
 \Aut(KNP)
$
correspond to each other.
\end{proof}

Restricted to
the subgroup $B^1(P,Z(K))$ of
1-coboundaries,
theorem~\ref{thm:ext aut Q-equivariant},
asserts that
when corollary~\ref{cor:ext aut B^1}
is applied to
the extension~$(KNP)$,
the resulting canonical isomorphism
\[
  \begin{array}[t]{rcl}
  - \star
  \ :\ 
  B^1(P,Z(K))
  &
  \rTo^{\simeq}
  &
  \conj_N(Z(K))
  \\
  \partial z_0
  &
  \rMapsto
  &
  \conj_N(z_0^{-1})
  \end{array}
  \qquad
  \text{is also $Q$-equivariant}
  .
\]
Note that
this is merely
a reformulation of
the fact that
if $z_0 \in Z(K)$
and $q \in Q$,
then
for any choice of
a lift
$\Sigma(q) \in \Aut_K(N)$
of $\Theta(q) \in \Out(N;K)$,
one has
\[
  \Sigma(q)
  \circ
  \conj_N(z_0^{-1})
  \circ
  \Sigma(q)^{-1}
  \ =\
  \conj_N
  \bigl(\,
    \upperleft
    {\Sigma(q)}
    {z_0^{-1}}
  \,\bigr)
  \ =\
  \conj_N
  \bigl(\,
    \upperleft
    {\theta(q)}
    {z_0^{-1}}
  \,\bigr)
  \qquad
  \text{in $\Aut(KNP)$}
  .
\]
Upon passing to
the quotient groups,
we obtain:

\begin{cor}
\label{cor:ext aut R-equivariant H^1}
Suppose
$(\,
 KNP
 \,,\,
 PQR
 \,,\,
 \Theta
 \,)
$
is an iterated extension problem.
Then
corollary~\ref{cor:ext aut H^1}
applied to
the extension~$(KNP)$
yields
the canonical isomorphism
\[
  - \star
  \ :\ 
  H^1(P,Z(K))
  \rTo^{\simeq}
  \Out(KNP;K)
  \qquad
  \text{which is $R$-equivariant}
  .
\]
Here,
the action of $R$
on $H^1(P,Z(K))$
is given by
notation~\ref{notatn:R action on H^1},
while that on $\Out(KNP;K)$
is given by
notation~\ref{notatn:R action on Out(KNP;K)}.
In other words,
for any $r \in R$
and any $q \in Q$
such that
$\phibar(q) = r$
in $R$,
if
$[\lambda] \in H^1(P,Z(K))$
and
$\etabar \in \Out(KNP;K)$
correspond to each other,
then
$\upperleft{r}{[\lambda]}
 \in
 H^1(P,Z(K))
$
and
$\Theta(q)
 \circ
 \etabar
 \circ
 \Theta(q)^{-1}
 \in
 \Out(KNP;K)
$
correspond to each other.
\end{cor}

\begin{defn}
\label{defn:aut Theta-compatible}
An automorphism
$\eta \in \Aut_K(N)$
of $N$
is called
\emph{$\Theta$-compatible}
iff
for any $q \in Q$,
one has
\[
  \Theta(q)
  \circ
  \etabar
  \circ
  \Theta(q)^{-1}
  \ =\
  \etabar
  \qquad
  \text{in $\Out(N;K)$}
  ,
\]
where $\etabar$ denotes
the image of $\eta$
in $\Out(N;K)$;
in other words,
iff
$\etabar \in \Out(N;K)$
commutes with
$\Theta(q)$
for all $q \in Q$.
\end{defn}

If
$(\,
 KNP
 \,,\,
 PQR
 \,,\,
 \Theta
 \,)
$
is an iterated extension problem,
the automorphisms
$\eta \in \Aut(KNP)$
of the extension $(KNP)$
which are $\Theta$-compatible
will be of particular interest
to us;
the significance of
these automorphisms
will be explained later
in section~\ref{sect:iterext twist}.
Thus we introduce
the group
\begin{equation}
  \Aut_\Theta(KNP)
  \ :=\
  \{\
    \eta \in \Aut(KNP)
    \ :\
    \text{$\eta$ is $\Theta$-compatible}
  \ \}
  ,
  \label{eqn:Aut_Theta(KNP)}
\end{equation}
which contains
the subgroup
$\conj_N(Z(K))
 =
 \conj_N(K) \cap \Aut(KNP)
$.
Let
$\Out_\Theta(KNP;K)$ denote
the image of $\Aut_\Theta(KNP)$
in $\Out(N;K)$;
we have
\begin{equation}
  \Out_\Theta(KNP;K)
  \ =\ 
  \dfrac{\Aut_\Theta(KNP)}{\conj_N(Z(K))}
  \ =\ 
  \begin{aligned}[t]
  \{\
    \etabar \in \Out(KNP;K)
    \ :\ 
    &
    \text{for any $q \in Q$,
          one has}
  \\
    &
    \Theta(q) \circ \etabar \circ \Theta(q)^{-1}
    \ =\ 
    \etabar
  \ \}
  .
  \end{aligned}
  \label{eqn:Out_Theta(KNP;K)}
\end{equation}

In the situation of
corollary~\ref{cor:ext aut R-equivariant H^1},
if $\lambda \in Z^1(P,Z(K))$
and
$\eta \in \Aut(KNP)$
correspond to each other,
then $\eta$ is
$\Theta$-compatible
if and only if
the cohomology class
$[\lambda] \in H^1(P,Z(K))$
is fixed
under the action of $R$,
i.e.~if and only if
$[\lambda]
 \in
 H^1(P,Z(K))^R
$.
In other words:

\begin{cor}
\label{cor:ext aut Theta-compatible H^1}
Suppose
$(\,
 KNP
 \,,\,
 PQR
 \,,\,
 \Theta
 \,)
$
is an iterated extension problem.
The canonical isomorphism
(cf.~corollary~\ref{cor:ext aut R-equivariant H^1})
obtained by
applying corollary~\ref{cor:ext aut H^1}
to the extension~$(KNP)$
restricts to
an isomorphism
\[
  - \star
  \ :\ 
  H^1(P,Z(K))^R
  \rTo^{\simeq}
  \Out_\Theta(KNP;K)
  \quad
  \ \subseteq\ 
  \Out(N;K)
  .
\]
\end{cor}

\section{Restriction
from $H^1(Q,Z(K))$
to $H^1(P,Z(K))^R$}

Let
$(\,
 G
 \,,\,
 i
 \,,\,
 \pi
 \,)
$
be an extension
of $K$ by $Q$
with outer action
$\theta$,
and denote
its $P$-subextension
by
$(KNP)
 =
 (\,
 N
 \,,\,
 i_0
 \,,\,
 \pi_0
 \,)
$.
As in definition~\ref{defn:P-subextension},
let $j : N \rInto G$ denote
the canonical inclusion,
so that
$(KNGQR)
 =
 (\,
 G
 \,,\,
 j
 \,,\,
 \pi
 \,)
$
is an iterated extension
of $(KNP)$ by $(PQR)$,
whose
$Q$-main extension
is the given extension
$(KGQ)
 =
 (\,
 G
 \,,\,
 i
 \,,\,
 \pi
 \,)
$.
Any automorphism $\xi$
of the extension~$(KGQ)$
maps $j(N) \subseteq G$
to itself,
because
$\xi$ induces
the trivial automorphism
on $Q$.
The restriction $\xi|_N$
of $\xi$ to $N$
is thus
a well-defined
automorphism of
the $P$-subextension $(KNP)$;
it is characterized by
the property that
$j \circ \xi|_N
 =
 \xi \circ j
$
as homomorphisms
$N \rInto G$.
This gives
a canonical homomorphism
\begin{equation}
  \Aut(KGQ)
  \rTo
  \Aut(KNP)
  ,
  \qquad
  \xi
  \rMapsto
  \xi|_N
  ,
  \label{eqn:aut res}
\end{equation}
whose kernel is
by definition
the automorphism group
$\Aut(KNGQR)$
of the iterated extension
$(KNGQR)$.
Via the canonical isomorphism of
theorem~\ref{thm:ext aut}
applied to
the extensions
$(KGQ)$ and $(KNP)$,
one sees that
this homomorphism
corresponds to
the restriction map
of cocycles;
i.e.~the following diagram
commutes:
\[
  \begin{diagram}
  Z^1(Q,Z(K))
  &
  \rTo^{\quad \res \quad}
  &
  Z^1(P,Z(K))
  \\
  \dTo^{
    \text{thm.~\ref{thm:ext aut}}
    \quad
    \wr|
  }_{
    \quad
    - \star
  }
  &
  &
  \dTo^{
    - \star
    \quad
  }_{
    |\wr
    \quad
    \text{thm.~\ref{thm:ext aut}}
  }
  \\
  \Aut(KGQ)
  &
  \rTo_{\text{\eqref{eqn:aut res}}}
  &
  \Aut(KNP)
  .
  \end{diagram}
\]
Passing to
the quotient in cohomology,
we obtain
a commutative diagram
in which
the restriction homomorphism
maps
$H^1(Q,Z(K))$
into the subgroup
$H^1(P,Z(K))^R
 \subseteq
 H^1(P,Z(K))
$;
this amounts to
the following:

\begin{prop}
\label{prop:aut res}
Let $\Theta : Q \rTo \Out(N;K)$ be
the mod-$K$outer action of $Q$ on $N$
(cf.~definition~\ref{defn:mod-K outer action of iterext})
of the iterated extension $(KNGQR)$.
Then
for any automorphism
$\xi \in \Aut(KGQ)$
of the extension~$(KGQ)$,
its restriction $\xi|_N$ to $N$
is a $\Theta$-compatible automorphism
of the extension $(KNP)$.
In other words,
the canonical homomorphism
in~\eqref{eqn:aut res}
maps $\Aut(KGQ)$
into $\Aut_\Theta(KNP)$,
and
the following diagram
commutes:
\[
  \begin{diagram}
  H^1(Q,Z(K))
  &
  \rTo^{\quad \res \quad}
  &
  H^1(P,Z(K))^R
  \\
  \dTo^{
    \text{cor.~\ref{cor:ext aut H^1}}
    \quad
    \wr|
  }_{
    \quad
    - \star
  }
  &
  {\Bigg.}
  &
  \dTo^{
    - \star
    \quad
  }_{
    |\wr
    \quad
    \text{cor.~\ref{cor:ext aut Theta-compatible H^1}}
  }
  \\
  \Out(KGQ;K)
  &
  \rTo
  &
  \Out_\Theta(KNP;K)
  .
  \end{diagram}
\]
\end{prop}

\begin{proof}
Here is a direct verification of
the $\Theta$-compatibility of $\xi|_N$.
For any $q \in Q$,
we choose a lift
$\Sigma(q) \in \Aut_K(N)$
of
$\Theta(q) \in \Out(N;K)$
and compute
the effect of
the automorphism
$\Sigma(q)
 \circ
 \xi|_N
 \circ
 \Sigma(q)^{-1}
 \in
 \Aut_K(N)
$
on an element $n \in N$.
By definition,
$\Sigma(q)$ acts on $N$
via conjugation by
any element $g \in G$
such that
$\pi(g) = q$ in $Q$.
It follows that
\[
  \begin{aligned}
  j
  \Bigl(\ 
  \upperleft
  {\Sigma(q)
   \,\circ\,
   \xi|_N
   \,\circ\,
   \Sigma(q)^{-1}
  }
  {(\,n\,)}
  \ \Bigr)
  &
  \ =\
  g
  \cdot
  \upperleft
  {\xi}
  {\Bigl(\,
    g^{-1}
    \cdot
    j(n)
    \cdot
    g
   \,\Bigr)
  }
  \cdot
  g^{-1}
  \\
  &
  \ =\
  \bigl(\,
  \underbrace{
    g
    \,
    \upperleft{\xi}{g^{-1}}
  }_{
    z_0^{-1}
  }
  \,\bigr)
  \cdot
  \upperleft
  {\xi}
  {j(n)}
  \cdot
  \bigl(\,
  \underbrace{
    \upperleft{\xi}{g^{}}
    \,
    g^{-1}
  }_{
    z_0
  }
  \,\bigr)
  \ =\
  j
  \Bigl(\ 
  \upperleft
  {\conj_N(z_0^{-1})
   \,\circ\,
   \xi|_N
  }
  {(\,n\,)}
  \ \Bigr)
  \qquad
  \text{in $G$}
  ,
  \end{aligned}
\]
where
$z_0
 :=
 \upperleft{\xi}{g^{}}
 \,
 g^{-1}
$
in
$Z(K)$.
This shows that
\[
  \Sigma(q)
  \circ
  \xi|_N
  \circ
  \Sigma(q)^{-1}
  \ =\
  \conj_N(z_0^{-1})
  \circ
  \xi|_N
  \qquad
  \text{in $\Aut(KNP)$}
  ,
\]
which gives
what we want.
\end{proof}

\begin{rmk}
Suppose
an iterated extension problem
$(\,
 KNP
 \,,\,
 PQR
 \,,\,
 \Theta
 \,)
$
is given in advance,
and
$(\,
 G
 \,,\,
 i
 \,,\,
 \pi
 \,)
$
arises as
an extension of $K$ by $Q$
with outer action $\theta$.
Let $N' := \pi^{-1}(P)$,
and let
$\Theta' : Q \rTo \Out(N';K)$ be
the mod-$K$-outer action
induced by
the conjugation action of $G$
on $N'$.
Proposition~\ref{prop:aut res}
thus applies
and shows that
restriction of
an automorphism $\xi$
of the extension~$(KGQ)$
is a $\Theta'$-compatible automorphism
of its $P$-subextension~$(KN'P)$.
Our given extension $(KNP)$
and
the $P$-subextension $(KN'P)$
of $(KGQ)$
are both
extensions of $K$ by $P$
with the same
outer action $\theta|_P$,
but they need not be
isomorphic extensions.
However,
by theorem~\ref{thm:ext aut},
the automorphism groups
$\Aut(KNP)$ and $\Aut(KN'P)$
of these two extensions
depend only on
the extension problem
$(\,
 K
 \,,\,
 P
 \,,\,
 \theta|_P
 \,)
$;
they are therefore
canonically isomorphic.
By identifying
these two automorphism groups,
we see that
proposition~\ref{prop:aut res}
remains valid as stated,
provided that
we interpret
the notation $\xi|_N$
as referring to
the automorphism of
the extension $(KNP)$
obtained from
the restriction
of $\xi$ to
the $P$-subextension $(KN'P)$
via the canonical isomorphism
between
$\Aut(KNP)$
and
$\Aut(KN'P)$.
\end{rmk}

By the commutative diagram~\ref{diag:ext aut infl}
and proposition~\ref{prop:aut res},
the exactness of
\[
  H^1(R,Z(K)^P)
  \rTo^{\quad \infl \quad}
  H^1(Q,Z(K))
  \rTo^{\quad \res \quad}
  H^1(P,Z(K))^R
\]
in the sequence~\eqref{eqn:long exact seq}
translates as:

\begin{prop}
The subgroup
$\Out(KNGQR;K)
 \subseteq
 \Out(KGQ;K)
$
is the kernel of
the canonical homomorphism
\[
  \Out(KGQ;K)
  \rTo
  \Out_\Theta(KNP;K)
  \qquad
  \text{induced by
        the homomorphism~\eqref{eqn:aut res}}
  .
\]
In other words,
if $\xi \in \Aut(KGQ)$
is an automorphism of
the extension~$(KGQ)$,
its restriction $\xi|_N$ to $N$
lies in $\conj_N(Z(K))$
if and only if
there exists
$z_0 \in Z(K)$
such that
$\conj_G(z_0)
 \circ
 \xi
$
belongs to
$\Aut(KNGQR)
 \subseteq
 \Aut(KGQ)
$.
\end{prop}

\begin{proof}
Suppose
$\xi|_N$ belongs to
$\conj_N(Z(K))$;
then
$\xi|_N
 =
 \conj_N(z_0^{-1})
$
in $\Aut(KNP)$
for some
$z_0 \in Z(K)$,
and hence
the composite automorphism
$\conj_G(z_0) \circ \xi$
of $G$
induces the identity
on $N$ and on $Q$,
which means that
it belongs to
$\Aut(KNGQR)$.
The converse is clear
from the fact that
$\Aut(KNGQR)$
is the kernel of
the homomorphism~\eqref{eqn:aut res}.
\end{proof}

\section{$H^2(R,Z(K)^P)$
and the classification of
iterated extensions}
\label{sect:iterext classification}

Let
$(\,
 KNP
 \,,\,
 PQR
 \,,\,
 \Theta
 \,)
$
be an iterated extension problem
(cf.~definition~\ref{defn:iterext pbm}).
Throughout this section,
we fix the following
choices of:
\[
  \begin{aligned}
  &
  \text{a section}
  &
  \quad
  \ubar
  &
  \ :\ 
  R
  \rDotsto
  Q
  &
  \quad
  &
  \text{of}
  \quad
  Q \rOnto^{\phibar} R
  ,
  \\
  \text{and}
  \quad
  &
  \text{a lifting}
  &
  \quad
  \Delta
  &
  \ :\ 
  R
  \rDotsto
  \Aut_K(N)
  &
  \quad
  &
  \text{of}
  \quad
  \Theta \circ \ubar : R \rDotsto \Out(N;K)
  .
  \end{aligned}
\]
(Recall that
according to
the convention
we have imposed,
sections and liftings
are required to
send
the identity element of
the source group
to
the identity element of
the target group;
thus
$\ubar(1_R) = 1_Q$
and
$\Delta(1_R) = \id_N$.)

\begin{defn}
A \emph{$(\ubar,\Delta)$-sectioned
iterated extension
of $(KNP)$ by $(PQR)$}
consists of
a quadruple
$(\,
 G
 \,,\,
 j
 \,,\,
 \pi
 \,,\,
 u
 \,)
$,
where
the triplet
$(\,
 G
 \,,\,
 j
 \,,\,
 \pi
 \,)
$
is an
iterated extension
of $(KNP)$ by $(PQR)$,
and
$u : R \rDotsto G$
is a section of
$G \rOnto^{\phi} R$
such that
\[
  \begin{aligned}
  \pi \circ u
  &
  \ =\
  \ubar
  &
  \qquad
  &
  \text{as maps}
  \quad
  R \rDotsto Q
  \\
  \text{and}
  \qquad
  \conj_N^G \circ u
  &
  \ =\
  \Delta
  &
  \qquad
  &
  \text{as maps}
  \quad
  R \rDotsto \Aut_K(N)
  .
  \end{aligned}
\]
Two such
sectioned iterated extensions
$(\,
 G_\ell
 \,,\,
 j_\ell
 \,,\,
 \pi_\ell
 \,,\,
 u_\ell
 \,)
$
(for $\ell=1,2$)
are \emph{isomorphic}
iff
there exists
an isomorphism of
iterated extensions
$\varphi
 :
 (\,
 G_1
 \,,\,
 j_1
 \,,\,
 \pi_1
 \,)
 \rTo^{\simeq}
 (\,
 G_2
 \,,\,
 j_2
 \,,\,
 \pi_2
 \,)
$
such that
$\varphi \circ u_1
 =
 u_2
$
as maps
$R \rDotsto G_2$.
\end{defn}

If
$(\,
 G
 \,,\,
 j
 \,,\,
 \pi
 \,,\,
 u
 \,)
$
is a $(\ubar,\Delta)$-sectioned
iterated extension
of $(KNP)$ by $(PQR)$,
its mod-$K$outer action
is necessarily
equal to $\Theta$.
Indeed,
if we choose
a section
$s_0 : P \rDotsto N$
of $N \rOnto^{\pi_0} P$,
we can define
the map
$s : Q \rDotsto G$
in terms of $u$ and $s_0$
by setting,
for any $q \in Q$
written in the form
$q = \jbar(p) \cdot \ubar(r)$
(with $p \in P$ and $r \in R$),
\[
  s(q)
  \ :=\ 
  j(\,s_0(p)\,)
  \cdot
  u(r)
  \qquad
  \text{in $G$}
  .
\]
Then
it is clear that
$s$ is sections of
$G \rOnto^{\pi} Q$.
The conjugation action
of $s(q)$ on $N$
is given by
$\conj_N^G(s(q))
 =
 \conj_N(s_0(p))
 \circ
 \Delta(r)
$
in $\Aut_K(N)$,
whose image in $\Out(N;K)$
is
$\Theta_P(s_0(p))
 \circ
 \Theta(\ubar(r))
 =
 \Theta(q)
$.
Thus
$\conj_N^G \circ s : Q \rDotsto \Aut_K(N)$
is a lifting of $\Theta$,
which shows that
the diagram~\eqref{diag:G Theta}
commutes;
this proves our claim.

Conversely,
any iterated extension
of $(KNP)$ by $(PQR)$
with mod-$K$outer action $\Theta$
can be enriched into
a $(\ubar,\Delta)$-sectioned
iterated extension:

\begin{lemma}
\label{lemma:choice of u}
Let
$(\,
 G
 \,,\,
 j
 \,,\,
 \pi
 \,)
$
be an iterated extension
of $(KNP)$ by $(PQR)$
with mod-$K$outer action
$\Theta$.
There exists
a section
$u : R \rDotsto G$
of $G \rOnto^{\phi} R$
such that
$(\,
 G
 \,,\,
 j
 \,,\,
 \pi
 \,,\,
 u
 \,)
$
is a $(\ubar,\Delta)$-sectioned
iterated extension.
Multiplying $u$ by
any map
$R \rDotsto Z(K)^P$
results in
another such section,
and all such sections
are obtained this way.
\end{lemma}

\begin{proof}
We start with
any section
$u$ of $\phi$.
The commutativity of
diagram~\eqref{diag:KNGQR}
shows that
$\pi \circ u$
is a lifting of
$Q \rOnto^{\phibar} R$,
and hence
$\pi \circ u$
differs multiplicatively
from $\ubar$
by a map
$R \rDotsto P$.
Since $\pi$
sends $N$
surjectively onto $P$,
we can adjust
our choice of $u$
by a map
$R \rDotsto N$
to get
$\pi \circ u = \ubar$.
The commutativity of
the diagram~\eqref{diag:G Theta}
then implies that
$\conj_N^G \circ u$
is a lifting of
$\Theta \circ \ubar$.
Since
$\Delta$ is also
a lifting of
$\Theta \circ \ubar$,
the two maps
$\conj_N^G \circ u$
and $\Delta$
differ multiplicatively
from each other
by a map
$R \rDotsto \conj_N(K)$;
and since
$\conj_N$ sends $K$
surjectively onto $\conj_N(K)$,
we can further adjust
our section $u$
by a map
$R \rDotsto K$
to get
$\conj_N^G \circ u = \Delta$,
which shows
the existence claim.
The remaining assertions
follow from
the observation that
$Z(K)^P$ is precisely
the intersection of
$Z(N) = \Ker(\conj_N)$
with
$K = \Ker(\pi)$
in $N$.
\end{proof}

Let
$(\,
 G
 \,,\,
 j
 \,,\,
 \pi
 \,,\,
 u
 \,)
$
be a fixed
$(\ubar,\Delta)$-sectioned
iterated extension
of $(KNP)$ by $(PQR)$.
For any 1-cocycle
$d \in Z^2(R,Z(K)^P)$,
let
$m_d : G \times G \rDotsto G$
be the map
given by
the product of
$d(\phi(-),\phi(-))$
with the multiplication map
in $G$;
that is,
\[
  m_d(g_1,g_2)
  \ :=\
  d(\phi(g_1),\phi(g_2))
  \cdot
  g_1
  \cdot
  g_2
  \qquad
  \text{in $G$}
  .
\]

\begin{lemma}
\label{lemma:G boxtimes d well-def}
The underlying set of $G$
given with $m_d$
as the multiplication map
is a group;
more precisely,
the map $m_d$
is associative,
has $1_G$ as
the identity element,
and its inversion map
is given by
\[
  v_d
  \ :\
  G
  \rDotsto
  G
  ,
  \qquad
  v_d(g)
  \ :=\
  \upperleft
  {\theta_0(\phi(g))^{-1}}
  {d(\phi(g),\phi(g)^{-1})^{-1}}
  \cdot
  g^{-1}
  .
\]
Moreover,
if we let
$d \boxtimes G$ denote
the resulting group
with $m_d$ as multiplication,
the maps
$j : N \rInto d \boxtimes G$
and
$\pi : d \boxtimes G \rOnto Q$
are homomorphisms,
and the map
$u : R \rDotsto d \boxtimes G$
is a section of
$d \boxtimes G \rOnto^{\phi} R$
such that
$(\,
 d \boxtimes G
 \,,\,
 j
 \,,\,
 \pi
 \,,\,
 u
 \,)
$
is a
$(\ubar,\Delta)$-sectioned
iterated extension
of $(KNP)$ by $(PQR)$.
\end{lemma}

\begin{proof}
For any
$g_1,g_2,g_3 \in G$,
let
\[
  \begin{aligned}
  g_{12}
  &
  \ :=\
  m_d(g_1,g_2)
  \ =\
  d(\phi(g_1),\phi(g_2))
  \cdot
  g_1
  \cdot
  g_2
  \\
  \text{and}
  \qquad
  g_{23}
  &
  \ :=\
  m_d(g_2,g_3)
  \ =\
  d(\phi(g_2),\phi(g_3))
  \cdot
  g_2
  \cdot
  g_3
  \qquad
  \text{in $G$}
  ;
  \end{aligned}
\]
thus
$\phi(g_{12})
 =
 \phi(g_1)\phi(g_2)
$
and
$\phi(g_{23})
 =
 \phi(g_2)\phi(g_3)
$
in $R$.
The cocycle relation
satisfied by $d$
yields
\[
  \begin{aligned}
  m_d(m_d(g_1,g_2),g_3)
  &
  \ =\
  d(\phi(g_{12}),\phi(g_3))
  \cdot
  g_{12}
  \cdot
  g_3
  \\
  &
  \ =\
  d(\phi(g_1)\phi(g_2),\phi(g_3))
  \cdot
  d(\phi(g_1),\phi(g_2))
  \cdot
  g_1
  \cdot
  g_2
  \cdot
  g_3
  \\
  &
  \ =\
  d(\phi(g_1),\phi(g_2)\phi(g_3))
  \cdot
  \underbrace{
    \upperleft
    {\theta_0(\phi(g_1))}
    {d(\phi(g_2),\phi(g_3))}
  }_{
    \ =\
    g_1
    \cdot
    d(\phi(g_2),\phi(g_3))
    \cdot
    g_1^{-1}
  }
  \cdot
  g_1
  \cdot
  g_2
  \cdot
  g_3
  \\
  &
  \ =\
  d(\phi(g_1),\phi(g_{23}))
  \cdot
  g_1
  \cdot
  g_{23}
  \ =\
  m_d(g_1,m_d(g_2,g_3))
  \qquad
  \text{in $G$}
  ,
  \end{aligned}
\]
which shows that
$m_d$ is associative.
The fact that
$d(1_R,r) = d(r,1_R) = 1_{Z(K)^P}$
for any $r \in R$
shows that
$m_d(1_G,g) = m_d(g,1_G) = g$
for any $g \in G$.
Since
$\phi(v_d(g)) = \phi(g)^{-1}$
in $R$,
we have
\[
  m_d(g,v_d(g))
  \ =\
  d(\phi(g),\phi(g)^{-1})
  \cdot
  g
  \cdot
  \underbrace{
    \upperleft
    {\theta_0(\phi(g))^{-1}}
    {d(\phi(g),\phi(g)^{-1})^{-1}}
  }_{
    \ =\
    g^{-1}
    \cdot
    d(\phi(g),\phi(g)^{-1})^{-1}
    \cdot
    g
  }
  \cdot
  g^{-1}
  \ =\
  1_G
  \qquad
  \text{in $G$}
  ;
\]
this together with
the associativity of $m_d$
show that
we also have
$m_d(v_d(g),g) = 1_G$
in $G$.
Thus
the underlying set of $G$
given with $m_d$
as the multiplication map
is a group,
which we denote as
$d \boxtimes G$
from now on.
We continue to use
the dot-product notation for
multiplication in $G$,
but every
$m_d$-multiplication
in $d \boxtimes G$
will be written out
explicitly.

For any $n_1,n_2 \in N$
and any $g_1,g_2 \in d \boxtimes G$,
one has
\[
  \begin{aligned}
  m_d(j(n_1),j(n_2))
  &
  \ =\
  \underbrace{
    d(1_R,1_R)
  }_{
    \ =\
    1_{Z(K)^P}
  }
  \cdot
  j(n_1)
  \cdot
  j(n_2)
  &
  \ =\
  &
  j(n_1n_2)
  &
  \qquad
  &
  \text{in $d \boxtimes G$}
  ,
  \\
  \pi(m_d(g_1,g_2))
  &
  \ =\
  \pi
  \bigl(\,
    \underbrace{
      d(\phi(g_1),\phi(g_2))
    }_{
      \text{in $Z(K)^P$}
    }
    \cdot
    g_1
    \cdot
    g_2
  \,\bigr)
  &
  \ =\
  &
  \pi(g_1)
  \,
  \pi(g_2)
  &
  \qquad
  &
  \text{in $Q$}
  ,
  \\
  \text{and}
  \qquad
  \conj_N^G(m_d(g_1,g_2))
  &
  \ =\
  \conj_N^G
  \bigl(\,
    \underbrace{
      d(\phi(g_1),\phi(g_2))
    }_{
      \text{in $Z(K)^P$}
    }
    \cdot
    g_1
    \cdot
    g_2
  \,\bigr)
  &
  \ =\
  &
  \conj_N^G(g_1)
  \,
  \conj_N^G(g_2)
  &
  \qquad
  &
  \text{in $\Aut_K(N)$}
  .
  \end{aligned}
\]
These identities
show that
\[
  j
  \ :\
  N
  \rTo
  d \boxtimes G
  ,
  \qquad
  \pi
  \ :\
  d \boxtimes G
  \rTo
  Q
  \qquad
  \text{and}
  \qquad
  \conj_N^G
  \ :\
  d \boxtimes G
  \rTo
  \Aut_K(N)
\]
are homomorphisms.
It is then clear that
$(\,
 d \boxtimes G
 \,,\,
 j
 \,,\,
 \pi
 \,)
$
is an iterated extension
of $(KNP)$ by $(PQR)$.
The fact that
$d(1_R,\phi(g))
 =
 d(\phi(g),1_R)
 =
 1_{Z(K)^P}
$
means that
for any $n \in N$
and any $g \in d \boxtimes G$,
one has
\[
  m_d
  \bigl(\,
  j
  \bigl(\,
    \upperleft
    {\conj_N^G(g)}
    {n}
  \,\bigr)
  \,,\,
  g
  \,\bigr)
  \ =\
  j
  \bigl(\,
    \upperleft
    {\conj_N^G(g)}
    {n}
  \,\bigr)
  \cdot
  g
  \ =\
  g
  \cdot
  j(n)
  \ =\
  m_d(g,j(n))
  \qquad
  \text{in $d \boxtimes G$}
  ,
\]
or equivalently,
\[
  j
  \bigl(\,
    \upperleft
    {\conj_N^G(g)}
    {n}
  \,\bigr)
  \ =\
  m_d
  \bigl(\,
    m_d(g,j(n))
    \,,\,
    v_d(g)
  \,\bigr)
  \ =\ 
  j
  \bigl(\,
    \upperleft
    {\conj_N^{d \boxtimes G}(g)}
    {n}
  \,\bigr)
  \qquad
  \text{in $d \boxtimes G$}
  .
\]
This shows that
$\conj_N^G$
is also equal to
the conjugation action
$\conj_N^{d \boxtimes G}$
of $d \boxtimes G$ on $N$.
Therefore,
the map
$u : R \rDotsto d \boxtimes G$
is a section of
$d \boxtimes G \rOnto^{\phi} R$
satisfying
\[
  \begin{aligned}
  \pi \circ u
  &
  \ =\
  \ubar
  &
  \qquad
  &
  \text{as maps}
  \quad
  R \rDotsto Q
  \\
  \text{and}
  \qquad
  \conj_N^{d \boxtimes G} \circ u
  &
  \ =\
  \Delta
  &
  \qquad
  &
  \text{as maps}
  \quad
  R \rDotsto \Aut_K(N)
  ,
  \end{aligned}
\]
whence
$(\,
 d \boxtimes G
 \,,\,
 j
 \,,\,
 \pi
 \,,\,
 u
 \,)
$
is a $(\ubar,\Delta)$-sectioned
iterated extension
of $(KNP)$ by $(PQR)$.
\end{proof}

\begin{thm}
\label{thm:iterext classification}
Let
$(\,
 G
 \,,\,
 j
 \,,\,
 \pi
 \,,\,
 u
 \,)
$
be a $(\ubar,\Delta)$-sectioned
iterated extension
of $(KNP)$ by $(PQR)$.
Then the map
\[
  \begin{array}{rcl}
  - \boxtimes G
  \ :\
  Z^2(R,Z(K)^P)
  &
  \
  \rTo^{\simeq}
  \
  &
  \left\{\
  \begin{array}{l}
  \text{isomorphism classes of}
  \\
  \text{$(\ubar,\Delta)$-sectioned
        iterated extensions}
  \\
  \text{of $(KNP)$ by $(PQR)$}
  \end{array}
  \ \right\}
  ,
  \\[4ex]
  d
  &
  \rMapsto
  &
  \begin{array}[t]{l}
  \text{the isomorphism class of}
  \\
  (\,
  d \boxtimes G
  \,,\,
  j
  \,,\,
  \pi
  \,,\,
  u
  \,)
  \quad
  \text{as defined above}
  \end{array}
  \end{array}
\]
is a well-defined bijection.
\end{thm}

Here,
$Z(K)^P$ is regarded as
an $R$-module
via the action $\theta_0$
as in notation~\ref{notatn:theta_0}.
Note that
$Z^2(R,Z(K)^P)$
depends only on
the given data
$(\,
 K
 \,,\,
 PQR
 \,,\,
 \theta
 \,)
$
as in notation~\ref{notatn:fixed data},
whereas
the set
on the right hand side
is defined
only when
the iterated extension problem
$(\,
 KNP
 \,,\,
 PQR
 \,,\,
 \Theta
 \,)
$
as well as
the choices of
$\ubar$ and $\Delta$
are given;
moreover,
the bijection itself
depends on
the choice of
$(\,
  G
  \,,\,
  j
  \,,\,
  \pi
  \,,\,
  u
 \,)
$
as a
$(\ubar,\Delta)$-sectioned
iterated extension
(assuming that
one exists).

The above lemma
shows that
the map $- \boxtimes G$ in question
is well-defined;
hence
the proof theorem~\ref{thm:iterext classification}
will be accomplished
when we show that
$- \boxtimes G$ is injective
and surjective.
Our work
is facilitated by
the following result,
which gives a criterion
for showing that
two sectioned iterated extensions
are isomorphic.

\begin{lemma}
\label{lemma:isom criterion for sectioned iterext}
For $\ell=1,2$,
let
$(\,
 G_\ell
 \,,\,
 j_\ell
 \,,\,
 \pi_\ell
 \,,\,
 u_\ell
 \,)
$
be a $(\ubar,\Delta)$-sectioned
iterated extension
of $(KNP)$ by $(PQR)$,
and let
$f_\ell : R \times R \rDotsto N$
be the (left) factor set
characterized by
the property that
for any $r_1,r_2 \in R$,
one has
\[
  u_\ell(r_1)
  \cdot
  u_\ell(r_2)
  \ =\
  j_\ell
  (\,
    f_\ell(r_1,r_2)
  \,)
  \cdot
  u_\ell(r_1r_2)
  \qquad
  \text{in $G_\ell$}
  .
\]
Then the two
$(\ubar,\Delta)$-sectioned
iterated extensions
$(\,
 G_\ell
 \,,\,
 j_\ell
 \,,\,
 \pi_\ell
 \,,\,
 u_\ell
 \,)
$
are isomorphic
if and only if
$f_1 = f_2$
as maps
$R \times R \rDotsto N$.
\end{lemma}

\begin{proof}
First, suppose
$\varphi : G_1 \rTo^{\simeq} G_2$
is an isomorphism of
$(\ubar,\Delta)$-sectioned
iterated extensions.
For any $r_1,r_2 \in R$,
applying $\varphi$ to
the identity
\[
  u_1(r_1)
  \cdot
  u_1(r_2)
  \ =\
  j_1
  (\,
    f_1(r_1,r_2)
  \,)
  \cdot
  u_1(r_1r_2)
  \qquad
  \text{in $G_1$}
\]
gives
\[
  u_2(r_1)
  \cdot
  u_2(r_2)
  \ =\
  j_2
  (\,
    f_1(r_1,r_2)
  \,)
  \cdot
  u_2(r_1r_2)
  \qquad
  \text{in $G_2$}
  .
\]
Comparing this
with the identity
\[
  u_2(r_1)
  \cdot
  u_2(r_2)
  \ =\
  j_2
  (\,
    f_2(r_1,r_2)
  \,)
  \cdot
  u_2(r_1r_2)
  \qquad
  \text{in $G_2$}
  ,
\]
we see that
$f_1 = f_2$
as maps
$R \times R \rDotsto N$.

Conversely,
suppose we have
$f_1 = f_2$
as maps
$R \times R \rDotsto N$.
For $\ell=1,2$,
let
$\conj_N^{G_\ell} : G_\ell \rTo \Aut_K(N)$
denote
the conjugation action
of $G_\ell$ on $N$,
characterized by
the property that
for any $g \in G_\ell$
and any $n \in N$,
one has
\[
  j_\ell
  (\,
    \upperleft
    {\conj_N^{G_\ell}(g)}
    {n}
  \,)
  \ =\
  g
  \cdot
  j_\ell(n)
  \cdot
  g^{-1}
  \qquad
  \text{in $G_\ell$}
  .
\]
By assumption,
we have
$\conj_N^{G_\ell} \circ u_\ell = \Delta$
as maps
$R \rDotsto \Aut_K(N)$.
Next,
let $n_\ell : G_\ell \rDotsto N$ be
the projection map
corresponding to
the section $u_\ell$,
characterized by
the property that
for any $g \in G_\ell$,
one has
\[
  g
  \ =\
  j_\ell(n_\ell(g))
  \cdot
  u_\ell(\phi_\ell(g))
  \qquad
  \text{in $G_\ell$}
  .
\]
Define the map
$\varphi
 :
 G_1
 \rTo
 G_2
$
by setting,
for each $g \in G_1$,
\[
  \varphi(g)
  \ :=\
  j_2(n_1(g))
  \cdot
  u_2(\phi_1(g))
  \qquad
  \text{in $G_2$}
  .
\]
Then
for any $g,g' \in G_1$,
\[
  \begin{aligned}
  g
  \,
  g'
  &
  \ =\
  j_1(n_1(g))
  \cdot
  u_1(\phi_1(g))
  \cdot
  j_1(n_1(g'))
  \cdot
  u_1(\phi_1(g'))
  \\
  &
  \ =\
  j_1(n_1(g))
  \cdot
  j_1
  \bigl(\,
  \upperleft
  {(\,
      \conj_N^{G_1}
      \,\circ\,
      u_1
   \,)
   (\phi_1(g))
  }
  {n_1(g')}
  \,\bigr)
  \cdot
  u_1(\phi_1(g))
  \cdot
  u_1(\phi_1(g'))
  \\
  &
  \ =\
  j_1
  \bigl(\,
  \underbrace{
  n_1(g)
  \cdot
  \upperleft
  {\Delta(\phi_1(g))}
  {n_1(g')}
  \cdot
  f_1(\phi_1(g),\phi_1(g'))
  }_{
    \ =\
    n_1(gg')
  }
  \,\bigr)
  \cdot
  u_1
  \bigl(\,
    \phi_1(gg')
  \,\bigr)
  \qquad
  \text{in $G_1$}
  ,
  \end{aligned}
\]
and hence
by definition,
\[
  \varphi(gg')
  \ =\
  j_2
  \bigl(\,
  \underbrace{
  n_1(g)
  \cdot
  \upperleft
  {\Delta(\phi_1(g))}
  {n_1(g')}
  \cdot
  f_1(\phi_1(g),\phi_1(g'))
  }_{
    \ =\
    n_1(gg')
  }
  \,\bigr)
  \cdot
  u_2
  \bigl(\,
    \phi_1(gg')
  \,\bigr)
  \qquad
  \text{in $G_2$}
  .
\]
On the other hand,
\[
  \begin{aligned}
  \varphi(g)
  \,
  \varphi(g')
  &
  \ =\
  j_2(n_1(g))
  \cdot
  u_2(\phi_1(g))
  \cdot
  j_2(n_1(g'))
  \cdot
  u_2(\phi_1(g'))
  \\
  &
  \ =\
  j_2(n_1(g))
  \cdot
  j_2
  \bigl(\,
  \upperleft
  {(\,
      \conj_N^{G_2}
      \,\circ\,
      u_2
   \,)
   (\phi_1(g))
  }
  {n_1(g')}
  \,\bigr)
  \cdot
  u_2(\phi_1(g))
  \cdot
  u_2(\phi_1(g'))
  \\
  &
  \ =\
  j_2
  \bigl(\,
  \underbrace{
  n_1(g)
  \cdot
  \upperleft
  {\Delta(\phi_1(g))}
  {n_1(g')}
  \cdot
  f_2(\phi_1(g),\phi_1(g'))
  }_{
    \ =\
    n_2(gg')
  }
  \,\bigr)
  \cdot
  u_2
  \bigl(\,
    \phi_1(gg')
  \,\bigr)
  \qquad
  \text{in $G_2$}
  .
  \end{aligned}
\]
Comparing
the final expressions
for $\varphi(gg')$
and $\varphi(g)\,\varphi(g')$,
we see that
the assumption
$f_1 = f_2$
implies that
$\varphi(g)\,\varphi(g')
 =
 \varphi(gg')
$,
whence $\varphi$ is
a group homomorphism.
Reversing the roles of
$G_1$ and $G_2$
then yields
a homomorphism
$G_2 \rTo G_1$
which is evidently
the inverse of $\varphi$,
whence $\varphi$ is
a group isomorphism.
The fact that
$u_\ell(1_R) = 1_{G_\ell}$
(for $\ell=1,2$)
means that
if $g \in G_\ell$
is of the form
$g = j_\ell(n)$
for some $n \in N$,
then
$n_\ell(g) = n$ in $N$;
from this
it follows that
$\varphi \circ j_1
 =
 j_2
$
as homomorphisms
$N \rInto G_2$.
The fact that
$\pi_\ell \circ u_\ell = \ubar$
as maps
$R \rDotsto Q$
means that
for any $g \in G_1$,
\[
  \begin{aligned}
  (\pi_2 \circ \varphi)(g)
  &
  \ =\
  \pi_2
  \bigl(\,
    j_2(n_1(g))
    \cdot
    u_2(\phi_1(g))
  \,\bigr)
  \\
  &
  \ =\
  (\jbar \circ \pi_0)
  (n_1(g))
  \cdot
  (\pi_2 \circ u_2)
  (\phi_1(g))
  \\
  &
  \ =\
  (\jbar \circ \pi_0)
  (n_1(g))
  \cdot
  (\pi_1 \circ u_1)
  (\phi_1(g))
  \\
  &
  \ =\
  \pi_1
  \bigl(\,
    j_1(n_1(g))
    \cdot
    u_1(\phi_1(g))
  \,\bigr)
  \ =\
  \pi_1(g)
  \qquad
  \text{in $Q$}
  ;
  \end{aligned}
\]
whence
$\pi_2 \circ \varphi
 =
 \pi_1
$
as homomorphisms
$G_1 \rOnto Q$.
Finally,
an element $g \in G_\ell$
of the form
$g = u_\ell(r)$
for some $r \in R$
gives
$n_\ell(g) = 1_N$ in $N$;
from this
it follows that
$\varphi \circ u_1 = u_2$
as maps
$R \rDotsto G_2$.
Therefore,
$\varphi : G_1 \rTo^{\simeq} G_2$
is an isomorphism of
$(\ubar,\Delta)$-sectioned
iterated extensions.
\end{proof}

\begin{proof}
[Proof of theorem~\ref{thm:iterext classification}]
For $\ell=1,2$,
let $d_\ell \in Z^2(R,Z(K)^P)$ be
a 2-cocycle,
which is mapped by
$- \boxtimes G$
to the
$(\ubar,\Delta)$-sectioned
iterated extension
$(\,
 d_\ell \boxtimes G
 \,,\,
 j
 \,,\,
 \pi
 \,,\,
 u
 \,)
$;
its (left) factor set
$f_\ell : R \times R \rDotsto N$
is then characterized by
the property that
for any $r_1,r_2 \in R$,
one has
\[
  m_{d_\ell}
  \bigl(\,
    u(r_1)
    \,,\,
    u(r_2)
  \,\bigr)
  \ =\
  m_{d_\ell}
  \bigl(\,
    j(\,f_\ell(r_1,r_2)\,)
    \,,\,
    u(r_1r_2)
  \,\bigr)
  \qquad
  \text{in $d_\ell \boxtimes G$}
  ,
\]
which,
since $d_\ell$ is
a normalized cocycle,
means that
\[
  d_\ell(r_1,r_2)
  \cdot
  u(r_1)
  \cdot
  u(r_2)
  \ =\
  j(\,f_\ell(r_1,r_2)\,)
  \cdot
  u(r_1r_2)
  \qquad
  \text{in $G$}
  .
\]
If the two
$(\ubar,\Delta)$-sectioned
iterated extensions
$(\,
 G_\ell
 \,,\,
 j_\ell
 \,,\,
 \pi_\ell
 \,,\,
 u_\ell
 \,)
$
(for $\ell=1,2$)
are isomorphic,
then
$f_1 = f_2$
as maps
$R \times R \rDotsto N$
by lemma~\ref{lemma:isom criterion for sectioned iterext},
from which
it follows that
$d_1 = d_2$
as 2-cocycles
$R \times R \rDotsto Z(K)^P$.
Hence
the map $- \boxtimes G$
is injective.

We now show
the surjectivity of $- \boxtimes G$.
Let
$(\,
 G^*
 \,,\,
 j^*
 \,,\,
 \pi^*
 \,,\,
 u^*
 \,)
$
be any
$(\ubar,\Delta)$-sectioned
iterated extension
of $(KNP)$ by $(PQR)$.
The (left) factor sets
$f : R \times R \rDotsto N$
and
$f^* : R \times R \rDotsto N$
of
$(\,
 G
 \,,\,
 j
 \,,\,
 \pi
 \,,\,
 u
 \,)
$
and
$(\,
 G^*
 \,,\,
 j^*
 \,,\,
 \pi^*
 \,,\,
 u^*
 \,)
$
are characterized by
the property that
for any $r_1,r_2 \in R$,
one has
\begin{align}
  u(r_1)
  \cdot
  u(r_2)
  &
  \ =\
  j
  (\,
    f(r_1,r_2)
  \,)
  \cdot
  u(r_1r_2)
  &
  \qquad
  &
  \text{in $G$}
  \label{eqn:f}
  \\
  \text{and}
  \qquad
  u^*(r_1)
  \cdot
  u^*(r_2)
  &
  \ =\
  j^*
  (\,
    f^*(r_1,r_2)
  \,)
  \cdot
  u^*(r_1r_2)
  &
  \qquad
  &
  \text{in $G^*$}
  .
  \label{eqn:f^*}
\end{align}
Applying the homomorphisms
$\conj_N^G$ and $\conj_N^{G^*}$
to equations~\eqref{eqn:f}
and~\eqref{eqn:f^*}
respectively,
we obtain
\[
  \conj_N^G
  \bigl(\,
  j
  (\,
    f(r_1,r_2)
  \,)
  \,\bigr)
  \ =\
  \Delta(r_1)
  \circ
  \Delta(r_2)
  \circ
  \Delta(r_1r_2)^{-1}
  \ =\
  \conj_N^{G^*}
  \bigl(\,
  j^*
  (\,
    f^*(r_1,r_2)
  \,)
  \,\bigr)
  \qquad
  \text{in $\Aut_K(N)$}
  ,
\]
which shows that
$f^*
 =
 d
 \cdot
 f
$
for some map
$d : R \times R \rDotsto Z(N)$.
On the other hand,
applying the homomorphisms
$\pi$ and $\pi^*$
to equations~\eqref{eqn:f}
and~\eqref{eqn:f^*}
respectively,
we have
\[
  \pi
  \bigl(\,
  j
  (\,
    f(r_1,r_2)
  \,)
  \,\bigr)
  \ =\
  \ubar(r_1)
  \,
  \ubar(r_2)
  \,
  \ubar(r_1r_2)^{-1}
  \ =\
  \pi^*
  \bigl(\,
  j^*
  (\,
    f^*(r_1,r_2)
  \,)
  \,\bigr)
  \qquad
  \text{in $Q$}
  ,
\]
which implies that
$d = f^* \cdot f^{-1}$
takes values
in $Z(N) \cap K = Z(K)^P$.
The associativity of multiplication
in $G$ and $G^*$
shows that
the factor sets
$f$ and $f^*$
satisfy the same
``non-abelian cocycle'' relation:
for any
$r_1,r_2,r_3 \in R$,
one has
\[
  \begin{aligned}
  f(r_1,r_2)
  \,
  f(r_1r_2,r_3)
  &
  \ =\
  \upperleft
  {\Delta(r_1)}
  {f(r_2,r_3)}
  \,
  f(r_1,r_2r_3)
  &
  \qquad
  &
  \\
  \text{and}
  \qquad
  f^*(r_1,r_2)
  \,
  f^*(r_1r_2,r_3)
  &
  \ =\
  \upperleft
  {\Delta(r_1)}
  {f^*(r_2,r_3)}
  \,
  f^*(r_1,r_2r_3)
  &
  \qquad
  &
  \text{in $N$}
  .
  \end{aligned}
\]
These and the fact that
the automorphism
$\Delta(r_1)$ of $N$
induces the automorphism
$\theta_0(r_1)$ of $Z(K)^P$
imply that
\[
  d(r_1,r_2)
  \,
  d(r_1r_2,r_3)
  \ =\
  \upperleft
  {\theta_0(r_1)}
  {d(r_2,r_3)}
  \,
  d(r_1,r_2r_3)
  \qquad
  \text{in $Z(K)^P$}
  ;
\]
thus
$d : R \times R \rDotsto Z(K)^P$
is a 2-cocycle.
The map $- \boxtimes G$
sends $d \in Z^2(R,Z(K)^P)$
to the isomorphism class of
the $(\ubar,\Delta)$-sectioned
iterated extension
$(\,
 d \boxtimes G
 \,,\,
 j
 \,,\,
 \pi
 \,,\,
 u
 \,)
$,
whose corresponding
(left) factor set
$f' : R \times R \rDotsto N$
is characterized by
the property that
for any $r_1,r_2 \in R$,
one has
\[
  m_d
  \bigl(\,
    u(r_1)
    \,,\,
    u(r_2)
  \,\bigr)
  \ =\
  m_d
  \bigl(\,
    j(\,f'(r_1,r_2)\,)
    \,,\,
    u(r_1r_2)
  \,\bigr)
  \qquad
  \text{in $d \boxtimes G$}
  ,
\]
which is to say
\[
  d(r_1,r_2)
  \cdot
  u(r_1)
  \cdot
  u(r_2)
  \ =\
  j(\,f'(r_1,r_2)\,)
  \cdot
  u(r_1r_2)
  \qquad
  \text{in $G$}
  .
\]
Comparing this
with equation~\eqref{eqn:f},
we see that
$f' = d \cdot f = f^*$
as maps
$R \times R \rDotsto N$.
Lemma~\ref{lemma:isom criterion for sectioned iterext}
can now be applied
to show that
$(\,
 d \boxtimes G
 \,,\,
 j
 \,,\,
 \pi
 \,,\,
 u
 \,)
$
and
$(\,
 G^*
 \,,\,
 j^*
 \,,\,
 \pi^*
 \,,\,
 u^*
 \,)
$
are isomorphic
$(\ubar,\Delta)$-sectioned
iterated extensions.
Hence
the map $- \boxtimes G$
is surjective.
\end{proof}

\begin{rmk}
\label{rmk:iterext classification}
The proof shows that
the inverse of
the bijection $- \boxtimes G$
of theorem~\ref{thm:iterext classification}
is given by
\[
  \begin{array}{rcl}
  \left\{\
  \begin{array}{l}
  \text{isomorphism classes of}
  \\
  \text{$(\ubar,\Delta)$-sectioned
        iterated extensions}
  \\
  \text{of $(KNP)$ by $(PQR)$}
  \end{array}
  \ \right\}
  &
  \
  \rTo^{\simeq}
  \
  &
  Z^2(R,Z(K)^P)
  ,
  \\[3ex]
  (\,
  G'
  \,,\,
  j'
  \,,\,
  \pi'
  \,,\,
  u'
  \,)
  &
  \rMapsto
  &
  \left(\
    \begin{aligned}
    (r_1,r_2)
    \ \mapsto\
    &
    f'(u'(r_1),u'(r_2))
    \cdot
    \\
    &
    f(u(r_1),u(r_2))^{-1}
    \end{aligned}
  \ \right)
  ,
  \end{array}
\]
where $f$ and $f'$
are the (left) factor sets of
$(\,
 G
 \,,\,
 j
 \,,\,
 \pi
 \,,\,
 u
 \,)
$
and
$(\,
 G'
 \,,\,
 j'
 \,,\,
 \pi'
 \,,\,
 u'
 \,)
$
respectively.
\end{rmk}

\begin{defn}
Two $(\ubar,\Delta)$-sectioned
iterated extensions
$(\,
 G_\ell
 \,,\,
 j_\ell
 \,,\,
 \pi_\ell
 \,,\,
 u_\ell
 \,)
$
of $(KNP)$ by $(PQR)$
are \emph{equivalent}
iff
the underlying
iterated extensions
$(\,
 G_\ell
 \,,\,
 j_\ell
 \,,\,
 \pi_\ell
 \,)
$
(without the sections)
are isomorphic.
\end{defn}

By lemma~\ref{lemma:choice of u},
it follows that
an equivalence class of
$(\ubar,\Delta)$-sectioned
iterated extensions
of $(KNP)$ by $(PQR)$
is the same as
an isomorphism class of
iterated extensions
of $(KNP)$ by $(PQR)$
with mod-$K$outer action
$\Theta$.

\begin{lemma}
\label{lemma:equivalent iterext}
Let
$(\,
 G
 \,,\,
 j
 \,,\,
 \pi
 \,,\,
 u
 \,)
$
be a $(\ubar,\Delta)$-sectioned
iterated extension
of $(KNP)$ by $(PQR)$.
Multiplying the section $u$ by
any 1-cochain
$z : R \rDotsto Z(K)^P$
results in
another section
$z \cdot u$
such that
$(\,
 G
 \,,\,
 j
 \,,\,
 \pi
 \,,\,
 z \cdot u
 \,)
$
is a $(\ubar,\Delta)$-sectioned
iterated extension
of $(KNP)$ by $(PQR)$,
which is equivalent to
$(\,
 G
 \,,\,
 j
 \,,\,
 \pi
 \,,\,
 u
 \,)
$
by construction.
Conversely,
any
$(\ubar,\Delta)$-sectioned
iterated extension
of $(KNP)$ by $(PQR)$
which is equivalent to
$(\,
 G
 \,,\,
 j
 \,,\,
 \pi
 \,,\,
 u
 \,)
$
is isomorphic to one
obtained this way.
\end{lemma}

\begin{proof}
It is clear that
$z \cdot u : R \rDotsto G$
given by
$(z \cdot u)(r)
 :=
 z(r)\,u(r)
$
is also
a section of
$G \rOnto^{\phi} R$,
and because
$z$ takes values in
$Z(K)^P = Z(N) \cap K$,
we have
\[
  \begin{aligned}
  \pi \circ (z \cdot u)
  &
  \ =\
  \ubar
  &
  \qquad
  &
  \text{as maps}
  \quad
  R \rDotsto Q
  \\
  \text{and}
  \qquad
  \conj_N^G \circ (z \cdot u)
  &
  \ =\
  \Delta
  &
  \qquad
  &
  \text{as maps}
  \quad
  R \rDotsto \Aut_K(N)
  .
  \end{aligned}
\]
This shows that
$(\,
 G
 \,,\,
 j
 \,,\,
 \pi
 \,,\,
 z \cdot u
 \,)
$
is also
a $(\ubar,\Delta)$-sectioned
iterated extension
of $(KNP)$ by $(PQR)$.
By definition,
any $(\ubar,\Delta)$-sectioned
iterated extension
of $(KNP)$ by $(PQR)$
which is
equivalent to
$(\,
 G
 \,,\,
 j
 \,,\,
 \pi
 \,,\,
 u
 \,)
$
must be
isomorphic to
$(\,
 G
 \,,\,
 j
 \,,\,
 \pi
 \,,\,
 u'
 \,)
$
for some section
$u' : R \rDotsto G$
of $G \rOnto^{\phi} R$;
and since
\[
  \begin{aligned}
  \pi \circ u'
  &
  \ =\
  \ubar
  &
  \ =\
  &
  \pi \circ u
  &
  \qquad
  &
  \text{as maps}
  \quad
  R \rDotsto Q
  \\
  \text{and}
  \qquad
  \conj_N^G \circ u'
  &
  \ =\
  \Delta
  &
  \ =\
  &
  \conj_N^G \circ u
  &
  \qquad
  &
  \text{as maps}
  \quad
  R \rDotsto \Aut_K(N)
  ,
  \end{aligned}
\]
it follows that
$u'$ and $u$
differ multiplicatively by
some 1-cochain
$z : R \rDotsto Z(K)^P$.
\end{proof}

\begin{cor}
\label{cor:iterext classification B^2}
The bijection $- \boxtimes G$
of theorem~\ref{thm:iterext classification}
restricts to
a bijection
\[
  \begin{array}{rcl}
  - \boxtimes G
  \ :\
  B^2(R,Z(K)^P)
  &
  \rTo^{\simeq}
  &
  \left\{\
  \begin{array}{l}
  \text{isomorphism classes of}
  \\
  \text{$(\ubar,\Delta)$-sectioned
        iterated extensions}
  \\
  \text{of $(KNP)$ by $(PQR)$}
  \\
  \text{which are equivalent to}
  \quad
  (\,
  G
  \,,\,
  j
  \,,\,
  \pi
  \,,\,
  u
  \,)
  \end{array}
  \ \right\}
  ,
  \\[6ex]
  \partial z
  &
  \rMapsto
  &
  \begin{array}[t]{l}
  \text{the isomorphism class of}
  \\
  (\,
  G
  \,,\,
  j
  \,,\,
  \pi
  \,,\,
  z \cdot u
  \,)
  \quad
  \text{as defined above}
  ,
  \end{array}
  \end{array}
\]
where
$\partial z$ denotes
the 2-coboundary
$\partial z(r_1,r_2)
 :=
 z(r_1)
 \cdot
 \upperleft
 {\theta_0(r_1)}
 {z(r_2)}
 \cdot
 z(r_1r_2)^{-1}
$
for any 1-cochain
$z : R \rDotsto Z(K)^P$.
\end{cor}

\begin{proof}
For any 1-cochain
$z : R \rDotsto Z(K)^P$,
the 2-coboundary
$\partial z \in B^2(R,Z(K)^P)$
is mapped by $- \boxtimes G$
to the isomorphism class of
the $(\ubar,\Delta)$-sectioned
iterated extension
$(\,
 \partial z \boxtimes G
 \,,\,
 j
 \,,\,
 \pi
 \,,\,
 u
 \,)
$,
whose corresponding
(left) factor set
$f' : R \times R \rDotsto N$
is given by
$f' = (\partial z) \cdot f$,
where $f$ is
the (left) factor set of 
$(\,
 G
 \,,\,
 j
 \,,\,
 \pi
 \,,\,
 u
 \,)
$.
On the other hand,
if $f'' : R \times R \rDotsto N$
is the (left) factor set of
the $(\ubar,\Delta)$-sectioned
iterated extension
$(\,
 G
 \,,\,
 j
 \,,\,
 \pi
 \,,\,
 z \cdot u
 \,)
$,
then
for any $r_1,r_2 \in R$,
one has
\[
  z(r_1)\,u(r_1)
  \cdot
  z(r_2)\,u(r_2)
  \ =\
  j(\,f''(r_1,r_2)\,)
  \cdot
  z(r_1r_2)\,u(r_1r_2)
  \qquad
  \text{in $G$}
  .
\]
Since
$z(r_1r_2)
 \in
 Z(K)^P
$
commutes with
$j(\,f''(r_1,r_2)\,)
 \in
 N
$,
this shows that
\[
  \begin{aligned}
  \partial z(r_1,r_2)
  \cdot
  u(r_1)
  \cdot
  u(r_2)
  &
  \ =\
  z(r_1r_2)^{-1}
  \cdot
  z(r_1)
  \cdot
  \upperleft
  {\theta_0(r_1)}
  {z(r_2)}
  \cdot
  u(r_1)
  \cdot
  u(r_2)
  \\
  &
  \ =\
  z(r_1r_2)^{-1}
  \cdot
  z(r_1)
  \cdot
  u(r_1)
  \cdot
  z(r_2)
  \cdot
  u(r_2)
  \\
  &
  \ =\
  j(\,f''(r_1,r_2)\,)
  \cdot
  u(r_1r_2)
  \qquad
  \text{in $G$}
  .
  \end{aligned}
\]
Comparing this
with equation~\eqref{eqn:f},
we see that
$f'' = (\partial z) \cdot f = f'$
as maps
$R \times R \rDotsto N$.
Lemma~\ref{lemma:isom criterion for sectioned iterext}
can now be applied
to show that
$(\,
 \partial z \boxtimes G
 \,,\,
 j
 \,,\,
 \pi
 \,,\,
 u
 \,)
$
and
$(\,
 G
 \,,\,
 j
 \,,\,
 \pi
 \,,\,
 z \cdot u
 \,)
$
are isomorphic
$(\ubar,\Delta)$-sectioned
iterated extensions.
\end{proof}

Lemma~\ref{lemma:choice of u}
and
theorem~\ref{thm:iterext classification}
show that
$Z^2(R,Z(K)^P)$
acts transitively
(from the left)
on the set of
isomorphism classes of
iterated extensions
of $(KNP)$ by $(PQR)$
with mod-$K$outer action $\Theta$,
while
corollary~\ref{cor:iterext classification B^2}
shows that
the stabilizer subgroup
of any given
isomorphism class
is $B^2(R,Z(K)^P)$.
Hence we have:

\begin{cor}
\label{cor:iterext classification H^2}
Let
$(\,
 G
 \,,\,
 j
 \,,\,
 \pi
 \,)
$
be an iterated extension
of $(KNP)$ by $(PQR)$
with mod-$K$outer action $\Theta$.
The bijection $- \boxtimes G$
of theorem~\ref{thm:iterext classification}
induces
a bijection
\[
  \begin{array}{rcl}
  - \boxtimes G
  \ :\
  H^2(R,Z(K)^P)
  &
  \rTo^{\simeq}
  &
  \left\{\
  \begin{array}{l}
  \text{isomorphism classes of}
  \\
  \text{iterated extensions}
  \\
  \text{of $(KNP)$ by $(PQR)$}
  \\
  \text{with mod-$K$ outer action $\Theta$}
  \end{array}
  \ \right\}
  \end{array}
\]
which is independent of
the auxiliary choice of
the pair $(\ubar,\Delta)$.
\end{cor}

\begin{notatn}
\label{notatn:iterext cohom class}
For any pair of
iterated extensions
$(\,
 G
 \,,\,
 j
 \,,\,
 \pi
 \,)
$
and
$(\,
 G'
 \,,\,
 j'
 \,,\,
 \pi'
 \,)
$
of $(KNP)$ by $(PQR)$
with the same
mod-$K$outer action
$\Theta$,
let
$\dfrac{(G',j',\pi')}{(G,j,\pi)}
 \ \in\
 H^2(R,Z(K)^P)
$
denote the unique
cohomology class~$[d]$
such that
$(\,
 G'
 \,,\,
 j'
 \,,\,
 \pi'
 \,)
$
is isomorphic to
$(\,
 d \boxtimes G
 \,,\,
 j
 \,,\,
 \pi
 \,)
$
for any 2-cocycle
$d \in Z^2(R,Z(K)^P)$
belonging to
the cohomology class~$[d]$.
We also write
\[
  (\,
  G'
  \,,\,
  j'
  \,,\,
  \pi'
  \,)
  \ \cong\
  [d]
  \boxtimes
  (\,
  G
  \,,\,
  j
  \,,\,
  \pi
  \,)
  \qquad
  \text{as iterated extensions
        of $(KNP)$ by $(PQR)$}
  .
\]
\end{notatn}

\section{Transgression
from $H^1(P,Z(K))^R$
to $H^2(R,Z(K)^P)$}
\label{sect:iterext twist}

The \emph{transgression homomorphism}
\[
  \tgr
  \ :\
  H^1(P,Z(K))^R
  \rTo
  H^2(R,Z(K)^P)
  ,
  \qquad
  [\lambda]
  \rMapsto\relax
  \tgr[\lambda]
  ,
\]
which appear in
the exact sequence~\eqref{eqn:long exact seq},
arises from
the $E_2$-spectral sequence
for the extension $(PQR)$
with coefficients in $Z(K)$;
let us first recall
its explicit description.

Given
$[\lambda] \in H^1(P,Z(K))^R$,
we choose a 1-cocycle
$\lambda \in Z^1(P,Z(K))$
representing it.
Let $w : Q \rDotsto Z(K)$ be
any 1-cochain
such that
$w|_P = \lambda$
and such that
$\partial w$
factors through
$R \times R$
and takes values in $Z(K)^P$.
Then $w$
defines a 2-cocycle
$d_\lambda
 :
 R \times R
 \rDotsto
 Z(K)^P
$
(for the action $\theta_0$
as in notation~\ref{notatn:theta_0})
characterized by
the property that
for any $q_1,q_2 \in Q$,
one has
\[
  d_\lambda
  \bigl(\,
    \phibar(q_1)
    \,,\,
    \phibar(q_2)
  \,\bigr)
  \ =\
  \partial w
  (\,q_1,q_2\,)
  \ =\
  w(q_1)
  \cdot
  \upperleft
  {\theta(q_1)}
  {w(q_2)}
  \cdot
  w(q_1q_2)^{-1}
  \qquad
  \text{in $Z(K)^P$}
  .
\]
The cohomology class
$[d_\lambda]$
in $H^2(R,Z(K)^P)$
is independent of
the choices of
the 1-cocycle~$\lambda$
and
the 1-cochain~$w$
with the above properties.
The transgression image of
$[\lambda]$
is then defined as
\[
  \tgr[\lambda]
  \ :=\
  [d_\lambda]
  \qquad
  \text{in $H^2(R,Z(K)^P)$}
  .
\]
The existence of
a 1-cochain $w$
with the above properties
can be established
by the following construction.
Choose a section
$\ubar : R \rDotsto Q$
of $Q \rOnto^{\phibar} R$.
The $R$-invariance of
$[\lambda]$
means that
we can choose a map
$z : R \rDotsto Z(K)$
such that
for any $r \in R$
and any $p \in P$,
one has
\begin{equation}
  \upperleft
  {\theta(\ubar(r))}
  {\lambda
   \left(
     \ubar(r)^{-1}\,p\,\ubar(r)
   \right)
  }
  \ =\
  z(r)^{-1}
  \cdot
  \upperleft
  {\theta|_P(p)}
  {z(r)}
  \cdot
  \lambda(p)
  \qquad
  \text{in $Z(K)$}
  .
  \label{eqn:z}
\end{equation}
We define
$w : Q \rDotsto Z(K)$
by setting,
for any $q \in Q$
written in the form
$q = \jbar(p) \cdot \ubar(r)$
(with $p \in P$ and $r \in R$),
\[
  w(q)
  \ :=\
  \lambda(p)
  \cdot
  \upperleft
  {\theta|_P(p)}
  {z(r)}
  \ =\
  z(r)
  \cdot
  \upperleft
  {\theta(\ubar(r))}
  {\lambda
   (\,
     \ubar(r)^{-1}\,p\,\ubar(r)
   \,)
  }
  .
\]
One verifies that
the 1-cochain~$w$
as defined
has the required properties.
The resulting 2-cocycle
$d_\lambda : R \times R \rDotsto Z(K)^P$
is then given by
\begin{equation}
  d_\lambda(r_1,r_2)
  \ =\
  z(r_1)
  \cdot
  \upperleft
  {\theta(\ubar(r_1))}
  {z(r_2)}
  \cdot
  z(r_1r_2)^{-1}
  \cdot
  \upperleft
  {\theta(\ubar(r_1r_2))}
  {\lambda
   \left(\,
     \ubar(r_1r_2)^{-1}
     \,
     \ubar(r_1)\ubar(r_2)
   \,\right)^{-1}
  }
  .
  \label{eqn:d_lambda}
\end{equation}

The transgression homomorphism
can be interpreted
in terms of
the iterated extensions.
The discussion is facilitated
by the following
general result.

\begin{lemma}
\label{lemma:twist of conjugation action}
Let $G$ be a group,
and let $N \rInto^{j} G$ be
the inclusion homomorphism of
a normal subgroup.
Let $\eta \in \Aut(N)$ be
an automorphism of $N$.
The group $G$
acts by conjugation
on $N$
in two ways:
\[
  \conj_N^G
  \ :\
  G
  \rTo
  \Aut(N)
  \qquad
  \text{and}
  \qquad
  \conj_N^{G,\eta}
  \ :\
  G
  \rTo
  \Aut(N)
  ,
\]
characterized by
the property that
for any $g \in G$
and any $n \in N$,
one has
\[
  j
  \bigl(
    \upperleft
    {\conj_N^G(g)}
    {n}
  \bigr)
  \ =\
  g
  \cdot
  j(n)
  \cdot
  g^{-1}
  \qquad
  \text{and}
  \qquad
  (j \circ \eta)
  \bigl(
    \upperleft
    {\conj_N^{G,\eta}(g)}
    {n}
  \bigr)
  \ =\
  g
  \cdot
  (j \circ \eta)(n)
  \cdot
  g^{-1}
  \qquad
  \text{in $G$}
  .
\]
Then
$\conj_N^G$ and $\conj_N^{G,\eta}$
satisfy
the following relation:
for any $g \in G$,
one has
\[
  \conj_N^{G,\eta}(g)
  \ =\
  \eta^{-1}
  \circ
  \conj_N^G(g)
  \circ
  \eta
  \qquad
  \text{in $\Aut(N)$}
  .
\]
\end{lemma}

\begin{proof}
By the characterizing property
of $\conj_N^{G,\eta}$,
we have to show that
for any $g \in G$
and any $n \in N$,
one has
\[
  (j \circ \eta)
  \bigl(
  \upperleft
  {(\eta^{-1}
    \,\circ\,
    \conj_N^G(g)
    \,\circ\,
    \eta
    )}
  {n}
  \bigr)
  \ =\
  g
  \cdot
  (j \circ \eta)(n)
  \cdot
  g^{-1}
  \qquad
  \text{in $G$}
  .
\]
Since $\eta \in \Aut(N)$
is an automorphism,
we may write
$n' = \upperleft{\eta}{n}$
and reduce ourselves
to showing that
for any $g \in G$
and any $n' \in N$
one has
\[
  j
  \bigl(
    \upperleft
    {\conj_N^G(g)}
    {n'}
  \bigr)
  \ =\
  g
  \cdot
  j(n')
  \cdot
  g^{-1}
  \qquad
  \text{in $G$}
  ;
\]
but this holds by
the characterizing property
of $\conj_N^G$.
\end{proof}

Let
$(\,
 KNP
 \,,\,
 PQR
 \,,\,
 \Theta
 \,)
$
be an iterated extension problem
(cf.~definition~\ref{defn:iterext pbm}),
and let
$(\,
 G
 \,,\,
 j
 \,,\,
 \pi
 \,)
$
be an iterated extension
of $(KNP)$ by $(PQR)$
with mod-$K$outer action
$\Theta$.
For any
automorphism $\eta$
of the extension $(KNP)$,
we can pre-compose
the inclusion
$j : N \rInto G$
with
the automorphism $\eta$
to obtain
a ``twisted'' inclusion
\[
  j^\eta
  \ :\
  N
  \rInto
  G
  ,
  \qquad
  \text{given by}
  \quad
  j^\eta
  \ :=\
  j \circ \eta
  .
\]
The inclusion $j^\eta$
has the same image in $G$
as $j$ does,
and the fact that
$\eta$ induces
the trivial automorphism
on $K$ and on $P$
means that
$(\,
 G
 \,,\,
 j^\eta
 \,,\,
 \pi
 \,)
$
is still
an iterated extension
of $(KNP)$ by $(PQR)$.
Its mod-$K$outer action
$\Theta^\etabar$
is defined by
the ``twisted'' conjugation action
$\conj_N^{G,\eta}$ of $G$ on $N$
in the notation
of lemma~\ref{lemma:twist of conjugation action},
so that
the diagram~\eqref{diag:G Theta}
with $\conj_N^G,\Theta$
replaced by
$\conj_N^{G,\eta},\Theta^\etabar$
still commutes
and has exact rows.
It follows from
lemma~\ref{lemma:twist of conjugation action}
that
for any $q \in Q$,
one has
\[
  \Theta^{\etabar}(q)
  \ =\
  \etabar^{-1}
  \circ
  \Theta(q)
  \circ
  \etabar
  \qquad
  \text{in $\Out(N;K)$}
  ,
\]
where $\etabar$ denotes
the image of $\eta$
in $\Out(N;K)$;
in the terminology of
definition~\ref{defn:Aut(KNP)-conj Theta}
to be introduced later,
this says that
the ``twisted'' mod-$K$outer action
is an $\Aut(KNP)$-conjugate of $\Theta$.
Referring back to
definition~\ref{defn:aut Theta-compatible},
we see that
the iterated extension
$(\,
 G
 \,,\,
 j^\eta
 \,,\,
 \pi
 \,)
$
has $\Theta$ as
its mod-$K$outer action
if and only if
the automorphism
$\eta \in \Aut(KNP)$
is $\Theta$-compatible.

The group
$\Aut_\Theta(KNP)$
of $\Theta$-compatible automorphisms
(cf.~\eqref{eqn:Aut_Theta(KNP)})
of the extension $(KNP)$
thus acts
(from the right)
on the set of
isomorphism classes of
iterated extensions
of $(KNP)$ by $(PQR)$
with mod-$K$outer action $\Theta$,
sending
$\eta \in \Aut_\Theta(KNP)$
to the isomorphism class of
$(\,
 G
 \,,\,
 j^\eta
 \,,\,
 \pi
 \,)
$.
Moreover,
if the automorphism
$\eta \in \Aut_\Theta(KNP)$
is of the form
$\eta = \conj_N(z_0^{-1})$
for some $z_0 \in Z(K)$,
the resulting
iterated extension
$(\,
 G
 \,,\,
 j^\eta
 \,,\,
 \pi
 \,)
$
is isomorphic to
$(\,
 G
 \,,\,
 j
 \,,\,
 \pi
 \,)
$
via
$\conj_G(z_0) : G \rTo^{\simeq} G$.
Hence
(cf.~\eqref{eqn:Out_Theta(KNP;K)})
we have a well-defined map
\begin{equation}
  \begin{array}{rcl}
  \Out_\Theta(KNP;K)
  &
  \rTo
  &
  \left\{\
  \begin{array}{l}
  \text{isomorphism classes of}
  \\
  \text{iterated extensions}
  \\
  \text{of $(KNP)$ by $(PQR)$}
  \\
  \text{with mod-$K$ outer action $\Theta$}
  \end{array}
  \ \right\}
  ,
  \\[5ex]
  \etabar
  &
  \rMapsto
  &
  (\,
  G
  \,,\,
  j^\eta
  \,,\,
  \pi
  \,)
  \quad
  \begin{array}[t]{l}
    \text{for any $\eta \in \Aut_\Theta(KNP)$}
    \\
    \text{mapping to $\etabar \in \Out_\Theta(KNP;K)$}
    .
  \end{array}
  \end{array}
  \label{eqn:iterext twist}
\end{equation}

\begin{prop}
\label{prop:aut tgr}
Let $[\lambda] \in H^1(P,Z(K))^R$ be
an $R$-invariant cohomology class,
represented by
the 1-cocycle
$\lambda \in Z^1(P,Z(K))$.
Let
$\eta \in \Aut_\Theta(KNP)$
be the $\Theta$-compatible automorphism
of the extension $(KNP)$
corresponding to $\lambda$.
For any
iterated extension
$(\,
  G
  \,,\,
  j
  \,,\,
  \pi
 \,)
$
of $(KNP)$ by $(PQR)$
with mod-$K$outer action $\Theta$,
consider
the iterated extension
$(\,
  G
  \,,\,
  j^\eta
  \,,\,
  \pi
 \,)
$
obtained by
twisting the inclusion $j$
by $\eta$,
so that
$j^\eta := j \circ \eta$.
Then
\[
  \tgr[\lambda]
  \ =\
  \dfrac{(G,j^\eta,\pi)}{(G,j,\pi)}
  \qquad
  \text{in $H^2(R,Z(K)^P)$}
  .
\]
In other words,
the following diagram
commutes:
\[
  \begin{diagram}
  H^1(P,Z(K))^R
  &
  \rTo^{\quad \tgr \quad}
  &
  H^2(R,Z(K)^P)
  \\
  \dTo^{
    \text{cor.~\ref{cor:ext aut Theta-compatible H^1}}
    \quad
    \wr|
  }_{
    \quad
    - \star
  }
  &
  {\Bigg.}
  &
  \dTo^{
    - \boxtimes G
    \quad
  }_{
    |\wr
    \quad
    \text{cor.~\ref{cor:iterext classification H^2}}
  }
  \\
  \Out_\Theta(KNP;K)
  &
  \rTo_{\qquad \text{\eqref{eqn:iterext twist}} \qquad}
  &
  \mathrlap{
  \left\{\
  \begin{array}{l}
  \text{isomorphism classes of}
  \\
  \text{iterated extensions}
  \\
  \text{of $(KNP)$ by $(PQR)$}
  \\
  \text{with mod-$K$ outer action $\Theta$}
  \end{array}
  \ \right\}
  .
  }
  \qquad\qquad\qquad
  \end{diagram}
  \qquad\qquad\qquad\qquad
\]
\end{prop}

\begin{proof}
Choose a section
$\ubar : R \rDotsto Q$
of $Q \rOnto^{\phibar} R$,
and choose a lifting
$\Delta : R \rDotsto \Aut_K(N)$
of
$\Theta \circ \ubar : R \rDotsto \Out(N;K)$.
Let
$u : R \rDotsto G$
and
$u^\eta : R \rDotsto G$
be sections of
$G \rOnto^{\phi} R$
chosen by
applying lemma~\ref{lemma:choice of u}
to the iterated extensions
$(\,
 G
 \,,\,
 j
 \,,\,
 \pi
 \,)
$
and
$(\,
 G
 \,,\,
 j^\eta
 \,,\,
 \pi
 \,)
$
respectively,
so that
\[
  \begin{aligned}
  \pi
  \circ
  u
  &
  \ =\
  &
  \ubar
  &
  \ =\
  &
  \pi
  \circ
  u^\eta
  \qquad
  &
  \text{as maps}
  \quad
  R
  \rDotsto
  Q
  ,
  \\
  \text{and}
  \qquad
  \conj_N^G
  \circ
  u
  &
  \ =\
  &
  \Delta
  &
  \ =\
  &
  \conj_N^{G,\eta}
  \circ
  u^\eta
  \qquad
  &
  \text{as maps}
  \quad
  R
  \rDotsto
  \Aut_K(N)
  .
  \end{aligned}
\]
We note in passing that
these relations
only determine
the sections $u$ and $u^\eta$
modulo $Z(K)^P$.
Thus
$(\,
 G
 \,,\,
 j
 \,,\,
 \pi
 \,,\,
 u
 \,)
$
and
$(\,
 G
 \,,\,
 j^\eta
 \,,\,
 \pi
 \,,\,
 u^\eta
 \,)
$
are
$(\ubar,\Delta)$-sectioned
iterated extensions
of $(KNP)$ by $(PQR)$.
Their (left) factor sets
$f
 :
 R \times R
 \rDotsto
 N
$
and
$f^\eta
 :
 R \times R
 \rDotsto
 N
$
are characterized by
the property that
for any $r_1,r_2 \in R$,
one has
\begin{align}
  u(r_1)
  \cdot
  u(r_2)
  &
  \ =\
  j
  (\,
    f(r_1,r_2)
  \,)
  \cdot
  u(r_1r_2)
  &
  \qquad
  &
  \label{eqn:j f}
  \\
  \text{and}
  \qquad
  u^\eta(r_1)
  \cdot
  u^\eta(r_2)
  &
  \ =\
  j^\eta
  (\,
    f^\eta(r_1,r_2)
  \,)
  \cdot
  u^\eta(r_1r_2)
  &
  \qquad
  &
  \text{in $G$}
  .
  \label{eqn:j^eta f^eta}
\end{align}
By theorem~\ref{thm:iterext classification}
and remark~\ref{rmk:iterext classification},
there is a
2-cocycle
$d : R \times R \rDotsto Z(K)^P$
such that
\[
  f^\eta
  \ =\
  d
  \cdot
  f
  \qquad
  \text{as maps}
  \quad
  R \times R
  \rDotsto N
  ,
\]
and by notation~\ref{notatn:iterext cohom class},
the cohomology class of $d$
is precisely
$[d]
 =
 \dfrac{(G,j^\eta,\pi)}{(G,j,\pi)}
$
in $H^2(R,Z(K)^P)$.

By definition~\ref{defn:aut Theta-compatible},
the fact that
$\eta$ is $\Theta$-compatible
means precisely that
we can choose
a map
$z : R \rDotsto Z(K)$
such that
for any $r \in R$,
one has
\begin{equation}
  \Delta(r)
  \circ
  \eta
  \circ
  \Delta(r)^{-1}
  \ =\
  \conj_N(z(r)^{-1})
  \circ
  \eta
  \qquad
  \text{in $\Aut(KNP)$}
  .
  \label{eqn:conj by Delta}
\end{equation}
The map $z$
is the same as that
in equation~\eqref{eqn:z},
and so
it can be used
in~\eqref{eqn:d_lambda}
to determine
the transgression image
$\tgr[\lambda]$ of $[\lambda]$.
We claim that
$z$ can be interpreted as
the multiplicative difference
modulo $Z(K)^P$
between
the two sections
$u$ and $u^\eta$,
in the sense that
for any $r \in R$,
one has
\[
  u^\eta(r)
  \ =\
  z(r)
  \cdot
  u(r)
  \qquad
  \text{in $Z(K)$ modulo $Z(K)^P$}
  .
\]
To see this,
we merely have to check that
the map
$r \mapsto z(r) \cdot u(r)$
satisfies
the same relations
which determine $u^\eta$
modulo $Z(K)^P$;
and indeed,
we have
\[
  \begin{aligned}
  \pi
  (\,
    z(r)
    \cdot
    u(r)
  \,)
  &
  \ =\
  \pi(\,u(r)\,)
  &
  \ =\
  &
  \ubar(r)
  &
  \qquad
  &
  \text{in $Q$}
  \\
  \text{and}
  \qquad
  \conj_N^{G,\eta}
  (\,
    z(r)
    \cdot
    u(r)
  \,)
  &
  \ =\
  \eta^{-1}
  \circ
  \conj_N^G
  (\,
    z(r)
    \cdot
    u(r))
  \,)
  \circ
  \eta
  \\
  &
  \ =\
  \eta^{-1}
  \circ
  \conj_N^G(z(r))
  \circ
  \Delta(r)
  \circ
  \eta
  &
  \ =\
  &
  \Delta(r)
  &
  \qquad
  &
  \text{in $\Aut_K(N)$}
  ,
  \end{aligned}
\]
where the last equality
is obtained
by rewriting equation~\eqref{eqn:conj by Delta}.
We note that
in both equations~\eqref{eqn:z}
and~\eqref{eqn:conj by Delta},
the map $z$
is only determined
modulo $Z(K)^P$;
we are free to
multiply it by
any map
$R \rDotsto Z(K)^P$.
Accordingly,
we shall assume that
the map
$z : R \rDotsto Z(K)$
has been chosen
so that
$u^\eta = z \cdot u$
as maps
$R \rDotsto G$.

To evaluate
the 2-cocycle $d$ explicitly,
we shall compute
\begin{equation}
  i(d(r_1,r_2))
  \ =\
  j^\eta
  \Bigl(\
    f^\eta(r_1,r_2)
    \cdot
    f(r_1,r_2)^{-1}
  \ \Bigr)
  \ =\
  j^\eta
  (\,
    f^\eta(r_1,r_2)
  \,)
  \cdot
  j
  (\,
    \upperleft
    {\eta}
    {(\,f(r_1,r_2)\,)}
  \,)^{-1}
  \qquad
  \text{in $G$}
  .
  \label{eqn:d}
\end{equation}
Thanks to
our (justified) assumption that
$u^\eta = z \cdot u$,
we can proceed
to rewrite~\eqref{eqn:j^eta f^eta} as
\begin{equation}
  \begin{aligned}[b]
  j^\eta
  (\,
    f^\eta(r_1,r_2)
  \,)
  &
  \ =\
  z(r_1)
  u(r_1)
  \cdot
  z(r_2)
  u(r_2)
  \cdot
  u(r_1r_2)^{-1}
  z(r_1r_2)^{-1}
  \\
  &
  \ =\
  z(r_1)
  \cdot
  \upperleft
  {\theta(\ubar(r_1))}
  {z(r_2)}
  \cdot
  u(r_1)
  \,
  u(r_2)
  \,
  u(r_1r_2)^{-1}
  \cdot
  z(r_1r_2)^{-1}
  \\
  &
  \ =\ 
  z(r_1)
  \cdot
  \upperleft
  {\theta(\ubar(r_1))}
  {z(r_2)}
  \cdot
  j(f(r_1,r_2))
  \cdot
  z(r_1r_2)^{-1}
  \qquad
  \text{in $G$}
  ,
  \end{aligned}
  \label{eqn:d1}
\end{equation}
where the last equality holds
acoording to~\eqref{eqn:j f}.
Next,
equation~\eqref{eqn:conj by Delta}
and the relation
$\conj_N^G \circ u = \Delta$
gives
\[
  \conj_N^G(u(r_1r_2))
  \circ
  \eta
  \circ
  \conj_N^G(u(r_1r_2))^{-1}
  \ =\
  \conj_N(z(r_1r_2)^{-1})
  \circ
  \eta
  \qquad
  \text{in $\Aut(KNP)$}
  ,
\]
which we now apply
to the element
$f(r_1,r_2) \in N$
to get
\begin{multline*}
  \upperleft
  {(\,
    \conj_N^G(u(r_1r_2))
    \,\circ\,
    \eta
   \,)
  }
  {\bigl(\,
    u(r_1r_2)^{-1}
    \cdot
    j(f(r_1,r_2))
    \cdot
    u(r_1r_2)
   \,\bigr)
  }
  \\
  \ =\ 
  z(r_1r_2)^{-1}
  \cdot
  j
  (\,
    \upperleft
    {\eta}
    {(\,f(r_1,r_2)\,)}
  \,)
  \cdot
  z(r_1r_2)
  \qquad
  \text{in $G$}
  .
\end{multline*}
We expand
(only) the left hand side
using the fact that
the automorphism
$\eta \in \Aut(KNP)$
corresponds to
the cocycle
$\lambda \in Z^1(P,Z(K))$.
Since
$\pi_0(f(r_1,r_2))
 =
 \ubar(r_1)
 \,
 \ubar(r_2)
 \,
 \ubar(r_1r_2)^{-1}
$
in $P$
according to~\eqref{eqn:j f},
we see that
$z(r_1r_2)^{-1}
 \cdot
 j
 (\,
   \upperleft
   {\eta}
   {(\,f(r_1,r_2)\,)}
 \,)
 \cdot
 z(r_1r_2)
$
is equal to
\[
  u(r_1r_2)
  \cdot
  \lambda
  (\,
    \ubar(r_1r_2)^{-1}
    \,
    \ubar(r_1)
    \,
    \ubar(r_2)
  \,)
  \cdot
  \Bigl(\ 
    u(r_1r_2)^{-1}
    \cdot
    j(f(r_1,r_2))
    \cdot
    u(r_1r_2)
  \ \Bigr)
  \cdot
  u(r_1r_2)^{-1}
  \qquad
  \text{in $G$}
  ,
\]
and hence
\begin{multline}
  j
  (\,
    \upperleft
    {\eta}
    {(\,f(r_1,r_2)\,)}
  \,)
  \\
  \ =\ 
  z(r_1r_2)
  \cdot
  \upperleft
  {\theta(\ubar(r_1r_2))}
  {\lambda
   (\,
    \ubar(r_1r_2)^{-1}
    \,
    \ubar(r_1)
    \,
    \ubar(r_2)
   \,)
  }
  \cdot
  j(f(r_1,r_2))
  \cdot
  z(r_1r_2)^{-1}
  \qquad
  \text{in $G$}
  .
  \label{eqn:d2}
\end{multline}
Substituting~\eqref{eqn:d1}
and~\eqref{eqn:d2}
into~\eqref{eqn:d},
we obtain
\[
  d(r_1,r_2)
  \ =\
  z(r_1)
  \cdot
  \upperleft
  {\theta(\ubar(r_1))}
  {z(r_2)}
  \cdot
  \upperleft
  {\theta(\ubar(r_1r_2))}
  {\lambda
   (\,
     \ubar(r_1r_2)^{-1}
     \,
     \ubar(r_1)
     \,
     \ubar(r_2)
   \,)^{-1}
  }
  \cdot
  z(r_1r_2)^{-1}
  \qquad
  \text{in $Z(K)^P$}
  .
\]
This coincides with
the 2-cocycle
$d_\lambda
 :
 R \times R
 \rDotsto
 Z(K)^P
$
obtained in
equation~\eqref{eqn:d_lambda},
which represents
the transgression image
$\tgr[\lambda] \in H^2(R,Z(K)^P)$
of $[\lambda] \in H^1(P,Z(K))^R$.
From this
we conclude that
$\tgr[\lambda]
 =
 [d]
 =
 \dfrac{(G,j^\eta,\pi)}{(G,j,\pi)}
$
in $H^2(R,Z(K)^P)$.
\end{proof}

By propositions~\ref{prop:aut res}
and~\ref{prop:aut tgr},
the exactness of
\[
  H^1(Q,Z(K))
  \rTo^{\quad \res \quad}
  H^1(P,Z(K))^R
  \rTo^{\quad \tgr \quad}
  H^2(R,Z(K)^P)
\]
in the sequence~\eqref{eqn:long exact seq}
translates as:

\begin{prop}
Let
$(\,
  G
  \,,\,
  j
  \,,\,
  \pi
 \,)
$
be an iterated extension
of $(KNP)$ by $(PQR)$
with mod-$K$outer action $\Theta$,
and let $\eta \in \Aut(KNP)$ be
a $\Theta$-compatible automorphism
of the extension $(KNP)$.
Then
the iterated extension
$(\,
  G
  \,,\,
  j^\eta
  \,,\,
  \pi
 \,)
$
is isomorphic to
$(\,
  G
  \,,\,
  j
  \,,\,
  \pi
 \,)
$
if and only if
there exists
an automorphism
$\xi \in \Aut(KGQ)$
of the $Q$-main extension $(KGQ)$
such that
$\eta$ is the restriction
$\xi|_N$ of $\xi$ to $N$.
\end{prop}

\begin{proof}
This can be seen directly
as follows.
If $\xi \in \Aut(KGQ)$ is
any automorphism of
the $Q$-main extension $(KGQ)$,
then
one already has
$\xi \circ \pi = \pi$;
hence
$\xi : G \rTo^{\simeq} G$
is an isomorphism
between
the iterated extensions
$(\,
  G
  \,,\,
  j
  \,,\,
  \pi
 \,)
$
and
$(\,
  G
  \,,\,
  j^\eta
  \,,\,
  \pi
 \,)
$
if and only if
one also has
$\xi \circ j
 =
 j^\eta
 =
 j \circ \eta
$,
which is the case
if and only if
$\xi|_N = \eta$
in $\Aut(KNP)$.
\end{proof}

\section{$H^2(Q,Z(K))$
and the classification of
extensions}

Consider
the extension problem
$(\,
 K
 \,,\,
 Q
 \,,\,
 \theta
 \,)
$
deduced from
our given data
$(\,
 K
 \,,\,
 PQR
 \,,\,
 \theta
 \,)
$
in notation~\ref{notatn:fixed data}.
Throughout this section,
we fix the choice of
\[
  \text{a lifting}
  \quad
  \delta
  \ :\ 
  Q
  \rDotsto
  \Aut(K)
  \quad
  \text{of}
  \quad
  \theta : Q \rTo \Out(K)
  .
\]
The results of
section~\ref{sect:iterext classification}
specialize
to analogous results
for the extension $(KGQ)$
by putting $P = \{1\}$
and hence
$R = Q$,
$\phi = \pi$,
and
$N = K$,
$j = i$.
We state these results
in this section
for later references.

\begin{defn}
A \emph{$\delta$-sectioned extension
of $K$ by $Q$}
is a quadruple
$(\,
 G
 \,,\,
 i
 \,,\,
 \pi
 \,,\,
 s
 \,)
$,
where
the triplet
$(\,
 G
 \,,\,
 i
 \,,\,
 \pi
 \,)
$
is an extension
of $K$ by $Q$,
and
$s : Q \rDotsto G$
is a section of
$G \rOnto^{\pi} Q$
such that
\[
  \conj_K^G \circ s
  \ =\
  \delta
  \qquad
  \text{as maps}
  \quad
  Q \rDotsto \Aut(K)
  .
\]
Two $\delta$-sectioned extensions
$(\,
 G_\ell
 \,,\,
 i_\ell
 \,,\,
 \pi_\ell
 \,,\,
 s_\ell
 \,)
$
(for $\ell=1,2$)
are \emph{isomorphic}
iff
there exists
an isomorphism of extensions
$\varphi
 :
 (\,
 G_1
 \,,\,
 i_1
 \,,\,
 \pi_1
 \,)
 \rTo^{\simeq}
 (\,
 G_2
 \,,\,
 i_2
 \,,\,
 \pi_2
 \,)
$
such that
$\varphi \circ s_1
 =
 s_2
$
as maps
$Q \rDotsto G_2$.
They are \emph{equivalent}
iff
the underlying extensions
$(\,
 G_\ell
 \,,\,
 i_\ell
 \,,\,
 \pi_\ell
 \,)
$
(without the sections)
are isomorphic.
\end{defn}

The outer action of
a $\delta$-sectioned extension
is necessarily
equal to $\theta$;
conversely,
any extension
of $K$ by $Q$
with outer action $\theta$
can be enriched into
a $\delta$-sectioned
extension
(because
$\conj_K^G$ maps $K$
surjectively onto $\Inn(K)$).
Thus,
an equivalence class of
$\delta$-sectioned extensions
of $K$ by $Q$
is the same as
an isomorphism class of
extensions
of $K$ by $Q$
with outer action
$\theta$.

Let
$(\,
 G
 \,,\,
 i
 \,,\,
 \pi
 \,,\,
 s
 \,)
$
be a fixed
$\delta$-sectioned
extension of $K$ by $Q$.
For any 1-cocycle
$e \in Z^2(Q,Z(K))$,
let
$m_e : G \times G \rDotsto G$
be the map
given by
\[
  m_e(g_1,g_2)
  \ :=\
  e(\pi(g_1),\pi(g_2))
  \cdot
  g_1
  \cdot
  g_2
  \qquad
  \text{in $G$}
  .
\]
As in lemma~\ref{lemma:G boxtimes d well-def},
one shows that
the underlying set of $G$
given with $m_e$
as the multiplication map
is a group,
and that
if $e \boxtimes G$ denotes
the resulting group
with $m_e$ as multiplication,
the maps
$i : K \rInto e \boxtimes G$
and
$\pi : e \boxtimes G \rOnto Q$
are homomorphisms,
and the map
$s : R \rDotsto e \boxtimes G$
is a section of
$e \boxtimes G \rOnto^{\pi} R$
making
$(\,
 e \boxtimes G
 \,,\,
 i
 \,,\,
 \pi
 \,,\,
 s
 \,)
$
a $\delta$-sectioned extension
of $K$ by $Q$.
As in lemma~\ref{lemma:equivalent iterext},
multiplying the section $s$ by
any 1-cochain
$z : R \rDotsto Z(K)$
results in
another section
$z \cdot s$
such that
$(\,
 G
 \,,\,
 i
 \,,\,
 \pi
 \,,\,
 z \cdot s
 \,)
$
is a $\delta$-sectioned extension
of $K$ by $Q$,
which is equivalent to
$(\,
 G
 \,,\,
 i
 \,,\,
 \pi
 \,,\,
 s
 \,)
$
by construction;
and conversely,
any
$\delta$-sectioned extension
of $K$ by $Q$
which is equivalent to
$(\,
 G
 \,,\,
 i
 \,,\,
 \pi
 \,,\,
 s
 \,)
$
is isomorphic to one
obtained this way.

\begin{thm}
\label{thm:ext classification}
Let
$(\,
 G
 \,,\,
 i
 \,,\,
 \pi
 \,,\,
 s
 \,)
$
be a $\delta$-sectioned extension
of $K$ by $Q$.
Then the map
\[
  \begin{array}{rcl}
  - \boxtimes G
  \ :\
  Z^2(Q,Z(K))
  &
  \
  \rTo^{\simeq}
  \
  &
  \left\{\
  \begin{array}{l}
  \text{isomorphism classes of}
  \\
  \text{$\delta$-sectioned extensions
        of $K$ by $Q$}
  \end{array}
  \ \right\}
  ,
  \\[2ex]
  e
  &
  \rMapsto
  &
  \begin{array}[t]{l}
  \text{the isomorphism class of}
  \\
  (\,
  e \boxtimes G
  \,,\,
  i
  \,,\,
  \pi
  \,,\,
  s
  \,)
  \quad
  \text{as defined above}
  \end{array}
  \end{array}
\]
is a well-defined bijection,
whose inverse
is given by
\[
  \begin{array}{rcl}
  \left\{\
  \begin{array}{l}
  \text{isomorphism classes of}
  \\
  \text{$\delta$-sectioned extensions
        of $K$ by $Q$}
  \end{array}
  \ \right\}
  &
  \
  \rTo^{\simeq}
  \
  &
  Z^2(Q,Z(K))
  ,
  \\[2ex]
  (\,
  G'
  \,,\,
  i'
  \,,\,
  \pi'
  \,,\,
  s'
  \,)
  &
  \rMapsto
  &
  \left(\
    \begin{aligned}
    (q_1,q_2)
    \ \mapsto\
    &
    h'(s'(q_1),s'(q_2))
    \cdot
    \\
    &
    h(s(q_1),s(q_2))^{-1}
    \end{aligned}
  \ \right)
  ,
  \end{array}
\]
where $h$ and $h'$
are the (left) factor sets of
$(\,
 G
 \,,\,
 i
 \,,\,
 \pi
 \,,\,
 s
 \,)
$
and
$(\,
 G'
 \,,\,
 i'
 \,,\,
 \pi'
 \,,\,
 s'
 \,)
$
respectively.
\end{thm}

\begin{cor}
The bijection $- \boxtimes G$
of theorem~\ref{thm:ext classification}
restricts to
a bijection
\[
  \begin{array}{rcl}
  - \boxtimes G
  \ :\
  B^2(Q,Z(K))
  &
  \rTo^{\simeq}
  &
  \left\{\
  \begin{array}{l}
  \text{isomorphism classes of}
  \\
  \text{$\delta$-sectioned extensions
        of $K$ by $Q$}
  \\
  \text{which are equivalent to}
  \quad
  (\,
  G
  \,,\,
  i
  \,,\,
  \pi
  \,,\,
  s
  \,)
  \end{array}
  \ \right\}
  ,
  \\[4ex]
  \partial z
  &
  \rMapsto
  &
  \begin{array}[t]{l}
  \text{the isomorphism class of}
  \\
  (\,
  G
  \,,\,
  i
  \,,\,
  \pi
  \,,\,
  z \cdot s
  \,)
  \quad
  \text{as defined above}
  ,
  \end{array}
  \end{array}
\]
where
$\partial z$ denotes
the 2-coboundary
$\partial z(q_1,q_2)
 :=
 z(q_1)
 \cdot
 \upperleft
 {\theta(q_1)}
 {z(q_2)}
 \cdot
 z(q_1q_2)^{-1}
$
for any 1-cochain
$z : Q \rDotsto Z(K)$.
\end{cor}

\begin{cor}
\label{cor:ext classification H^2}
Let
$(\,
 G
 \,,\,
 i
 \,,\,
 \pi
 \,)
$
be an extension
of $K$ by $Q$
with outer action $\theta$.
The bijection $- \boxtimes G$
of theorem~\ref{thm:ext classification}
induces
a bijection
\[
  \begin{array}{rcl}
  - \boxtimes G
  \ :\
  H^2(Q,Z(K))
  &
  \rTo^{\simeq}
  &
  \left\{\
  \begin{array}{l}
  \text{isomorphism classes of}
  \\
  \text{extensions of $K$ by $Q$}
  \\
  \text{with outer action $\theta$}
  \end{array}
  \ \right\}
  \end{array}
\]
which is independent of
the auxiliary choice of
the lifting $\delta$.
\end{cor}

\begin{notatn}
\label{notatn:ext cohom class}
For any pair of
extensions
$(\,
 G
 \,,\,
 i
 \,,\,
 \pi
 \,)
$
and
$(\,
 G'
 \,,\,
 i'
 \,,\,
 \pi'
 \,)
$
of $K$ by $Q$
with the same outer action
$\theta$,
let
$\dfrac{(G',i',\pi')}{(G,i,\pi)}
 \ \in\
 H^2(Q,Z(K))
$
denote the unique
cohomology class~$[e]$
such that
$(\,
 G'
 \,,\,
 i'
 \,,\,
 \pi'
 \,)
$
is isomorphic to
$(\,
 e \boxtimes G
 \,,\,
 i
 \,,\,
 \pi
 \,)
$
for any 2-cocycle
$e \in Z^2(Q,Z(K))$
belonging to
the cohomology class~$[e]$.
We also write
\[
  (\,
  G'
  \,,\,
  i'
  \,,\,
  \pi'
  \,)
  \ \cong\
  [e]
  \boxtimes
  (\,
  G
  \,,\,
  i
  \,,\,
  \pi
  \,)
  \qquad
  \text{as extensions of $K$ by $Q$}
  .
\]
\end{notatn}

\begin{rmk}
Corollary~\ref{cor:ext classification H^2}
appears as
theorem~11.1 in~\cite{EilenbergMacLane-GroupCohom-II},
proven there directly
(i.e.~without going through
$Z^1$ and $B^1$)
by means of
a generalization of
the Baer-product construction
for group extensions
(cf.~\cite{EilenbergMacLane-GroupCohom-II}~\S5).
\end{rmk}

\section{Restriction
from $H^2(Q,Z(K))$
to $H^2(P,Z(K))$}

Let
$(\,
 G
 \,,\,
 i
 \,,\,
 \pi
 \,)
$
be an extension
of $K$ by $Q$
with outer action
$\theta$,
and let
$(\,
 N
 \,,\,
 i_0
 \,,\,
 \pi_0
 \,)
$
denote
its $P$-subextension
(cf.~definition~\ref{defn:P-subextension}).
If
$(\,
 G'
 \,,\,
 i'
 \,,\,
 \pi'
 \,)
$
is any extension
of $K$ by $Q$
with the same
outer action $\theta$,
its $P$-subextension
$(\,
 N'
 \,,\,
 i_0'
 \,,\,
 \pi_0'
 \,)
$
has the same
outer action as
$(\,
 N
 \,,\,
 i_0
 \,,\,
 \pi_0
 \,)
$:
they are both given by
the restriction
$\theta|_P$
of $\theta$ to $P$.
Hence
we have
a well-defined map
\begin{equation}
  \begin{array}{rcl}
  \left\{\
  \begin{array}{l}
  \text{isomorphism classes of}
  \\
  \text{extensions of $K$ by $Q$}
  \\
  \text{with outer action $\theta$}
  \end{array}
  \ \right\}
  &
  \rTo
  &
  \left\{\
  \begin{array}{l}
  \text{isomorphism classes of}
  \\
  \text{extensions of $K$ by $P$}
  \\
  \text{with outer action $\theta|_P$}
  \end{array}
  \ \right\}
  ,
  \\[4ex]
  (\,
  G'
  \,,\,
  i'
  \,,\,
  \pi'
  \,)
  &
  \rMapsto
  &
  (\,
  N'
  \,,\,
  i_0'
  \,,\,
  \pi_0'
  \,)
  .
  \end{array}
  \label{eqn:ext res}
\end{equation}
By corollary~\ref{cor:ext classification H^2}
applied to
the extensions
$(KGQ)$ and $(KNP)$,
there exist unique
cohomology classes
\[
  [e]
  \ :=\
  \dfrac{(G',i',\pi')}{(G,i,\pi)}
  \ \in\
  H^2(Q,Z(K))
  \qquad
  \text{and}
  \qquad
  [e_0]
  \ :=\
  \dfrac{(N',i_0',\pi_0')}{(N,i_0,\pi_0)}
  \ \in\
  H^2(P,Z(K))
\]
such that
\[
  \begin{aligned}
  (\,
  G'
  \,,\,
  i'
  \,,\,
  \pi'
  \,)
  &
  \ \cong\
  [e]
  \boxtimes
  (\,
  G
  \,,\,
  i
  \,,\,
  \pi
  \,)
  &
  \qquad
  &
  \text{as extensions of $K$ by $Q$}
  \\
  \text{and}
  \qquad
  (\,
  N'
  \,,\,
  i_0'
  \,,\,
  \pi_0'
  \,)
  &
  \ \cong\
  [e_0]
  \boxtimes
  (\,
  N
  \,,\,
  i_0
  \,,\,
  \pi_0
  \,)
  &
  \qquad
  &
  \text{as extensions of $K$ by $P$}
  .
  \end{aligned}
\]

\begin{prop}
\label{prop:ext res}
The restriction homomorphism
$\res$ in cohomology
maps $[e]$ to $[e_0]$.
In other words,
the following diagram
commutes:
\[
  \begin{diagram}
  H^2(Q,Z(K))
  &
  \rTo^{\res}
  &
  H^2(P,Z(K))
  \\
  \dTo^{
    \text{cor.~\ref{cor:ext classification H^2}}
    \quad
    \wr|
  }_{
    \quad
    - \boxtimes G
  }
  &
  {\Bigg.}
  &
  \dTo^{
    - \boxtimes N
    \quad
  }_{
    |\wr
    \quad
    \text{cor.~\ref{cor:ext classification H^2}}
  }
  \\
  \left\{\
  \begin{array}{l}
  \text{isomorphism classes of}
  \\
  \text{extensions of $K$ by $Q$}
  \\
  \text{with outer action $\theta$}
  \end{array}
  \ \right\}
  &
  \rTo_{\qquad \text{\eqref{eqn:ext res}} \qquad}
  &
  \left\{\
  \begin{array}{l}
  \text{isomorphism classes of}
  \\
  \text{extensions of $K$ by $P$}
  \\
  \text{with outer action $\theta|_P$}
  \end{array}
  \ \right\}
  .
  \end{diagram}
\]
\end{prop}

\begin{proof}
Choose
a lifting
$\delta : Q \rDotsto \Aut(K)$
of
$\theta : Q \rTo \Out(K)$.
Since
$\conj_K^G$ sends $K$
surjectively onto
$\Inn(K)$,
we may choose
sections
$s : Q \rDotsto G$
and
$s' : Q \rDotsto G'$
of
$G \rOnto^{\pi} Q$
and
$G' \rOnto^{\pi'} Q$
respectively,
such that
\[
  \conj_K^G \circ s
  \ =\
  \delta
  \ =\
  \conj_K^G \circ s'
  \qquad
  \text{as maps}
  \quad
  Q \rDotsto \Aut(K)
  ;
\]
Thus
$(\,
 G
 \,,\,
 i
 \,,\,
 \pi
 \,,\,
 s
 \,)
$
and
$(\,
 G'
 \,,\,
 i'
 \,,\,
 \pi'
 \,,\,
 s'
 \,)
$
are $\delta$-sectioned extensions
of $K$ by $Q$.
Their factor sets
$h : Q \times Q \rDotsto K$
and
$h' : Q \times Q \rDotsto K$
are characterized by
the property that
for any $q_1,q_2 \in Q$,
one has
\[
  \begin{aligned}
  s(q_1)
  \cdot
  s(q_2)
  &
  \ =\
  i
  (\,
    h(q_1,q_2)
  \,)
  \cdot
  s(q_1q_2)
  &
  \qquad
  &
  \text{in $G$}
  \\
  \text{and}
  \qquad
  s'(q_1)
  \cdot
  s'(q_2)
  &
  \ =\
  i'
  (\,
    h'(q_1,q_2)
  \,)
  \cdot
  s'(q_1q_2)
  &
  \qquad
  &
  \text{in $G'$}
  .
  \end{aligned}
\]
By theorem~\ref{thm:ext classification},
there is a
2-cocycle
$e : Q \times Q \rDotsto Z(K)$
such that
\[
  h'
  \ =\
  e
  \cdot
  h
  \qquad
  \text{as maps}
  \quad
  Q \times Q \rDotsto K
  ,
\]
and by notation~\ref{notatn:ext cohom class},
the cohomology class of $e$
is precisely
$[e]
 =
 \dfrac{(G',i',\pi')}{(G,i,\pi)}
$
in $H^2(Q,Z(K))$.

Let
$s_0 : P \rDotsto N$
and
$s_0' : P \rDotsto N'$
be the restrictions to $P$
of $s$ and $s'$
respectively,
characterized by
the property that
\[
  j \circ s_0
  \ =\ 
  s \circ \jbar
  \quad
  \text{as maps}
  \quad
  P \rDotsto G
  \qquad
  \text{and}
  \qquad
  j \circ s_0'
  \ =\ 
  s' \circ \jbar
  \quad
  \text{as maps}
  \quad
  P \rDotsto G'
  .
\]
These are
sections of
$N \rOnto^{\pi_0} P$
and
$N' \rOnto^{\pi_0'} P$
respectively,
satisfying
\[
  \conj_K^N \circ s_0
  \ =\
  \delta|_P
  \ =\
  \conj_K^N \circ s_0'
  \qquad
  \text{as maps}
  \quad
  P \rDotsto \Aut(K)
  ,
\]
where
$\delta|_P
 :
 P
 \rInto^{\jbar}
 Q
 \rDotsto^{\delta}
 \Aut(K)
$
denotes
the restriction of $\delta$
to $P$;
it is a lifting
of $\theta|_P$.
Hence
$(\,
 N
 \,,\,
 i_0
 \,,\,
 \pi_0
 \,,\,
 s_0
 \,)
$
and
$(\,
 N'
 \,,\,
 i_0'
 \,,\,
 \pi_0'
 \,,\,
 s_0'
 \,)
$
are
$\delta|_P$-sectioned extensions
of $K$ by $P$.
Evidently,
their factor sets
are given by
the restrictions
$h|_P : P \times P \rDotsto K$
and
$h'|_P: P \times P \rDotsto K$
of $h$ and $h'$
to $P \times P$
respectively.
From the relation
between $h$ and $h'$,
it follows that
$h|_P$ and $h'|_P$
satisfy
\[
  h'|_P
  \ =\
  e|_P
  \cdot
  h|_P
  \qquad
  \text{as maps}
  \quad
  P \times P \rDotsto K
  ,
\]
where $e|_P : P \times P \rDotsto Z(K)$
is the restriction of $e$
to $P \times P$.
By theorem~\ref{thm:ext classification}
and notation~\ref{notatn:ext cohom class}
applied to the extension~$(KNP)$,
it follows that
$e|_P \in Z^2(P,Z(K))$
is a 2-cocycle
belonging to
the cohomology class
$[e_0]
 =
 \dfrac{(N',i_0',\pi_0')}{(N,i_0,\pi_0)}
$
in $H^2(P,Z(K))$.
Hence
\[
  \res(\,[e]\,)
  \ =\
  [e_0]
  \qquad
  \text{in $H^2(P,Z(K))$}
  .
\]
\end{proof}

\section{$H^2_P(Q,Z(K))$
and the classification of
extensions
with a given $P$-subextension}

Let
$(\,
 G
 \,,\,
 i
 \,,\,
 \pi
 \,)
$
and
$(\,
 G'
 \,,\,
 i'
 \,,\,
 \pi'
 \,)
$
be extensions
of $K$ by $Q$
with outer action
$\theta$,
and let
$(\,
 N
 \,,\,
 i_0
 \,,\,
 \pi_0
 \,)
$
and
$(\,
 N'
 \,,\,
 i_0'
 \,,\,
 \pi_0'
 \,)
$
be their $P$-subextensions,
as in the previous section.
Recall that
we have defined
\[
  H^2_P(Q,Z(K))
  \ :=\
  \Ker
  \left(
    H^2(Q,Z(K))
    \rTo^{\res}
    H^2(P,Z(K))
  \right)
  .
\]
Thus,
if
$[e]
 \ :=\
 \dfrac{(G',i',\pi')}{(G,i,\pi)}
 \ \in\
 H^2(Q,Z(K))
$
is the cohomology class
such that
\[
  (\,
  G'
  \,,\,
  i'
  \,,\,
  \pi'
  \,)
  \ \cong\
  [e]
  \boxtimes
  (\,
  G
  \,,\,
  i
  \,,\,
  \pi
  \,)
  \qquad
  \text{as extensions of $K$ by $Q$}
  ,
\]
then by
proposition~\ref{prop:ext res},
the cohomology class $[e]$
belongs to
$H^2_P(Q,Z(K))$
if and only if
\[
  (\,
  N'
  \,,\,
  i_0'
  \,,\,
  \pi_0'
  \,)
  \ \cong\
  (\,
  N
  \,,\,
  i_0
  \,,\,
  \pi_0
  \,)
  \qquad
  \text{as extensions of $K$ by $P$}
  .
\]
Hence:

\begin{thm}
\label{thm:ext classification given restriction}
The bijection of
corollary~\ref{cor:ext classification H^2}
restricts to
a bijection
\[
  \begin{array}{rcl}
  - \boxtimes G
  \ :\
  H^2_P(Q,Z(K))
  &
  \rTo^{\simeq}
  &
  \left\{\
  \begin{array}{l}
  \text{isomorphism classes of}
  \\
  \text{extensions of $K$ by $Q$}
  \\
  \text{with outer action $\theta$}
  \\
  \text{whose $P$-subextension}
  \\
  \text{is isomorphic to}
  \quad
  (\,
  N
  \,,\,
  i_0
  \,,\,
  \pi_0
  \,)
  \end{array}
  \ \right\}
  .
  \end{array}
\]
\end{thm}

\section{Inflation
from $H^2(R,Z(K)^P)$
to $H^2_P(Q,Z(K))$}

Let
$(\,
 G
 \,,\,
 j
 \,,\,
 \pi
 \,)
$
be an iterated extension
of $(KNP)$ by $(PQR)$,
whose $Q$-main extension
$(\,
 G
 \,,\,
 i
 \,,\,
 \pi
 \,)
$
has outer action $\theta$.
By definition,
the $P$-subextension
of
$(\,
 G
 \,,\,
 i
 \,,\,
 \pi
 \,)
$
is
$(KNP)
 =
 (\,
 N
 \,,\,
 i_0
 \,,\,
 \pi_0
 \,)
$.
Let $\Theta$ denote
the mod-$K$outer action of
the iterated extension
$(\,
 G
 \,,\,
 j
 \,,\,
 \pi
 \,)
$.

Now let
$(\,
 G'
 \,,\,
 j'
 \,,\,
 \pi'
 \,)
$
be any iterated extension
of $(KNP)$ by $(PQR)$
with the same
mod-$K$outer action $\Theta$.
Its $Q$-main extension
$(\,
 G'
 \,,\,
 i'
 \,,\,
 \pi'
 \,)
$
is then an extension
of $K$ by $Q$
whose outer action
is also $\theta$;
moreover,
by construction,
the $P$-subextension of
$(\,
 G'
 \,,\,
 i'
 \,,\,
 \pi'
 \,)
$
is equal to
$(KNP)
 =
 (\,
 N
 \,,\,
 i_0
 \,,\,
 \pi_0
 \,)
$
as well.
Hence
we have
a well-defined map
\begin{equation}
  \begin{array}{rcl}
  \left\{\
  \begin{array}{l}
  \text{isomorphism classes of}
  \\
  \text{iterated extensions}
  \\
  \text{of $(KNP)$ by $(PQR)$}
  \\
  \text{with mod-$K$ outer action $\Theta$}
  \end{array}
  \ \right\}
  &
  \rTo
  &
  \left\{\
  \begin{array}{l}
  \text{isomorphism classes of}
  \\
  \text{extensions of $K$ by $Q$}
  \\
  \text{with outer action $\theta$}
  \\
  \text{whose $P$-subextension}
  \\
  \text{is isomorphic to}
  \quad
  (\,
  N
  \,,\,
  i_0
  \,,\,
  \pi_0
  \,)
  \end{array}
  \ \right\}
  ,
  \\[6ex]
  (\,
  G'
  \,,\,
  j'
  \,,\,
  \pi'
  \,)
  &
  \rMapsto
  &
  (\,
  G'
  \,,\,
  i'
  \,,\,
  \pi'
  \,)
  .
  \end{array}
  \label{eqn:ext infl}
\end{equation}
By corollary~\ref{cor:iterext classification H^2}
applied to
the iterated extension
$(KNGQR)$
and
by theorem~\ref{thm:ext classification given restriction}
applied to
the extension
$(KGQ)$,
there exist unique
cohomology classes
\[
  [d]
  \ :=\
  \dfrac{(G',j',\pi')}{(G,j,\pi)}
  \ \in\
  H^2(R,Z(K)^P)
  \qquad
  \text{and}
  \qquad
  [e]
  \ :=\
  \dfrac{(G',i',\pi')}{(G,i,\pi)}
  \ \in\
  H^2_P(Q,Z(K))
\]
such that
\[
  \begin{aligned}
  (\,
  G'
  \,,\,
  j'
  \,,\,
  \pi'
  \,)
  &
  \ \cong\
  [d]
  \boxtimes
  (\,
  G
  \,,\,
  j
  \,,\,
  \pi
  \,)
  &
  \qquad
  &
  \text{as iterated extensions
        of $(KNP)$ by $(PQR)$}
  \\
  \text{and}
  \qquad
  (\,
  G'
  \,,\,
  i'
  \,,\,
  \pi'
  \,)
  &
  \ \cong\
  [e]
  \boxtimes
  (\,
  G
  \,,\,
  i
  \,,\,
  \pi
  \,)
  &
  \qquad
  &
  \text{as extensions of $K$ by $Q$}
  .
  \end{aligned}
\]

\begin{prop}
\label{prop:iterext infl}
The inflation homomorphism
$\infl$ in cohomology
maps $[d]$ to $[e]$.
In other words,
the following diagram
commutes:
\[
  \begin{diagram}
  H^2(R,Z(K)^P)
  &
  \rTo^{\infl}
  &
  H^2_P(Q,Z(K))
  \\
  \dTo^{
    \text{cor.~\ref{cor:iterext classification H^2}}
    \quad
    \wr|
  }_{
    \quad
    - \boxtimes G
  }
  &
  {\Bigg.}
  &
  \dTo^{
    - \boxtimes G
    \quad
  }_{
    |\wr
    \quad
    \text{thm.~\ref{thm:ext classification given restriction}}
  }
  \\
  \left\{\
  \begin{array}{l}
  \text{isomorphism classes of}
  \\
  \text{iterated extensions}
  \\
  \text{of $(KNP)$ by $(PQR)$}
  \\
  \text{with mod-$K$ outer action $\Theta$}
  \end{array}
  \ \right\}
  &
  \rTo_{\qquad \text{\eqref{eqn:ext infl}} \qquad}
  &
  \left\{\
  \begin{array}{l}
  \text{isomorphism classes of}
  \\
  \text{extensions of $K$ by $Q$}
  \\
  \text{with outer action $\theta$}
  \\
  \text{whose $P$-subextension}
  \\
  \text{is isomorphic to}
  \quad
  (\,
  N
  \,,\,
  i_0
  \,,\,
  \pi_0
  \,)
  \end{array}
  \ \right\}
  .
  \qquad\qquad\qquad
  \end{diagram}
\]
\end{prop}

\begin{proof}
Choose a section
$\ubar : R \rDotsto Q$
of $Q \rOnto^{\phibar} R$,
and choose a lifting
$\Delta : R \rDotsto \Aut_K(N)$
of
$\Theta \circ \ubar : R \rDotsto \Out(N;K)$.
Apply lemma~\ref{lemma:choice of u}
to choose sections
$u : R \rDotsto G$
and
$u' : R \rDotsto G'$
of
$G \rOnto^{\phi} R$
and
$G' \rOnto^{\phi'} R$
respectively,
so that
\[
  \begin{aligned}
  \pi
  \circ
  u
  &
  \ =\
  &
  \ubar
  &
  \ =\
  &
  \pi
  \circ
  u'
  \qquad
  &
  \text{as maps}
  \quad
  R
  \rDotsto
  Q
  ,
  \\
  \text{and}
  \qquad
  \conj_N^G
  \circ
  u
  &
  \ =\
  &
  \Delta
  &
  \ =\
  &
  \conj_N^{G'}
  \circ
  u'
  \qquad
  &
  \text{as maps}
  \quad
  R
  \rDotsto
  \Aut_K(N)
  .
  \end{aligned}
\]
Thus
$(\,
 G
 \,,\,
 j
 \,,\,
 \pi
 \,,\,
 u
 \,)
$
and
$(\,
 G'
 \,,\,
 j'
 \,,\,
 \pi'
 \,,\,
 u'
 \,)
$
are
$(\ubar,\Delta)$-sectioned
iterated extensions
of $(KNP)$ by $(PQR)$.
Their (left) factor sets
$f
 :
 R \times R
 \rDotsto
 N
$
and
$f^\eta
 :
 R \times R
 \rDotsto
 N
$
are characterized by
the property that
for any $r_1,r_2 \in R$,
one has
\[
  \begin{aligned}
  u(r_1)
  \cdot
  u(r_2)
  &
  \ =\
  j
  (\,
    f(r_1,r_2)
  \,)
  \cdot
  u(r_1r_2)
  &
  \qquad
  &
  \text{in $G$}
  \\
  \text{and}
  \qquad
  u'(r_1)
  \cdot
  u'(r_2)
  &
  \ =\
  j'
  (\,
    f'(r_1,r_2)
  \,)
  \cdot
  u'(r_1r_2)
  &
  \qquad
  &
  \text{in $G'$}
  .
  \end{aligned}
\]
By theorem~\ref{thm:iterext classification}
and remark~\ref{rmk:iterext classification},
there is a
2-cocycle
$d : R \times R \rDotsto Z(K)^P$
such that
\[
  f'
  \ =\
  d
  \cdot
  f
  \qquad
  \text{as maps}
  \quad
  R \times R
  \rDotsto N
  ,
\]
and by notation~\ref{notatn:iterext cohom class},
the cohomology class of $d$
is precisely
$[d]
 =
 \dfrac{(G',j',\pi')}{(G,j,\pi)}
$
in $H^2(R,Z(K)^P)$.

Next,
choose a section
$s_0 : P \rDotsto N$
of $N \rOnto^{\pi_0} P$,
and consider the following
commutative diagram:
\[
  \begin{diagram}
  &
  &
  P
  &
  &
  \\
  &
  \ldDotsto^{\conj_K^N \,\circ\, s_0}
  &
  \dInline_{\jbar}
  &
  \rdDotsto^{\conj_N \,\circ\, s_0}
  &
  \\
  \Aut(K)
  &
  &
  \HonV
  &
  \lTo^{
    \qquad
    \mathclap{NK}
    \qquad\qquad
  }
  &
  \Aut_K(N)
  \\
  &
  &
  \dTo
  &
  \ruDotsto(2,6)^{\mathclap{
      \qquad\qquad\qquad
      \Delta
  }}
  &
  \dOnto
  \\
  \dOnto
  &
  \qquad
  \qquad
  &
  Q
  &
  &
  \\
  &
  \ldTo^{\theta}
  \qquad
  &
  \dLine^{\phibar}
  \uDotsto_{\ubar}
  &
  \rdTo^{\mathclap{
      \Theta
      \qquad
  }}
  &
  \\
  \Out(K)
  &
  \lTo_{
    \qquad\qquad\qquad
  }
  &
  \HonV
  &
  &
  \Out(N;K)
  \\
  &
  &
  \dOnto
  \uDots
  &
  \ruDotsto_{\Theta \,\circ\, \ubar}
  &
  \\
  &
  &
  R
  &
  &
  &
  &
  \end{diagram}
\]
We note that
$\conj_K^N \circ s_0$
is a lift of
$\theta|_P = \theta \circ \jbar$,
and that
$NK \circ \Delta$
is a lift of
$\theta \circ \ubar$.
Therefore,
the map
$\delta : Q \rDotsto \Aut(K)$
defined by setting,
for any $q \in Q$
written in the form
$q = \jbar(p) \cdot \ubar(r)$
(with $p \in P$ and $r \in R$),
\begin{equation}
  \delta(q)
  \ :=\
  \conj_K^N(\,s_0(p)\,)
  \,\circ\,
  NK(\,\Delta(r)\,)
  \qquad
  \text{in $\Aut(K)$}
  ,
  \label{eqn:delta}
\end{equation}
is a lifting of
the outer action $\theta$
of $Q$ on $K$.

We now
define the maps
$s : Q \rDotsto G$
and
$s' : Q \rDotsto G'$
in terms of $u$, $u'$ and $s_0$
by setting,
for any $q \in Q$
written in the form
$q = \jbar(p) \cdot \ubar(r)$
(with $p \in P$ and $r \in R$),
\[
  s(q)
  \ :=\ 
  j(\,s_0(p)\,)
  \cdot
  u(r)
  \qquad
  \text{in $G$}
  ,
  \qquad
  \text{and}
  \qquad
  s'(q)
  \ :=\ 
  j(\,s_0(p)\,)
  \cdot
  u'(r)
  \qquad
  \text{in $G'$}
  .
\]
Then
it is clear that
$s$ and $s'$
are sections of
$G \rOnto^{\pi} Q$
and
$G' \rOnto^{\pi'} Q$
respectively.
Furthermore,
since
\[
  \begin{aligned}
  \conj_K^G \circ j \circ s_0
  &
  \ =\ 
  \conj_K^N \circ s_0
  &
  \qquad
  &
  \text{as maps}
  \quad
  P \rDotsto \Aut(K)
  ,
  \\
  \text{and}
  \qquad
  \conj_K^G \circ u
  \ =\ 
  \conj_K^{G'} \circ u'
  &
  \ =\ 
  NK \circ \Delta
  &
  \qquad
  &
  \text{as maps}
  \quad
  R \rDotsto \Aut(K)
  ,
  \end{aligned}
\]
it follows from~\eqref{eqn:delta}
that
\[
  \conj_K^G \circ s
  \ =\ 
  \conj_K^{G'} \circ s'
  \ =\
  \delta
  \qquad
  \text{as maps}
  \quad
  Q \rDotsto \Aut(K)
  .
\]
Therefore,
$(\,
 G
 \,,\,
 i
 \,,\,
 \pi
 \,,\,
 s
 \,)
$
and
$(\,
 G'
 \,,\,
 i'
 \,,\,
 \pi'
 \,,\,
 s'
 \,)
$
are
$\delta$-sectioned extensions
of $K$ by $Q$.
Their (left) factor sets
$h : Q \times Q \rDotsto K$
and
$h' : Q \times Q \rDotsto K$
are characterized by
the property that
for any $q_1,q_2 \in Q$,
one has
\[
  \begin{aligned}
  s(q_1)
  \cdot
  s(q_2)
  &
  \ =\
  i
  (\,
    h(q_1,q_2)
  \,)
  \cdot
  s(q_1q_2)
  &
  \qquad
  &
  \text{in $G$}
  \\
  \text{and}
  \qquad
  s'(q_1)
  \cdot
  s'(q_2)
  &
  \ =\
  i'
  (\,
    h'(q_1,q_2)
  \,)
  \cdot
  s'(q_1q_2)
  &
  \qquad
  &
  \text{in $G'$}
  .
  \end{aligned}
\]
By theorem~\ref{thm:ext classification},
there is a
2-cocycle
$e : Q \times Q \rDotsto Z(K)$
such that
\[
  h'
  \ =\
  e
  \cdot
  h
  \qquad
  \text{as maps}
  \quad
  Q \times Q \rDotsto K
  ,
\]
and by notation~\ref{notatn:ext cohom class}
and theorem~\ref{thm:ext classification given restriction},
the cohomology class of $e$
is precisely
$[e]
 =
 \dfrac{(G',i',\pi')}{(G,i,\pi)}
$
in $H^2_P(Q,Z(K))$.

We claim that
the 2-cocycles
$e : Q \times Q \rDotsto Z(K)$
and
$d : R \times R \rDotsto Z(K)^P$
obtained above
satisfy the identity:
for any $q_1,q_2 \in Q$,
one has
\[
  e(q_1,q_2)
  \ =\
  d
  (\,
    \phibar(q_1)
    \,,\,
    \phibar(q_2)
  \,)
  \qquad
  \text{in $Z(K)^P \subseteq Z(K)$}
  ;
\]
in other words,
$e \in Z^2(Q,Z(K))$
is the 2-cocycle
obtained from
$d \in Z^2(R,Z(K)^P)$
by inflation.
Indeed,
let $p_1,p_2,p_{12} \in P$
and $r_1,r_2 \in R$ be
the uniquely determined elements
such that
\[
  q_1
  \ =\
  \jbar(p_1)
  \cdot
  \ubar(r_1)
  ,
  \qquad
  q_2
  \ =\
  \jbar(p_2)
  \cdot
  \ubar(r_2)
  ,
  \qquad
  q_1q_2
  \ =\
  \jbar(p_{12})
  \cdot
  \ubar(r_1r_2)
  \qquad
  \text{in $Q$}
  .
\]
The characterizing equation
for the factor set $h$
then gives
\[
  \bigl(\,
    \underbrace{
    j(\,s_0(p_1)\,)
    \cdot
    u(r_1)
    }_{
      \ =\
      s(q_1)
    }
  \,\bigr)
  \cdot
  \bigl(\,
    \underbrace{
    j(\,s_0(p_2)\,)
    \cdot
    u(r_2)
    }_{
      \ =\
      s(q_2)
    }
  \,\bigr)
  \ =\
  i
  (\,
    h(q_1,q_2)
  \,)
  \cdot
  \bigl(\,
    \underbrace{
    j(\,s_0(p_{12})\,)
    \cdot
    u(r_1r_2)
    }_{
      \ =\
      s(q_1q_2)
    }
  \,\bigr)
  \qquad
  \text{in $G$}
  .
\]
Since we have
$\conj_N^G \circ u = \Delta$
by assumption,
it follows that
\[
  i
  (\,
    h(q_1,q_2)
  \,)
  \ =\
  j
  \bigl(\,
    s_0(p_1)
    \cdot
    \upperleft
    {\Delta(r_1)}
    {s_0(p_2)}
  \,\bigr)
  \cdot
  \underbrace{
    u(r_1)
    u(r_2)
    u(r_1r_2)^{-1}
  }_{
    \ =\
    j(\,f(r_1,r_2)\,)
  }
  \cdot
  j
  \bigl(\,
    s_0(p_{12})
  \,\bigr)^{-1}
  \qquad
  \text{in $G$}
  .
\]
Writing
$s_0(p_1)
 \cdot
 \upperleft
 {\Delta(r_1)}
 {s_0(p_2)}
$
briefly as~$n_0$,
we obtain
\[
  h(q_1,q_2)
  \ =\
  n_0
  \cdot
  f(r_1,r_2)
  \cdot
  s_0(p_{12})^{-1}
  \qquad
  \text{in $N$}
  .
\]
The same argument,
starting from
the characterizing equation
for the factor set $h'$,
shows that
\[
  h'(q_1,q_2)
  \ =\
  n_0
  \cdot
  f'(r_1,r_2)
  \cdot
  s_0(p_{12})^{-1}
  \qquad
  \text{in $N$}
  .
\]
Therefore,
\[
  \begin{aligned}
  e(q_1,q_2)
  &
  \ =\
  h'(q_1,q_2)
  \cdot
  h(q_1,q_2)^{-1}
  \\
  &
  \ =\
  n_0
  \cdot
  f'(r_1,r_2)
  \cdot
  f(r_1,r_2)^{-1}
  \cdot
  n_0^{-1}
  \ =\
  d(r_1,r_2)
  \qquad
  \text{in $Z(K)^P \subseteq Z(K)$}
  ,
  \end{aligned}
\]
proving our claim.
Hence
\[
  \infl(\,[d]\,)
  \ =\
  [e]
  \qquad
  \text{in $H^2(Q,Z(K))$}
  .
\]
\end{proof}

By propositions~\ref{prop:aut tgr}
and~\ref{prop:iterext infl},
the exactness of
\[
  H^1(P,Z(K))^R
  \rTo^{\quad \tgr \quad}
  H^2(R,Z(K)^P)
  \rTo^{\quad \infl \quad}
  H^2_P(Q,Z(K))
\]
in the sequence~\eqref{eqn:long exact seq}
translates as:

\begin{prop}
Let
$(\,
 G
 \,,\,
 j
 \,,\,
 \pi
 \,)
$
and
$(\,
 G'
 \,,\,
 j'
 \,,\,
 \pi'
 \,)
$
be iterated extensions of
$(KNP)$ by $(PQR)$
with the same
mod-$K$outer action $\Theta$.
Then
their $Q$-main extensions
$(\,
 G
 \,,\,
 i
 \,,\,
 \pi
 \,)
$
and
$(\,
 G'
 \,,\,
 i'
 \,,\,
 \pi'
 \,)
$
of $K$ by $Q$
are isomorphic
if and only if
there exists
an automorphism
$\eta \in \Aut(KNP)$
of the extension $(KNP)$
such that
the iterated extensions
$(\,
 G
 \,,\,
 j^\eta
 \,,\,
 \pi
 \,)
$
and
$(\,
 G'
 \,,\,
 j'
 \,,\,
 \pi'
 \,)
$
are isomorphic.
\end{prop}

Note that
the automorphism
$\eta \in \Aut(KNP)$
with the stated property
is necessarily
$\Theta$-compatible
(if it exists).

\begin{proof}
We give a direct argument.
For the ``if'' direction,
an isomorphism
$\varphi : G \rTo^{\simeq} G'$
between
the iterated extensions
$(\,
 G
 \,,\,
 j^\eta
 \,,\,
 \pi
 \,)
$
and
$(\,
 G'
 \,,\,
 j'
 \,,\,
 \pi'
 \,)
$
gives
$\varphi \circ j^\eta = j'$
and hence
by pre-composing with $i_0$,
one has
$\varphi \circ i = i'$,
which implies that
$\varphi$ is also
an isomorphism
between
the $Q$-main extensions
$(\,
 G
 \,,\,
 i
 \,,\,
 \pi
 \,)
$
and
$(\,
 G'
 \,,\,
 i'
 \,,\,
 \pi'
 \,)
$.
For the ``only if'' direction,
an isomorphism
$\varphi : G \rTo^{\simeq} G'$
between
the $Q$-main extensions
$(\,
 G
 \,,\,
 i
 \,,\,
 \pi
 \,)
$
and
$(\,
 G'
 \,,\,
 i'
 \,,\,
 \pi'
 \,)
$
gives
$\pi' \circ \varphi = \pi$,
whence
$\varphi$ must map
$j(N) \subseteq G$
isomorphically onto
$j'(N) \subseteq G'$
and thus induce
an automorphism
$\eta \in \Aut(N)$
such that
$\varphi \circ j
 =
 j' \circ \eta^{-1}
$;
from this
it follows that
$\eta$ lies in $\Aut(KNP)$
necessarily,
and that
$\varphi \circ j^\eta
 =
 j'
$,
which implies that
$\varphi$ is also
an isomorphism between
the iterated extensions
$(\,
 G
 \,,\,
 j^\eta
 \,,\,
 \pi
 \,)
$
and
$(\,
 G'
 \,,\,
 j'
 \,,\,
 \pi'
 \,)
$.
\end{proof}

\section{$H^1(R,H^1(P,Z(K)))$
and the classification of
mod-$K$outer actions}
\label{sect:mod-K outer action classification}

Let
\quad
$\smash{
 (KNP)
 \ :\ 
 K
 \rInto^{\quad i_0 \quad}
 N
 \rOnto^{\quad \pi_0 \quad}
 P
}$
\quad
be an extension
of $K$ by $P$
with outer action
$\theta|_P$,
and let
$\Theta_P : P \rTo \Out(N;K)$
denote
its mod-$K$outer action
(cf.~definition~\ref{defn:mod-K outer action of ext}).
To avoid
a proliferation of notation,
we will use the canonical isomorphism
of corollary~\ref{cor:ext aut H^1}
applied to the extension $(KNP)$
to identify
$H^1(P,Z(K))$
with the subgroup
$\Out(KNP;K)$ of $\Out(N;K)$
throughout this section.

\begin{thm}
\label{thm:mod-K outer action classification}
Let
$\Theta
 :
 Q
 \rTo
 \Out(N;K)
$
be a $(\theta,\conj_P^Q)$-prolongation of $\Theta_P$.
Then the map
\[
  \begin{array}{rcl}
  - \diamond \Theta
  \ :\
  Z^1(R,H^1(P,Z(K)))
  &
  \rTo^{\simeq}
  &
  \Bigl\{\
  \text{$(\theta,\conj_P^Q)$-prolongations of $\Theta_P$}
  \ \Bigr\}
  \\[1ex]
  \Gamma
  &
  \rMapsto
  &
  \text{the map}
  \quad
  \Gamma \diamond \Theta
  \ :=\ 
  \bigl(\
  q
  \mapsto
  \Gamma(\phibar(q))
  \cdot
  \Theta(q)
  \ \bigr)
  ,
  \end{array}
\]
is a well-defined bijection.
\end{thm}

Here,
$H^1(P,Z(K))$ is regarded as
an $R$-module
via the action
described in notation~\ref{notatn:R action on H^1}.
Note that
$Z^1(R,H^1(P,Z(K)))$ depends only on
the given data
$(\,
 K
 \,,\,
 PQR
 \,,\,
 \theta
 \,)
$
as in notation~\ref{notatn:fixed data},
whereas
the set
on the right hand side
is defined
only when
the extension $(KNP)$
(and hence
the mod-$K$outer action $\Theta_P$)
is given;
moreover;
the bijection itself
depends on
the choice of $\Theta$
as a $(\theta,\conj_P^Q)$-prolongation
of $\Theta_P$
(assuming that
one exists).

\begin{proof}
Let
$\Gamma
 \in
 Z^1(R,H^1(P,Z(K)))
$
be a 1-cocycle,
which we regard
as a map
from $R$ to $\Out(KNP;K)$.
Let $\Theta' : Q \rTo \Out(N;K)$ be
the map given by
$\Theta'(q)
 :=
 \Gamma(\phibar(q))
 \cdot
 \Theta(q)
$.
For any $q_1,q_2 \in Q$,
corollary~\ref{cor:ext aut R-equivariant H^1}
shows that
the cohomology class
$\upperleft{\phibar(q_1)}{\Gamma(\phibar(q_2))}$
in $H^1(P,Z(K))$
corresponds to
the element
$\Theta(q_1)
 \cdot
 \Gamma(\phibar(q_2))
 \cdot
 \Theta(q_1)^{-1}
$
in $\Out(KNP;K)$.
Thus
the cocycle relation
satisfied by $\Gamma$
yields
\[
  \begin{aligned}
  \Theta'(q_1q_2)
  &
  \ =\
  \Gamma(\phibar(q_1q_2))
  \cdot
  \Theta(q_1q_2)
  \\
  &
  \ =\
  \Gamma(\phibar(q_1))
  \cdot
  \underbrace{
    \upperleft
    {\phibar(q_1)}
    {\Gamma(\phibar(q_2))}
  }_{\mathclap{
    \ =\
    \Theta(\phibar(q_2)
    \cdot
    \Gamma(\phibar(q_2))
    \cdot
    \Theta(\phibar(q_2))^{-1}
  }}
  \cdot
  \Theta(q_1)
  \cdot
  \Theta(q_2)
  \ =\
  \Theta'(q_1)
  \cdot
  \Theta'(q_2)
  \qquad
  \text{in $\Out(N;K)$}
  ,
  \end{aligned}
\]
which shows that
$\Theta'$ is a homomorphism
from $Q$ to $\Out(N;K)$.
Since
$\Gamma(1_R)
 =
 1_{\Out(N;K)}
$,
it follows that
\[
  \Theta'
  \circ
  \jbar
  \ =\ 
  \Theta
  \circ
  \jbar
  \ =\ 
  \Theta_P
  \qquad
  \text{as maps}
  \quad
  P \rDotsto \Out(N;K)
  .
\]
On the other hand,
the composite homomorphisms
\[
  Q
  \rTo^{\Theta'}
  \Out(N;K)
  \rTo
  \Out(K)
  \qquad
  \text{and}
  \qquad
  Q
  \rTo^{\Theta'}
  \Out(N;K)
  \rTo
  \Aut(P)
\]
are equal to
$\theta$ and $\conj_P^Q$
respectively,
because
$\Gamma$ takes values in
$H^1(P,Z(K))
 =
 \Out(KNP;K)
$,
which is precisely
the kernel of
$\Out(N;K)
 \rTo
 \Out(K) \times \Aut(P)
$.
Thus
$\Theta'$ is
a $(\theta,\conj_P^Q)$-prolongation
of $\Theta_P$.
The map $- \diamond \Theta$
which sends
$\Gamma$ to $\Theta'$
is thus
a well-defined map.
If $- \diamond \Theta$ sends
$\Gamma_1,\Gamma_2 \in Z^1(R,H^1(P,Z(K)))$
to the same image,
then
$\Gamma_1(\phibar(q))
 =
 \Gamma_2(\phibar(q))
$
in $H^1(P,Z(K))$
for every $q \in Q$,
which implies that
$\Gamma_1 = \Gamma_2$;
hence
the map $- \diamond \Theta$
is injective.

We now show
the surjectivity of $- \diamond \Theta$.
Let
$\Theta^* : Q \rTo \Out(N;K)$
be a $(\theta,\conj_P^Q)$-prolongation
of $\Theta_P$.
We choose
any section
$\ubar : R \rDotsto Q$
of $Q \rOnto^{\phibar} R$,
and define the map
\[
  \Gamma
  \ :\
  R
  \rDotsto
  H^1(P,Z(K))
  ,
  \qquad
  \Gamma(r)
  \ :=\
  \Theta^*(\ubar(r))
  \cdot
  \Theta(\ubar(r))^{-1}
  .
\]
By assumption,
$\Theta$ and $\Theta^*$
become equal
when post-composed with
the canonical homomorphism
$\Out(N;K)
 \rTo
 \Out(K) \times \Aut(P)
$;
this shows that
the map $\Gamma$
is indeed well-defined,
taking values in
$H^1(P,Z(K)) = \Out(KNP;K)$.
(It will be seen
eventually that
$\Gamma$ is in fact
independent of the choice of
the section $\ubar$.)
Now let
$\fbar : R \times R \rDotsto P$
be the (right) factor set
corresponding to
the section $\ubar$,
characterized by
the property that
for any $r_1,r_2 \in R$,
one has
\[
  \ubar(r_1)
  \,
  \ubar(r_2)
  \ =\
  \ubar(r_1r_2)
  \cdot
  \jbar
  (\,
    \fbar(r_1,r_2)
  \,)
  \qquad
  \text{in $P$}
  .
\]
Then,
using the fact that
$\Theta$ and $\Theta^*$
are both prolongations
of $\Theta_P$,
we have
\[
  \begin{aligned}
  \Gamma(r_1r_2)
  &
  \ =\
  \Theta^*(\ubar(r_1r_2))
  \circ
  \Theta(\ubar(r_1r_2))^{-1}
  \\
  &
  \ =\
  \Theta^*(\ubar(r_1))
  \cdot
  \Theta^*(\ubar(r_2))
  \cdot
  \underbrace{
    \Theta^*(\jbar(\fbar(r_1,r_2)))^{-1}
  }_{
    \ =\
    \Theta_P(\fbar(r_1,r_2))^{-1}
  }
  \cdot
  \underbrace{
    \Theta(\jbar(\fbar(r_1,r_2)))
  }_{
    \ =\
    \Theta_P(\fbar(r_1,r_2))
  }
  \cdot
  \Theta(\ubar(r_2))^{-1}
  \cdot
  \Theta(\ubar(r_1))^{-1}
  \\
  &
  \ =\
  \Theta^*(\ubar(r_1))
  \cdot
  \underbrace{
    \Theta^*(\ubar(r_2))
    \cdot
    \Theta(\ubar(r_2))^{-1}
  }_{
    \ =\
    \Gamma(r_2)
  }
  \cdot
  \Theta(\ubar(r_1))^{-1}
  \\
  &
  \ =\
  \Gamma(r_1)
  \cdot
  \Theta(\ubar(r_1))
  \cdot
  \Gamma(r_2)
  \cdot
  \Theta(\ubar(r_1))^{-1}
  \ =\
  \Gamma(r_1)
  \cdot
  \upperleft
  {r_1}
  {\Gamma(r_2)}
  \qquad
  \text{in $H^1(P,Z(K))$}
  ,
  \end{aligned}
\]
where the last equality holds
by corollary~\ref{cor:ext aut R-equivariant H^1}.
This shows that
$\Gamma : R \rDotsto H^1(P,Z(K))$
is a 1-cocycle.
The map $- \diamond \Theta$
sends
$\Gamma \in Z^1(R,H^1(P,Z(K)))$
to the homomorphism
$\Theta' : Q \rTo \Out(N;K)$
given by
$\Theta'(q)
 :=
 \Gamma(\phibar(q))
 \cdot
 \Theta(q)
$.
For any element
$q \in Q$
written in the form
$q = \ubar(r) \cdot \jbar(p)$
(with $p \in P$ and $r \in R$),
we have
$\Theta(\jbar(p))
 =
 \Theta_P(p)
 =
 \Theta^*(\jbar(p))
$,
and so
\[
  \begin{aligned}
  \Theta'(q)
  &
  \ =\
  \bigl(\,
  \Theta^*(\ubar(r))
  \cdot
  \Theta(\ubar(r))^{-1}
  \,\bigr)
  \cdot
  \Theta(\ubar(r))
  \cdot
  \Theta(\jbar(p))
  \\
  &
  \ =\
  \Theta^*(\ubar(r))
  \cdot
  \Theta^*(\jbar(p))
  \ =\
  \Theta^*(q)
  \qquad
  \text{in $\Out(N;K)$}
  ,
  \end{aligned}
\]
which shows that
$\Theta' = \Theta^*$;
hence
the map $- \diamond \Theta$
is surjective.
\end{proof}

\begin{rmk}
The proof shows that
the inverse of
the bijection $- \diamond \Theta$
of theorem~\ref{thm:mod-K outer action classification}
is given by
\[
  \begin{array}{rcl}
  \Bigl\{\
  \text{$(\theta,\conj_P^Q)$-prolongations of $\Theta_P$}
  \ \Bigr\}
  &
  \rTo^{\simeq}
  &
  Z^1(R,H^1(P,Z(K)))
  \\[1ex]
  \Theta'
  &
  \rMapsto
  &
  \Bigl(\
    r
    \mapsto
    \Theta'(\ubar(r))
    \cdot
    \Theta(\ubar(r))^{-1}
  \ \Bigr)
  \end{array}
\]
for any choice of
a section $\ubar$ of $\phibar$.
\end{rmk}

\begin{defn}
\label{defn:Aut(KNP)-conj Theta}
Two mod-$K$outer actions
$\Theta_\ell
 :
 Q
 \rTo
 \Out(N;K)
$
of $Q$ on $N$
(for $\ell=1,2$)
are \emph{$\Aut(KNP)$-conjugate}
iff
there exists
an automorphism
$\eta \in \Aut(KNP)$
of the extension $(KNP)$
such that
for any $q \in Q$,
one has
\[
  \Theta_2(q)
  \ =\
  \etabar^{-1}
  \cdot
  \Theta_1(q)
  \cdot
  \etabar
  \qquad
  \text{in $\Out(N;K)$}
  ,
\]
where $\etabar \in \Out(KNP;K)$ denotes
the image of $\eta$
in $\Out(N;K)$.
In this case,
we write
$\Theta_2 = \Theta_1^{\etabar}$.
\end{defn}

If $\Theta$ is
a $(\theta,\conj_P^Q)$-prolongation of $\Theta_P$,
then so is
any $\Aut(KNP)$-conjugate $\Theta^\etabar$
of $\Theta$.
Indeed,
lemma~\ref{lemma:Theta_P fixed under Aut(KNP)}
shows that
$\Theta^\etabar \circ \jbar
 =
 \Theta \circ \jbar
 =
 \Theta_P
$,
and the fact that
$\eta$ induces
the trivial automorphism
on $K$ and on $P$
implies that
both $\Theta$ and $\Theta^\etabar$
induce
$\theta$ and $\conj_P^Q$.

\begin{cor}
\label{cor:mod-K outer action classification B^1}
The bijection $- \diamond \Theta$
of theorem~\ref{thm:mod-K outer action classification}
restricts to
a bijection
\[
  \begin{array}{rcl}
  - \diamond \Theta
  \ :\
  B^1(R,H^1(P,Z(K)))
  &
  \rTo^{\simeq}
  &
  \Bigl\{\
  \text{$\Aut(KNP)$-conjugates of $\Theta$}
  \ \Bigr\}
  ,
  \\[1ex]
  \partial\etabar
  &
  \rMapsto
  &
  \Theta^\etabar
  ,
  \end{array}
\]
where
$\partial\etabar$ denotes
the 1-coboundary
$\partial\etabar(r)
 :=
 \etabar^{-1}
 \cdot
 \upperleft{r}{\etabar}
$
for any
$\etabar \in H^1(P,Z(K))$.
\end{cor}

\begin{proof}
For any
$\etabar \in H^1(P,Z(K))$,
the coboundary
$\partial\etabar \in B^1(R,H^1(P,Z(K)))$
is mapped by $- \diamond \Theta$
to the mod-$K$outer action
$\Theta' : Q \rTo \Out(N;K)$
which sends
any element $q \in Q$
to
\[
  \Theta'(q)
  \ =\
  (\partial\etabar)(\phibar(q))
  \cdot
  \Theta(q)
  \ =\
  \etabar^{-1}
  \cdot
  \upperleft{\phibar(q)}{\etabar}
  \cdot
  \Theta(q)
  \qquad
  \text{in $\Out(N;K)$}
  .
\]
Corollary~\ref{cor:ext aut R-equivariant H^1}
allows us to replace
$\upperleft{\phibar(q)}{\etabar}$
by
$\Theta(q)
 \cdot
 \etabar
 \cdot
 \Theta(q)^{-1}
$
and see that
$\Theta'(q)
 =
 \Theta^\etabar(q)
$
in $\Out(N;K)$.
\end{proof}

Theorem~\ref{thm:mod-K outer action classification}
show that
$Z^1(R,H^1(P,Z(K)))$
acts transitively
(from the left)
on the set of
$\Aut(KNP)$-conjugacy classes of
$(\theta,\conj_P^Q)$-prolongations of $\Theta_P$,
while
corollary~\ref{cor:mod-K outer action classification B^1}
shows that
the stabilizer subgroup
of any given
$\Aut(KNP)$-conjugacy class
is $B^1(R,H^1(P,Z(K)))$.
Hence we have:

\begin{cor}
\label{cor:mod-K outer action classification H^1}
Let
$\Theta
 :
 Q
 \rTo
 \Out(N;K)
$
be a $(\theta,\conj_P^Q)$-prolongation of $\Theta_P$.
The bijection $- \diamond \Theta$
of theorem~\ref{thm:mod-K outer action classification}
induces a bijection
\[
  - \diamond \Theta
  \ :\
  H^1(R,H^1(P,Z(K)))
  \rTo^{\simeq}
  \left\{\
  \begin{array}{l}
  \text{$\Aut(KNP)$-conjugacy classes of}
  \\
  \text{$(\theta,\conj_P^Q)$-prolongations of $\Theta_P$}
  \end{array}
  \ \right\}
  .
\]
\end{cor}

\section{Reduction
from $H^2_P(Q,Z(K))$
to $H^1(R,H^1(P,Z(K)))$}

The \emph{reduction homomorphism}
\[
  \rd
  \ :\
  H^2_P(Q,Z(K))
  \rTo
  H^1(R,H^1(P,Z(K)))
  ,
  \qquad
  [e]
  \rMapsto\relax
  \rd[e]
  ,
\]
which appear in
the exact sequence~\eqref{eqn:long exact seq},
arises from
the $E_2$-spectral sequence
for the extension $(PQR)$
with coefficients in $Z(K)$;
let us first recall
its explicit description.

Given
a 2-cohomology class
$[e] \in H^2_P(Q,Z(K))$,
the fact that
$[e]$ becomes
the trivial class
when restricted to $P$
means that
we can choose
a representative 2-cocycle
$e : Q \times Q \rDotsto Z(K)$
with the property that
$e(\jbar(p_1),\jbar(p_2))
 =
 1_{Z(K)}
$
for any $p_1,p_2 \in P$.
For any $r \in R$
and any choice of
an element $q \in Q$
such that
$\phibar(q) = r$
in $R$,
we define the map
$\Gammatilde_e(r)_q : P \rDotsto Z(K)$
by setting
\begin{equation}
  \Gammatilde_e(r)_q(p)
  \ :=\
  e
  \bigl(\,
  q
  \,,\,
  q^{-1}\,\jbar(p)\,q
  \,\bigr)
  \cdot
  e
  \bigl(\,
  \jbar(p)
  \,,\,
  q
  \,\bigr)^{-1}
  .
  \label{eqn:Gammatilde_e(r)}
\end{equation}
Then
$\Gammatilde_e(r)_q$
is a 1-cocycle,
and its cohomology class
$\Gamma_e(r) \in H^1(P,Z(K))$
is independent of
the choice of
$q \in Q$ above.
The resulting map
$\Gamma_e : R \rDotsto H^1(P,Z(K))$,
which sends $r \in R$
to the cohomology class
$\Gamma_e(r) \in H^1(P,Z(K))$,
is then
a 1-cocycle
for the action of $R$
on $H^1(P,Z(K))$,
and its cohomology class
$[\Gamma_e] \in H^1(R,H^1(P,Z(K)))$
is independent of
the choice of
the representative 2-cocycle $e$
with the above property.
The reduction image of $[e]$
is then defined as
\[
  \rd[e]
  \ :=\
  [\Gamma_e]
  \qquad
  \text{in $H^1(R,H^1(P,Z(K)))$}
  .
\]

The reduction homomorphism
can be interpreted
in terms of
the extensions
and their outer actions.
Let
$(\,
 G
 \,,\,
 i
 \,,\,
 \pi
 \,)
$
and
$(\,
 G'
 \,,\,
 i'
 \,,\,
 \pi'
 \,)
$
be extensions
of $K$ by $Q$
with the same
outer action~$\theta$
and with
isomorphic
$P$-subextensions
$(\,
 N
 \,,\,
 i_0
 \,,\,
 \pi_0
 \,)
$
and
$(\,
 N'
 \,,\,
 i_0'
 \,,\,
 \pi_0'
 \,)
$.
As in definition~\ref{defn:P-subextension},
we have
the canonical inclusions
$j : N \rInto G$
and
$j' : N' \rInto G'$.
Let
$\Theta : Q \rTo \Out(N;K)$
be the mod-$K$outer action
of the iterated extension
$(\,
 G
 \,,\,
 j
 \,,\,
 \pi
 \,)
$;
it is
a $(\theta,\conj_P^Q)$-prolongation
of $\Theta_P$.
On the other hand,
consider the mod-$K$outer action
$\Theta' : Q \rTo \Out(N';K)$
of the iterated extension
$(\,
 G'
 \,,\,
 j'
 \,,\,
 \pi'
 \,)
$.
For any isomorphism
$\varphi : N' \rTo^{\simeq} N$
between
the $P$-subextensions
$(\,
 N'
 \,,\,
 i_0'
 \,,\,
 \pi_0'
 \,)
$
and
$(\,
 N
 \,,\,
 i_0
 \,,\,
 \pi_0
 \,)
$,
the homomorphism
$\upperleft{\varphi}{\Theta'} : Q \rTo \Out(N;K)$
given by
\[
  \upperleft{\varphi}{\Theta'}(q)
  \ :=\
  \varphi
  \circ
  \Theta'(q)
  \circ
  \varphi^{-1}
  \qquad
  \text{in $\Out(N;K)$}
\]
is a mod-$K$outer action
of $Q$ on $N$,
which is also
a $(\theta,\conj_P^Q)$-prolongation
of $\Theta_P$.
Another choice of
an isomorphism
between
$(\,
 N'
 \,,\,
 i_0'
 \,,\,
 \pi_0'
 \,)
$
and
$(\,
 N
 \,,\,
 i_0
 \,,\,
 \pi_0
 \,)
$
would be of the form
$\eta \circ \varphi$
for some $\eta \in \Aut(KNP)$,
and so
it follows that
the $\Aut(KNP)$-conjugacy class
of $\upperleft{\varphi}{\Theta'}$
is independent of
the choice of $\varphi$;
we denote it by
$[\Theta']$.
Hence
we have a well-defined map
\begin{equation}
  \begin{array}{rcl}
  \left\{\
  \begin{array}{l}
  \text{isomorphism classes of}
  \\
  \text{extensions of $K$ by $Q$}
  \\
  \text{with outer action $\theta$}
  \\
  \text{whose $P$-subextension}
  \\
  \text{is isomorphic to}
  \quad
  (\,
  N
  \,,\,
  i_0
  \,,\,
  \pi_0
  \,)
  \end{array}
  \ \right\}
  &
  \rTo
  &
  \left\{\
  \begin{array}{l}
  \text{$\Aut(KNP)$-conjugacy classes of}
  \\
  \text{$(\theta,\conj_P^Q)$-prolongations of $\Theta_P$}
  \end{array}
  \ \right\}
  ,
  \\[8ex]
  (\,
  G'
  \,,\,
  i'
  \,,\,
  \pi'
  \,)
  &
  \rMapsto
  &
  [\Theta']
  .
  \end{array}
  \label{eqn:mod-K outer action}
\end{equation}

\begin{prop}
\label{prop:ext rd}
With the above notation,
let $\Gamma \in Z^1(R,H^1(P,Z(K)))$ be
the unique 1-cocycle
such that
$\upperleft{\varphi}{\Theta'}
 =
 \Gamma \diamond \Theta
$
as $(\theta,\conj_P^Q)$-prolongation
of $\Theta_P$,
and let
\[
  [e]
  \ :=\
  \dfrac{(G',i',\pi')}{(G,i,\pi)}
  \ \in\
  H^2_P(Q,Z(K))
  .
\]
Then
\[
  \rd[e]
  \ =\
  [\Gamma]
  \qquad
  \text{in $H^1(R,H^1(P,Z(K)))$}
  .
\]
In other words,
the following diagram
commutes:
\[
  \begin{diagram}
  H^2_P(Q,Z(K))
  &
  \rTo^{\rd}
  &
  H^1(R,H^1(P,Z(K)))
  \\
  \dTo^{
    \text{thm.~\ref{thm:ext classification given restriction}}
    \quad
    \wr|
  }_{
    \quad
    - \boxtimes G
  }
  &
  {\Bigg.}
  &
  \dTo^{
    - \diamond \Theta
    \quad
  }_{
    |\wr
    \quad
    \text{cor.~\ref{cor:mod-K outer action classification H^1}}
  }
  \\
  \left\{\
  \begin{array}{l}
  \text{isomorphism classes of}
  \\
  \text{extensions of $K$ by $Q$}
  \\
  \text{with outer action $\theta$}
  \\
  \text{whose $P$-subextension}
  \\
  \text{is isomorphic to}
  \quad
  (\,
  N
  \,,\,
  i_0
  \,,\,
  \pi_0
  \,)
  \end{array}
  \ \right\}
  &
  \rTo_{\quad \text{\eqref{eqn:mod-K outer action}} \quad}
  &
  \left\{\
  \begin{array}{l}
  \text{$\Aut(KNP)$-conjugacy classes of}
  \\
  \text{$(\theta,\conj_P^Q)$-prolongations of $\Theta_P$}
  \end{array}
  \ \right\}
  .
  \end{diagram}
  \qquad\qquad
\]
\end{prop}

\begin{proof}
Choose
a lifting
$\delta : Q \rDotsto \Aut(K)$
of
$\theta : Q \rTo \Out(K)$.
Since
$\conj_K^G$ and $\conj_K^{G'}$
both send $K$
surjectively onto
$\Inn(K)$,
we may choose
sections
$s : Q \rDotsto G$
and
$s' : Q \rDotsto G'$
of
$G \rOnto^{\pi} Q$
and
$G' \rOnto^{\pi'} Q$
respectively,
such that
\begin{equation}
  \conj_K^G \circ s
  \ =\
  \delta
  \ =\
  \conj_K^{G'} \circ s'
  \qquad
  \text{as maps}
  \quad
  Q \rDotsto \Aut(K)
  .
  \label{eqn:conj_K^G s}
\end{equation}
Thus
$(\,
 G
 \,,\,
 i
 \,,\,
 \pi
 \,,\,
 s
 \,)
$
and
$(\,
 G'
 \,,\,
 i'
 \,,\,
 \pi'
 \,,\,
 s'
 \,)
$
are
$\delta$-sectioned extensions
of $K$ by $Q$.
We note in passing that
these conditions
only determine
the sections $s$ and $s'$
modulo $Z(K)$;
we are free to
multiply (say) $s$ by
any map
$Q \rDotsto Z(K)$.

Let
$s_0 : P \rDotsto N$
and
$s_0' : P \rDotsto N'$
be the restrictions to $P$
of $s$ and $s'$
respectively;
these are
sections of
$N \rOnto^{\pi_0} P$
and
$N' \rOnto^{\pi_0'} P$,
characterized by
the property that
\[
  j \circ s_0
  \ =\ 
  s \circ \jbar
  \quad
  \text{as maps}
  \quad
  P \rDotsto G
  \qquad
  \text{and}
  \qquad
  j' \circ s_0'
  \ =\ 
  s' \circ \jbar
  \quad
  \text{as maps}
  \quad
  P \rDotsto G'
  .
\]
Fix an isomorphism
 $\varphi : N' \rTo^{\simeq} N$
between
the $P$-subextensions
$(\,
 N'
 \,,\,
 i_0'
 \,,\,
 \pi_0'
 \,)
$
and
$(\,
 N
 \,,\,
 i_0
 \,,\,
 \pi_0
 \,)
$.
The map
$\varphi \circ s_0' : P \rDotsto N$
is then
also a section of
$N \rOnto^{\pi_0} P$.
We claim that
the sections $s$ and $s'$
can be chosen
to be compatible with $\varphi$,
in the sense that
$\varphi \circ s_0' = s_0$
as sections of
$N \rOnto^{\pi_0} P$.
Indeed,
for any $k \in K$
and any $p \in P$,
the isomorphism $\varphi$
transforms the equation
\[
  i_0'
  \bigl(\,
    \upperleft
    {\conj_K^{N'}(s_0'(p))}
    {k}
  \,\bigr)
  \ =\ 
  s_0'(p)
  \cdot
  i_0'(k)
  \cdot
  s_0'(p)^{-1}
  \qquad
  \text{in $N'$}
\]
into the equation
\[
  i_0
  \bigl(\,
    \upperleft
    {\conj_K^{N'}(s_0'(p))}
    {k}
  \,\bigr)
  \ =\ 
  (\varphi \circ s_0')(p)
  \cdot
  i_0(k)
  \cdot
  (\varphi \circ s_0')(p)^{-1}
  \ =\ 
  i_0
  \bigl(\,
    \upperleft
    {\conj_K^{N}((\varphi \circ s_0')(p))}
    {k}
  \,\bigr)
  \qquad
  \text{in $N$}
  ,
\]
which shows that
$\conj_K^{N'}(s_0'(p))
 =
 \conj_K^N((\varphi \circ s_0')(p))
$
in $\Aut(K)$.
On the other hand,
equation~\eqref{eqn:conj_K^G s}
implies that
$\conj_K^{N'}(s_0'(p))
 =
 \conj_K^{N}(s_0(p))
$
in $\Aut(K)$.
Hence
$\conj_K^N \circ (\varphi \circ s_0')
 =
 \conj_K^{N} \circ s_0
$
as maps
$P \rDotsto \Aut(K)$,
which implies that
$\varphi \circ s_0'$ and $s_0$
differ multiplicatively by
some map
$P \rDotsto Z(K)$.
We can therefore
adjust the section $s$
accordingly
(on the subgroup $P$
of its domain $Q$)
to achieve
$\varphi \circ s_0' = s_0$.

The (left) factor sets
$h : Q \times Q \rDotsto K$
and
$h' : Q \times Q \rDotsto K$
of the $\delta$-sectioned
extensions
$(\,
 G
 \,,\,
 i
 \,,\,
 \pi
 \,,\,
 s
 \,)
$
and
$(\,
 G'
 \,,\,
 i'
 \,,\,
 \pi'
 \,,\,
 s'
 \,)
$
are characterized by
the property that
for any $q_1,q_2 \in Q$,
one has
\begin{align}
  s(q_1)
  \cdot
  s(q_2)
  &
  \ =\
  i
  (\,
    h(q_1,q_2)
  \,)
  \cdot
  s(q_1q_2)
  &
  &
  \text{in $G$}
  \label{eqn:h}
  \\
  \text{and}
  \qquad
  s'(q_1)
  \cdot
  s'(q_2)
  &
  \ =\
  i'
  (\,
    h'(q_1,q_2)
  \,)
  \cdot
  s'(q_1q_2)
  &
  &
  \text{in $G'$}
  .
  \label{eqn:h'}
\end{align}
By theorem~\ref{thm:ext classification},
there is a
2-cocycle
$e : Q \times Q \rDotsto Z(K)$
such that
\[
  h'
  \ =\
  e
  \cdot
  h
  \qquad
  \text{as maps}
  \quad
  Q \times Q \rDotsto K
  ,
\]
and by notation~\ref{notatn:ext cohom class}
and theorem~\ref{thm:ext classification given restriction},
the cohomology class of $e$
is precisely
$[e]
 =
 \dfrac{(G',i',\pi')}{(G,i,\pi)}
$
in $H^2_P(Q,Z(K))$.
Then
for any $p_1,p_2 \in P$,
equation~\eqref{eqn:h'} gives
\[
  s_0'(p_1)
  \cdot
  s_0'(p_2)
  \ =\
  i_0'
  \Bigl(\,
  e(\jbar(p_1),\jbar(p_2))
  \cdot
  h(\jbar(p_1),\jbar(p_2))
  \,\Bigr)
  \cdot
  s_0'(p_1p_2)
  \qquad
  \text{in $N'$}
  ,
\]
which,
thanks to
our (justified) assumption that
$\varphi \circ s_0' = s_0$,
is transformed
by the isomorphism $\varphi$
into the equation
\[
  s_0(p_1)
  \cdot
  s_0(p_2)
  \ =\
  i_0
  \Bigl(\,
  e(\jbar(p_1),\jbar(p_2))
  \cdot
  h(\jbar(p_1),\jbar(p_2))
  \,\Bigr)
  \cdot
  s_0(p_1p_2)
  \qquad
  \text{in $N$}
  .
\]
But by equation~\eqref{eqn:h},
the left hand side
is equal to
$i_0
 \Bigl(\,
 h(\jbar(p_1),\jbar(p_2))
 \,\Bigr)
 \cdot
 s_0(p_1p_2)
$.
From this,
we see that
the cocycle $e$
has the property that
$e(\jbar(p_1),\jbar(p_2))
 =
 1_{Z(K)}
$
for any $p_1,p_2 \in P$,
and so
it can be used
in~\eqref{eqn:Gammatilde_e(r)}
to determine
the reduction image
$\rd[e]$ of $[e]$.

Note that
(cf.~definition~\ref{defn:mod-K outer action of iterext})
since $\Theta$ is
the mod-$K$outer action
of the iterated extension
$(\,
 G
 \,,\,
 j
 \,,\,
 \pi
 \,)
$,
it induced by
the conjugation action $\conj_N^G$
of $G$ on $N$
and it makes
the diagram~\eqref{diag:G Theta}
commutes;
similarly for $\Theta'$.
It follows that
the maps
\[
  \Sigma
  \ :=\
  \conj_N^G \circ s
  \ :\
  Q
  \rDotsto
  \Aut_K(N)
  \qquad
  \text{and}
  \qquad
  \Sigma'
  \ :=\
  \conj_{N'}^{G'} \circ s'
  \ :\
  Q
  \rDotsto
  \Aut_K(N')
\]
are liftings of
$\Theta : Q \rTo \Out(N;K)$
and
$\Theta' : Q \rTo \Out(N';K)$
respectively.

We now fix $q \in Q$
and compute the effects of
conjugation in $G$ and $G'$
respectively.
First,
for any $n \in N$
written in the form
$n = i_0(k) \cdot s_0(p)$
(with $k \in K$ and $p \in P$),
we have
\[
  \begin{aligned}
  s(q)^{-1}
  \cdot
  j(n)
  \cdot
  s(q)
  &
  \ =\
  s(q)^{-1}
  \cdot
  i(k)
  \cdot
  s(\jbar(p))
  \cdot
  s(q)
  \\
  &
  \ =\
  i
  \Bigl(\
  \upperleft
  {\delta(q)^{-1}}
  {k}
  \ \Bigr)
  \cdot
  s(q)^{-1}
  \cdot
  i
  \Bigl(\
  h(\jbar(p),q)
  \ \Bigr)
  \cdot
  s(\jbar(p)\,q)
  \\
  &
  \ =\
  i
  \Bigl(\
  \upperleft
  {\delta(q)^{-1}}
  {k}
  \cdot
  \upperleft
  {\delta(q)^{-1}}
  {h(\jbar(p),q)}
  \ \Bigr)
  \cdot
  s(q)^{-1}
  \cdot
  s(\jbar(p)\,q)
  \\
  &
  \ =\
  \begin{array}[t]{r}
  i
  \Bigl(\
  \upperleft
  {\delta(q)^{-1}}
  {k}
  \cdot
  \upperleft
  {\delta(q)^{-1}}
  {h(\jbar(p),q)}
  \cdot
  \upperleft
  {\delta(q)^{-1}}
  {h(q,q^{-1}\,\jbar(p)\,q)^{-1}}
  \ \Bigr)
  \cdot
  s(q^{-1}\,\jbar(p)\,q)
  \\
  \text{in $G$}
  ,
  \end{array}
  \end{aligned}
\]
where the last equality
follows from
the identity
\[
  s(q)
  \cdot
  s(q^{-1}\,\jbar(p)\,q)
  \ =\
  i
  \Bigl(\,
  h(q,q^{-1}\,\jbar(p)\,q)
  \,\Bigr)
  \cdot
  s(\jbar(p)\,q)
  \qquad
  \text{in $G$}
  \qquad
  \text{deduced from~\eqref{eqn:h}}
  .
\]
Consequently,
we see that
the automorphism
$\Sigma(q)^{-1}
 =
 \conj_N^G(s(q)^{-1})
 \in
 \Aut_K(N)
$,
which is a lift of
$\Theta(q)^{-1} \in \Out(N;K)$,
acts on $N$ by
sending $n = i_0(k) \cdot s_0(p)$ to
\begin{equation}
  \upperleft
  {\Sigma(q)^{-1}}
  {n}
  \ =\
  i_0
  \Bigl(\
  \upperleft
  {\delta(q)^{-1}}
  {k}
  \cdot
  \upperleft
  {\delta(q)^{-1}}
  {h(\jbar(p),q)}
  \cdot
  \upperleft
  {\delta(q)^{-1}}
  {h(q,q^{-1}\,\jbar(p)\,q)^{-1}}
  \ \Bigr)
  \cdot
  s_0
  \Bigl(\
  \upperleft
  {\conj_P^Q(q^{-1})}
  {p}
  \ \Bigr)
  \qquad
  \text{in $N$}
  .
  \label{eqn:Sigma inv on n}
\end{equation}
Next,
for any $n' \in N'$
written in the form
$n' = i_0'(k) \cdot s_0'(p)$
(with $k \in K$ and $p \in P$),
we have
\[
  \begin{aligned}
  s'(q)
  \cdot
  j'(n')
  \cdot
  s'(q)^{-1}
  &
  \ =\
  s'(q)
  \cdot
  i'(k)
  \cdot
  s'(\jbar(p))
  \cdot
  s'(q)^{-1}
  \\
  &
  \ =\
  i'
  \Bigl(\
  \upperleft
  {\delta(q)}
  {k}
  \ \Bigr)
  \cdot
  s'(q)
  \cdot
  s'(\jbar(p))
  \cdot
  s'(q)^{-1}
  \\
  &
  \ =\
  i'
  \Bigl(\
  \upperleft
  {\delta(q)}
  {k}
  \cdot
  h'(q,\jbar(p))
  \ \Bigr)
  \cdot
  s'(q\,\jbar(p))
  \cdot
  s'(q)^{-1}
  \\
  &
  \ =\
  i'
  \Bigl(\
  \upperleft
  {\delta(q)}
  {k}
  \cdot
  h'(q,\jbar(p))
  \cdot
  h'(q\,\jbar(p)\,q^{-1},q)^{-1}
  \ \Bigr)
  \cdot
  s'(q\,\jbar(p)\,q^{-1})
  \qquad
  \text{in $G'$}
  ,
  \end{aligned}
\]
where the last equality
follows from
the identity
\[
  s'(q\,\jbar(p)\,q^{-1})
  \cdot
  s'(q)
  \ =\
  i'
  \Bigl(\,
  h'(q\,\jbar(p)\,q^{-1},q)
  \,\Bigr)
  \cdot
  s'(q\,\jbar(p))
  \qquad
  \text{in $G'$}
  \qquad
  \text{deduced from~\eqref{eqn:h'}}
  .
\]
Thus,
we see that
the automorphism
$\Sigma'(q)
 =
 \conj_{N'}^{G'}(s'(q))
 \in
 \Aut_K(N')
$,
which is a lift of
$\Theta'(q) \in \Out(N';K)$,
acts on $N'$ by
sending $n' = i_0'(k) \cdot s_0'(p)$ to
\begin{equation}
  \upperleft
  {\Sigma'(q)}
  {n'}
  \ =\
  i_0'
  \Bigl(\
  \upperleft
  {\delta(q)}
  {k}
  \cdot
  h'(q,\jbar(p))
  \cdot
  h'(q\,\jbar(p)\,q^{-1},q)^{-1}
  \ \Bigr)
  \cdot
  s_0'
  \Bigl(\,
  \upperleft
  {\conj_P^Q(q)}
  {p}
  \,\Bigr)
  \qquad
  \text{in $N'$}
  .
  \label{eqn:Sigma' on n'}
\end{equation}
We now use
the isomorphism
$\varphi : N' \rTo^{\simeq} N$
to re-express this
as an equality in $N$.
Thus,
for any $n \in N$
written in the form
$n = i_0(k) \cdot s_0(p)$
(with $k \in K$ and $p \in P$),
we apply $\varphi$
to equation~\eqref{eqn:Sigma' on n'}
with
$n'
 :=
 \varphi^{-1}(n)
 =
 i_0'(k) \cdot s_0'(p)
 \in
 N'
$
and obtain
\[
  \upperleft
  {(\varphi
    \,\circ\,
    \Sigma'(q)
    \,\circ\,
    \varphi^{-1})
  }
  {n}
  \ =\
  i_0
  \Bigl(\
  \upperleft
  {\delta(q)}
  {k}
  \cdot
  h'(q,\jbar(p))
  \cdot
  h'(q\,\jbar(p)\,q^{-1},q)^{-1}
  \ \Bigr)
  \cdot
  s_0
  \Bigl(\,
  \upperleft
  {\conj_P^Q(q)}
  {p}
  \,\Bigr)
  \qquad
  \text{in $N$}
  .
\]
Note that
the automorphism
$\varphi\circ\Sigma'(q)\circ\varphi^{-1}
 \in
 \Aut_K(N)
$
is a lift of
$\upperleft{\varphi}{\Theta'}(q)
 \in
 \Out(N;K)
$.
We now replace $n$ by
$\upperleft{\Sigma(q)^{-1}}{n}$
in this last equation;
by~\eqref{eqn:Sigma inv on n},
this amounts to
replacing
\[
  p
  \quad
  \text{by}
  \quad
  \upperleft{\conj_P^Q(q^{-1})}{p}
  \qquad
  \text{and}
  \qquad
  k
  \quad
  \text{by}
  \quad
  \upperleft
  {\delta(q)^{-1}}
  {k}
  \cdot
  \upperleft
  {\delta(q)^{-1}}
  {h(\jbar(p),q)}
  \cdot
  \upperleft
  {\delta(q)^{-1}}
  {h(q,q^{-1}\,\jbar(p)\,q)^{-1}}
  ,
\]
and we arrive at
the final result
of our computations:
for any $q \in Q$,
the automorphism
given by
$\varphi
 \,\circ\,
 \Sigma'(q)
 \,\circ\,
 \varphi^{-1}
 \,\circ\,
 \Sigma(q)^{-1}
 \in
 \Aut_K(N)
$,
which is a lift of
the element
$\Gamma(\phibar(q))
 =
 \upperleft{\varphi}{\Theta'}(q)
 \circ
 \Theta(q)^{-1}
 \in
 \Out(KNP;K)
$,
sends $n \in N$ to
\[
  \begin{aligned}
  &
  i_0
  \Bigl(\
  k
  \cdot
  h(\jbar(p),q)
  \cdot
  h(q,q^{-1}\,\jbar(p)\,q)^{-1}
  \cdot
  h'(q,q^{-1}\,\jbar(p)\,q)
  \cdot
  h'(\jbar(p),q)^{-1}
  \ \Bigr)
  \cdot
  s_0(p)
  \\
  \ =\
  &
  i_0
  \Bigl(\
  \underbrace{
  e(q,q^{-1}\,\jbar(p)\,q)
  \cdot
  e(\jbar(p),q)^{-1}
  }_{
    \text{in $Z(K)$}
  }
  \ \Bigr)
  \cdot
  \underbrace{
  i_0(k)
  \cdot
  s_0(p)
  }_{
    \ =\
    n
  }
  \qquad
  \text{in $N$}
  .
  \end{aligned}
\]

The above computations
show that
for any $r \in R$
and any choice of
an element $q \in Q$
such that
$\phibar(q) = r$
in $R$,
$\Gamma(r) \in H^1(P,Z(K))$
is the cohomology class of
the 1-cocycle
\[
  \Gammatilde(r)_q
  \ =\
  \Bigl(\
  p
  \ \mapsto\
  e
  (\,
  q
  \,,\,
  q^{-1}\,\jbar(p)\,q
  \,)
  \cdot
  e
  (\,
  \jbar(p)
  \,,\,
  q
  \,)^{-1}
  \ \Bigr)
  .
\]
This coincides with
the 1-cocycle
$\Gammatilde_e(r) \in Z^1(P,Z(K))$
defined in
equation~\eqref{eqn:Gammatilde_e(r)},
so it follows that
$\Gamma(r)
 =
 \Gamma_e(r)
$
in $H^1(P,Z(K))$
for any $r \in R$.
Since
$\Gamma_e \in Z^1(R,H^1(P,Z(K)))$
represents
the reduction image
$\rd[e] \in H^1(R,H^1(P,Z(K)))$
of $[e] \in H^2_P(Q,Z(K))$,
we conclude that
$\rd[e] = [\Gamma]$
in $H^1(R,H^1(P,Z(K)))$.
\end{proof}

By propositions~\ref{prop:iterext infl}
and~\ref{prop:ext rd},
the exactness of
\[
  H^2(R,Z(K)^P)
  \rTo^{\quad \infl \quad}
  H^2_P(Q,Z(K))
  \rTo^{\quad \rd \quad}
  H^1(R,H^1(P,Z(K)))
\]
in the sequence~\eqref{eqn:long exact seq}
translates as:

\begin{prop}
Let
$(\,
 G
 \,,\,
 i
 \,,\,
 \pi
 \,)
$
and
$(\,
 G'
 \,,\,
 i'
 \,,\,
 \pi'
 \,)
$
be extensions of
$K$ by $Q$
with the same outer action
$\theta$
and
with isomorphic
$P$-subextensions
$(KNP)
 =
 (\,
 N
 \,,\,
 i_0
 \,,\,
 \pi_0
 \,)
$
and
$(\,
 N'
 \,,\,
 i_0'
 \,,\,
 \pi_0'
 \,)
$
respectively;
let
$\varphi : N' \rTo^{\simeq} N$
be such an isomorphism.
Let
\[
  \Theta
  \ :\
  Q
  \rTo
  \Out(N;K)
  \qquad
  \text{and}
  \qquad
  \Theta'
  \ :\
  Q
  \rTo
  \Out(N',K)
\]
be the mod-$K$outer actions
induced by
the conjugation actions of
$G$ on $N$
and
$G'$ on $N'$
respectively.
Then
$\Theta$ and $\upperleft{\varphi}{\Theta'}$
are $\Aut(KNP)$-conjugate
if and only if
there exists
an iterated extension
of $(KNP)$ by $(PQR)$,
having
$(\,
 G'
 \,,\,
 i'
 \,,\,
 \pi'
 \,)
$
as its $Q$-main extension,
and with $\Theta$
as its mod-$K$outer action.
\end{prop}

\begin{proof}
It is instructive
to prove this result
directly.
We first note that
any iterated extension
of $(KNP)$ by $(PQR)$
having
$(\,
 G'
 \,,\,
 i'
 \,,\,
 \pi'
 \,)
$
as its $Q$-main extension
must be of the form
$(\,
 G'
 \,,\,
 j^*
 \,,\,
 \pi'
 \,)
$
for some
injective homomorphism
$j^* : N \rInto G'$
such that
$j^* \circ i_0
 =
 i'
$
and
$\pi' \circ j^*
  =
 \jbar \circ \pi_0
$;
that is to say,
such that
$j^*$ makes
the following diagram
commute:
\begin{equation}
  \begin{diagram}
  K
  &
  \rInto^{\qquad i_0 \qquad}
  &
  N
  &
  \rOnto^{\qquad \pi_0 \qquad}
  &
  P
  \\
  \dEq
  &
  &
  \dInto_{j^*}
  &
  &
  \dInto_{\jbar}
  \\
  K
  &
  \rInto^{\qquad i' \qquad}
  &
  G'
  &
  \rOnto^{\qquad \pi' \qquad}
  &
  Q
  \end{diagram}
  \label{diag:j^vee}
\end{equation}
On the other hand,
if $j'$ denotes
the canonical inclusion
from $N'$ into $G'$,
then
the fact that
$\varphi : N' \rTo^{\simeq} N$
is an isomorphism of
extensions of $K$ by $P$
shows that
the composite inclusion
$j' \circ \varphi^{-1} : N \rInto G'$
also makes
the diagram~\eqref{diag:j^vee}
commute
(when $j^*$
is replaced by
$j' \circ \varphi^{-1}$).
From these,
it follows that
$j^*$ and $j' \circ \varphi^{-1}$
differ by
an automorphism of
the extension $(KNP)$:
there exists
$\eta \in \Aut(KNP)$
such that
$j^* \circ \eta
 =
 j' \circ \varphi^{-1}
$
as homomorphisms
$N \rInto G'$.
The conjugation action
$\conj_N^{G'*}$
of $G'$ on $N$
(with respect to $j^*$)
is characterized by
the property that
for any $g' \in G$
and any $n \in N$,
one has
\[
  j^*
  \bigl(
    \upperleft
    {\conj_N^{G'*}(g')}
    {n}
  \bigr)
  \ =\ 
  g'
  \cdot
  j^*(n)
  \cdot
  {g'}^{-1}
  \qquad
  \text{in $G'$}
  ,
\]
which,
using
$j^*
 =
 j' \circ \varphi^{-1} \circ \eta^{-1}
$,
we can rewrite as
\[
  j'
  \Bigl(\ 
    \upperleft
    {(\,
      \varphi^{-1}
      \,\circ\,
      \eta^{-1}
      \,\circ\,
      \conj_N^{G'*}(g')
     \,)
    }
    {n}
  \ \Bigr)
  \ =\ 
  g'
  \cdot
  j'
  \Bigl(\ 
    \upperleft
    {(\,
      \varphi^{-1}
      \,\circ\,
      \eta^{-1}
     \,)
    }
    {n}
  \ \Bigr)
  \cdot
  {g'}^{-1}
  \ =\ 
  j'
  \Bigl(\ 
    \upperleft
    {(\,
      \conj_{N'}^{G'}(g')
      \,\circ\,
      \varphi^{-1}
      \,\circ\,
      \eta^{-1}
     \,)
    }
    {n}
  \ \Bigr)
  \qquad
  \text{in $G'$}
  ,
\]
where $\conj_{N'}^{G'}$
is the conjugation action
of $G'$ on $N'$
(with respect to $j'$).
From this
we infer that
\[
  \eta^{-1}
  \circ
  \conj_{N}^{G'*}(g')
  \circ
  \eta
  \ =\ 
  \varphi
  \circ
  \conj_{N'}^{G'}(g')
  \circ
  \varphi^{-1}
  \qquad
  \text{in $\Aut_K(N)$}
  .
\]
The conjugation action
$\conj_N^{G'*}$
induces
the mod-$K$outer action
$\Theta^* : Q \rTo \Out(N;K)$
of the iterated extension
$(\,
 G'
 \,,\,
 j^*
 \,,\,
 \pi'
 \,)
$,
whereas
$\conj_{N'}^{G'}$
induces
$\Theta' : Q \rTo \Out(N';K)$;
hence
for any $q \in Q$,
one has
\[
  \etabar^{-1}
  \cdot
  \Theta^*(q)
  \cdot
  \etabar
  \ =\ 
  \varphi
  \circ
  \Theta'(q)
  \circ
  \varphi^{-1}
  \qquad
  \text{in $\Out(N;K)$}
  .
\]
This shows that
$(\Theta^*)^\etabar
 =
 \upperleft
 {\varphi}
 {\Theta'}
$
as mod-$K$outer actions
of $Q$ on $N$;
in other words,
$\Theta^*$
and
$\upperleft
 {\varphi}
 {\Theta'}
$
are $\Aut(KNP)$-conjugate
$(\theta,\conj_P^Q)$-prolongations of $\Theta_P$.

In the situation of
the proposition,
if we have
an iterated extension
$(\,
 G'
 \,,\,
 j^*
 \,,\,
 \pi'
 \,)
$
whose $Q$-main extension
is
$(\,
 G'
 \,,\,
 i'
 \,,\,
 \pi'
 \,)
$
and
whose mod-$K$outer action $\Theta^*$
is given by $\Theta$,
then
$\Theta = \Theta^*$
and
$\upperleft
 {\varphi}
 {\Theta'}
$
are $\Aut(KNP)$-conjugate.
Conversely,
if we have
$\upperleft
 {\varphi}
 {\Theta'}
 =
 \Theta^\etabar
$
for some
$\eta \in \Aut(KNP)$,
then
the injective homomorphism
$j^* : N \rInto G'$,
defined by setting
$j^* := j' \circ \varphi^{-1} \circ \eta^{-1}$,
makes the diagram~\eqref{diag:j^vee}
commute,
whence
$(\,
 G'
 \,,\,
 j^*
 \,,\,
 \pi'
 \,)
$
is an iterated extension
of $(KNP)$ by $(PQR)$
having
$(\,
 G'
 \,,\,
 i'
 \,,\,
 \pi'
 \,)
$
as its $Q$-main extension,
and its mod-$K$outer action
$\Theta^*$
satisfies
$(\Theta^*)^\etabar
 =
 \upperleft
 {\varphi}
 {\Theta'}
 =
 \Theta^\etabar
$,
which is to say
$\Theta^* = \Theta$.
\end{proof}




\providecommand{\bysame}{\leavevmode\hbox to3em{\hrulefill}\thinspace}

\end{document}